\documentclass[10pt,reqno]{amsart}

\usepackage{amssymb}
\usepackage{bm}
\usepackage[centertags]{amsmath}
\usepackage{amsfonts}
\usepackage{amsthm}
\usepackage{mathrsfs}
\usepackage[usenames]{xcolor}

\usepackage{comment}

\theoremstyle{definition}
\newtheorem{thm}{Theorem}[section]
\newtheorem{dfn}[thm]{Definition}
\newtheorem{prp}[thm]{Proposition}
\newtheorem{lem}[thm]{Lemma}
\newtheorem{cor}[thm]{Corollary}
\newtheorem{rmk}[thm]{Remark}
\newtheorem{ntt}[thm]{Notation}
\newtheorem{exa}[thm]{Example}

\newtheorem{problem}{Problem}

\newtheorem{thmA}{Theorem}

\newcommand{\Ga}{\Gamma}
\newcommand{\ga}{\gamma}
\newcommand{\De}{\Delta}
\newcommand{\al}{\alpha}
\newcommand{\de}{\delta}
\newcommand{\la}{\lambda}

\newcommand{\e}{\varepsilon}
\newcommand{\N}{\mathbb{N}}
\newcommand{\R}{\mathbb{R}}

\newcommand{\trn}[1]{{\left\vert\kern-0.25ex\left\vert\kern-0.25ex\left\vert #1 
    \right\vert\kern-0.25ex\right\vert\kern-0.25ex\right\vert}}
\newcommand{\qun}[1]{{\left\vert\kern-0.25ex\left\vert\kern-0.25ex\left\vert\kern-0.25ex\left\vert #1 
    \right\vert\kern-0.25ex\right\vert\kern-0.25ex\right\vert\kern-0.25ex\right\vert}}

\newcommand{\vertiii}[1]{{\left\vert\kern-0.25ex\left\vert\kern-0.25ex\left\vert #1 
    \right\vert\kern-0.25ex\right\vert\kern-0.25ex\right\vert}}

\DeclareMathOperator{\supp}{supp}

\DeclareMathOperator{\ran}{range}

\DeclareMathOperator{\ra}{rank}

\DeclareMathOperator{\ag}{age}

\DeclareMathOperator{\we}{weight}

\long\def\symbolfootnote[#1]#2{\begingroup%
\def\thefootnote{\fnsymbol{footnote}}\footnote[#1]{#2}\endgroup}

\allowdisplaybreaks

\begin{document}
 
\title[Diagonal scalar-plus-compact as Calkin algebras]{Algebras of diagonal operators of the form scalar-plus-compact are Calkin algebras}

\author[P. Motakis]{Pavlos Motakis}
\address{Department of Mathematics, Texas A\&M University, College Station, TX 77843-3368, U.S.A.}
\email{pavlos@math.tamu.edu}

\author[D. Puglisi]{Daniele Puglisi}
\address{Department of Mathematics and Computer Sciences, University of Catania,  Catania, 95125, Italy (EU)}
\email{dpuglisi@dmi.unict.it}

\author[A. Tolias]{Andreas Tolias}
\address{Department of Mathematics, University of Ioannina, Ioannina, 45110, Greece}
\email{atolias@uoi.gr}

\thanks{The first author's research was supported by NSF DMS-1600600.}

\begin{abstract}
For every Banach space $X$ with a Schauder basis consider the Banach algebra $\R I\oplus\mathcal{K}_\mathrm{diag}(X)$ of all diagonal operators that are of the form $\la I + K$. We prove that $\R I\oplus\mathcal{K}_\mathrm{diag}(X)$ is a Calkin algbra i.e., there exists a Banach space $\mathcal{Y}_X$ so that the Calkin algebra of $\mathcal{Y}_X$ is isomorphic as a Banach algebra to $\R I\oplus\mathcal{K}_\mathrm{diag}(X)$. Among other applications of this theorem we obtain that certain hereditarily indecomposable spaces  and the James spaces $J_p$ and their duals endowed with natural multiplications are Calkin algebras, that all non-reflexive Banach spaces with unconditional bases are isomorphic as Banach spaces to Calkin algebras, and that sums of reflexive spaces with unconditional bases with certain James-Tsirelson type spaces are isomorphic as Banach spaces to Calkin algebras.
\end{abstract}
 
\maketitle

\section*{Introduction}
 This paper is aiming to contribute to the ongoing effort of understanding the types of unital Banach algebras $A$ that may occur as Calkin algebras, i.e., those for which there exists a Banach space $X$ so that $A$ is isomorphic as a Banach algebra to the Calkin algebra of $X$. This is the quotient algebra $\mathcal{C}al(X) = \mathcal{L}(X)/\mathcal{K}(X)$ where $\mathcal{L}(X)$ denotes the unital Banach algebra of all bounded linear operators on $X$ and $\mathcal{K}(X)$ denotes the ideal of all compact ones. This unital Banach algebra was introduced by J. W. Calkin in \cite{Cal} in the 1940's, at first only for $X$ being the Hilbert space. The topic of Calkin algebras of general Banach spaces later gathered attention as well. A classical result from the 1950's, due to F. B. Atkinson (see \cite{At}), is that a bounded linear operator on a Banach space is Fredholm precisely when its equivalence class in the Calkin algebra is invertible. Another example from the same era is the observation of B. Yood in \cite{Y} that, unlike the algebra $\mathcal{L}(X)$, in certain cases the Calkin algebra of a Banach space  may fail to be semi-simple. This is true in particular for the space $L_1$.

The origins of the topic studied herein can be traced to only much later, namely to 2011, when the first example of a Banach space $\mathfrak{X}_\mathrm{AH}$ with the scalar-plus-compact property was presented by S. A. Argyros and R. G. Haydon in \cite{AH}. On this space every bounded linear operator is a scalar multiple of the identity plus a compact operator which means that the Calkin algebra of $\mathfrak{X}_\mathrm{AH}$ is one-dimensional. This was in fact the first time a Calkin algebra of a Banach space could be explicitly described. Of particular importance in the construction is that $\mathfrak{X}_\mathrm{AH}$ is a hereditarily indecomposable (HI) $\mathscr{L}_\infty$-Bourgain-Delbaen space. In the years that followed a number of examples of Calkin algebras have appeared that can be explicitly described in terms of classical Banach algebras. All these examples are in one way or another tightly knit together with the theory of HI spaces, the theory of $\mathscr{L}_\infty$-spaces, or both. In his PhD thesis M. Tarbard \cite{T} combined the technique of W. T. Gowers and B. Maurey from \cite{GM} with the technique from \cite{AH} to construct for every $n\in\N$ a Banach space the Calkin algebra of which coincides with all $n\times n$ upper triangular Toeplitz matrices and also a Banach space the Calkin algebra of which coincides isometrically with the convolution algebra of $\ell_1(\N_0)$. Tarbard posed the question of what unital Banach algebras can be realized as Calkin algebras. Alternatively, he proposed that one should seek for obstructions that would prevent a unital Banach algebra from being a Calkin algebra. This paper is focused on studying the first question and as of yet no results concerning the second one exist. A contribution to the first one was made by  T. Kania and N. J. Laustsen in \cite{KL} where they observed that carefully manipulating and taking finite direct sums of powers of appropriate versions of the space $\mathfrak{X}_\mathrm{AH}$ can lead to something surprising: all finite dimensional semi-simple complex algebras are Calkin algebras. In particular, for any natural numbers $m_1,\ldots,m_n$ the algebra $\mathbb{M}_{m_1}(\mathbb{K})\oplus\cdots\oplus\mathbb{M}_{m_n}(\mathbb{K})$ endowed with point-wise multiplication is a Calkin algebra. Here, $\mathbb{K}$ denotes the scalar field and $\mathbb{M}_{k}(\mathbb{K})$ denotes the algebra of all $k\times k$ matrices over $\mathbb{K}$. A remark worth making is that Tarbard's aforementioned finite dimensional examples of Calkin algebras are not semi-simple. In the infinite dimensional setting the first two authors and D. Zisimopoulou proved in \cite{MPZ} that for every countable compactum $K$ the space $C(K)$ is a Calkin algebra. A noteworthy reason for which one may be interested in these examples is that many of them provide insight into ideals of $\mathcal{L}(X)$. Indeed, all aforementioned algebras are Calkin algebras of spaces with the bounded approximation property and information about the ideal structure of the Calkin algebra can be lifted to study the ideals of the corresponding $\mathcal{L}(X)$ space. For a detailed exposition of this topic we refer the interested reader to the introduction of \cite{KL}.

The motivation for the present paper stems from \cite[Question 1, page 66]{MPZ} of the existence of a Banach space with an infinite dimensional and reflexive Calkin algebra. This is indeed interesting as all infinite dimensional aforementioned examples are either isomorphic to $\ell_1$ or $c_0$-saturated and thus on the far opposite side of being reflexive. This question is difficult to answer and a space with a reflexive Calkin algebra cannot have too many complemented subspaces. Instead, we were interested in investigating whether we could find a quasi-reflexive Calkin algebra. Recall that a Banach space is called quasi-reflexive (of order one) if its canonical image in its second dual is of co-dimension one. While affirmatively answering this question  we were able to identify a rather large variety of explicitly described spaces that can be realized as Calkin algebras that are, from a Banach spaces perspective, truly different to the previously understood examples. Many of them admit unconditional bases while one example is HI. The main result of the present paper has the following statement.

\begin{thmA}
\label{main theorem intro}
Let $X$ be a Banach space with a Schauder basis. Then there exists a Banach space $\mathcal{Y}_X$ so that the Calkin algebra of $\mathcal{Y}_X$ is isomorphic as a Banach algebra to $\R I\oplus\mathcal{K}_\mathrm{diag}(X)$.
\end{thmA}
In fact, for every $\e>0$ the space $\mathcal{Y}_X$ can be constructed so that the corresponding Banach algebra isomorphism $\Phi:\mathcal{Y}_X\to\R I\oplus\mathcal{K}_\mathrm{diag}(X)$ satisfies $\|\Phi\|\|\Phi^{-1}\| \leq 1+\e$.

Using a Theorem of S. A. Argyros, I. Deliyanni, and the third author, that in special cases explicitly describes the diagonal operators of a Banach space with a basis in terms of its dual  (see \cite{ADT}), we describe several examples of spaces $\R I\oplus \mathcal{K}_\mathrm{diag}(X)$ which are, by virtue of Theorem \ref{main theorem intro}, examples of Calkin algebras as well. The unital Banach algebra $\R I\oplus \mathcal{K}_\mathrm{diag}(X)$ is always commutative and semi-simple. The first example is a Hereditarily indecomposable Banach algebra $\mathfrak{X}_\mathrm{ADT}$  from \cite{ADT} that is additionally quasi-reflexive of order one.

\begin{thmA}
There exists a hereditarily indecomposable Calkin algebra that is quasi-reflexive of order one.
\end{thmA}
The possibility of such extreme behavior of the quotient $\mathcal{L}(X)/\mathcal{K}(X)$ contrasts the more canonical one of $\mathcal{L}(X)$. The latter space is always decomposable, containing complemented copies of both $X$ and $X^*$.

James' classical space $J_p$ from \cite{J} is, for each $1<p<\infty$, the Banach space consisting of all scalar sequences $(a_i)_i$ for which the quantity
\[\left\|(a_i)_i\right\|^p = \sup_{(E_k)_k}\sum_k\Big|\sum_{i\in E_k}a_i\Big|^p,\]
where the supremum is taken over all disjoint collections of intervals of $\N$, is finite. The spaces $J_p$ are quasi-reflexive of order one. It was for observed by A. D. Andrew and W. L. Green in \cite{AG} that $J_p$, after appropriate renorming, becomes a non-unital Banach algebra when endowed with coordinate-wise multiplication with respect to the basis $e_1$, $e_2-e_1$, $e_3-e_2,\ldots$. We denote the unitization of James space by $\R e_\omega\oplus J_p$, for $1<p<\infty$.

\begin{thmA}
\label{classical james banach algebra}
The spaces $\R e_\omega\oplus J_p$, $1<p<\infty$ are Calkin algebras.
\end{thmA}
S. F. Bellenot, R. G. Haydon, and E. Odell in \cite{BHO} extended the definition of James, based on the unit vector basis of $\ell_p$, to an arbitrary space $X$ with an unconditional basis to define the space $J(X)$, the ``jamesification'' of $X$. The space $J(X)$ is quasi-reflexive of order one whenever $X$ is reflexive. As it so happens this space is a non-unital Banach algebra as well the unitization of which we denote by $\R e_\omega\oplus J(X)$. Furthermore, a special subspace of $J^*(X)$ which we denote by $\mathcal{J}_*(X)$, and coincides with $J(X)^*$ when $X$ does not contain $\ell_1$,  is a separable unital Banach algebra.

\begin{thmA}
\label{all james banach algebra}
For any space $X$ with a normalized unconditional basis the spaces $\R e_\omega\oplus J(X)$ and $\mathcal{J}_*(X)$ are Calkin algebras.
\end{thmA}

Next we turn our attention to spaces with unconditional bases endowed with coordinate-wise multiplication. We study these spaces themselves as Banach algebras. We do not prove that they are Calkin algebras but they always embed in them as complemented ideals. In what follows ``$\oplus$'' denotes the direct sum of Banach spaces.

\begin{thmA}
\label{ideals with unconditiional bases}
Let $\mathcal{A}$ be Banach space with a normalized unconditional basis endowed with coordinate-wise multiplication and $\mathcal{B}$ be a Banach algebra. In the cases described in the list bellow there exists a Banach space $\mathcal{Y}$ so that the Calkin algebra $\mathcal{C}al(\mathcal{Y})$ contains an ideal $\tilde{\mathcal{A}}$ isomorphic as a Banach algebra to $\mathcal{A}$ and a subalgebra $\tilde{\mathcal{B}}$ isomorphic as a Banach algebra to $\mathcal{B}$ so that $\mathcal{C}al(\mathcal{Y}) = \tilde{\mathcal{A}}\oplus \tilde{\mathcal{B}}$.
\begin{itemize}

\item[(a)]  The space $\mathcal{B}$ is $C(\omega)$, the space of convergent scalar sequences with pointwise multiplication.

\item[(b)] The space $\mathcal{B}$ is $\mathrm{bv}_1$, the space of all scalar sequences of bounded variation with pointwise multiplication.

\item[(c)] The space $\mathcal{A}$ does not contain an isomorphic copy of $c_0$ and $\mathcal{B}$ is the space $J(T_{\mathcal{M}}^{1/2})^*$ where $T_{\mathcal{M}}^{1/2}$ is the Tsirelson space over an appropriate regular family $\mathcal{M}$.

\item[(d)] The space $\mathcal{A}$ does not contain an isomorphic copy of $\ell_1$ and $\mathcal{B}$ is the space $\mathbb{R}e_\omega\oplus J(T_{\mathcal{M}}^{1/2})$ where $T_{\mathcal{M}}^{1/2}$ is the Tsirelson space over an appropriate regular family $\mathcal{M}$.

\end{itemize}
\end{thmA}
In statement (c) the complexity, i.e. the Cantor-Bendixson index, of $\mathcal{M}$ depends on the Szlenk index of the natural predual of $\mathcal{A}$ and in statement (d) it depends on the Szlenk index of $\mathcal{A}$. Of course, the spaces $C(\omega)$ and $\mathrm{bv}_1$ are isomorphic as Banach spaces to $c_0$ and to $\ell_1$ respectively. By a well known theorem of James any non-reflexive Banach space $X$ with an unconditional basis is either isomorphic to $X\oplus c_0$ or to $X\oplus\ell_1$. Consequently statements (a) and (b) yield something interesting.

\begin{thmA}
\label{all non-reflexive unconditional}
Every non-reflexive Banach space $X$ with a normalized unconditional basis is isomorphic as a Banach space to a Calkin algebra (that contains a complemented ideal isomorphic as a Banach algebra to $X$ endowed with coordinate-wise multiplication).
\end{thmA}

For reflexive Banach spaces we do not obtain the same result however since the space $J(T_{\mathcal{M}}^{1/2})$ is quasi-reflexive of order one from statement (c) or (d) we may deduce the following. In the statement that follows multiplication is also coordinate-wise with respect to the given unconditional basis.

\begin{thmA}
\label{all reflexive ideals in quasireflexive}
Every reflexive Banach space with a normalized unconditional basis is isomorphic as a Banach algebra to a complemented ideal of a separable quasi-reflexive Calkin algebra.
\end{thmA}

It is worth mentioning that in certain cases, e.g. when the space $X$ is super-reflexive, in statements (c) and (d) of Theorem \ref{ideals with unconditiional bases} the space $J(T_\mathcal{M}^{1/2})$ may be replaced with the space $J_p$ for appropriate $1<p<\infty$. Going back to our initial question of the existence of quasi-reflexive Calkin algebras we also observe that these exist for any finite order. That is, for any $n\in\N$ there exists a Calkin algebra that is quasi-reflexive of order $n$. There are a few more examples that are mentioned throughout Section \ref{long section diagonal algebras}.

The technically most challenging part of this paper is the proof of Theorem \ref{main theorem intro}, that is, given a Banach space $X$ with a Schauder basis the construction of a space $\mathcal{Y}_X$ the Calkin algebra of which is isomorphic as a Banach algebra to $\R I\oplus\mathcal{K}_\mathrm{diag}(X)$. The definition of the space $\mathcal{Y}_X$ is based on a method of Zisimopoulou from \cite{Z} for defining direct sums $\mathcal{Z} = (\sum_k\oplus X_k)_\mathrm{AH}$ of a sequence of Banach spaces $(X_k)_k$ where the outside norm is in turn based on the Argyros-Haydon space from \cite{AH}. The main feature of the construction from \cite{Z} is that, under certain assumptions, every bounded linear operator $T:\mathcal{Z}\to\mathcal{Z}$ is a scalar operator plus an operator that vanishes on block sequences, namely a horizontally compact operator. A variation of this method was iterated transfinitely in \cite{MPZ} to show that for a countable compactum $K$ the space $C(K)$ is a Calkin algebra. In this paper we define an $X$-Bourgain-Delbaen-Argyros-Haydon direct sum $\mathcal{Y}_X = (\sum_k\oplus X_k)_\mathrm{AH}^X$ of a sequence of Argyros-Haydon Banach spaces $(X_k)_k$. This direct sum is designed so that the space $X$ is  crudely finitely representalbe in an appropriate block way. This is performed in such a manner so that the diagonal operators on $X$ can be viewed as compact perturbations of diagonal operators with respect to the decomposition  $(X_k)_k$ of $\mathcal{Y}_X$. The result is that the space $\R I\oplus\mathcal{K}_\mathrm{diag}(X)$ embeds into $\mathcal{C}al(\mathcal{Y}_X)$. The proof that this embedding is onto goes through all the delicate intricacies of the Argyros-Haydon construction, suitably modified to the context of the present setting. The main difference is that so called rapidly increasing block sequences have to be defined so as to take into consideration the ``$X$-part'' of the direct sum $(\sum_k\oplus X_k)_\mathrm{AH}^X$. In the end, the main theorem about the operators on the space is that they are of the form
\begin{equation}
\label{equation intro}
\tag{$\mathfrak{a}$}T = a_0I+\lim_n\Big(\sum_{k=1}^na_kI_k + K_n\Big)
\end{equation}
where the $I_k$ are projections onto the spaces $X_k$ and $K_n$ are compact operators. The definition of the space $\mathcal{Y}_X$ is presented comprehensively and most proofs are explained thoroughly. We have chosen to leave out a small number of details that are nearly exactly identical to proofs from other papers, for which we provide references.

The paper can be viewed as being divided into two main parts. The first part consists of Section \ref{long section diagonal algebras}. In this section the basics around the space $\R I\oplus \mathcal{K}_\mathrm{diag}(X)$ are discussed. Several examples of these spaces are presented, namely those mentioned in the introduction and a few others. As a result of Theorem \ref{main theorem intro} they are all Calkin algebras. Several results concerning James spaces and Tsirelson spaces are proved and utilized. To the most part the tools used are relatively elementary and a reading of this section should not be too challenging to the reader who is familiar with classical Banach space theory. With the exception of concluding Section \ref{section remarks and problems} the remaining Sections \ref{easy space}-\ref{Diagonal plus compact approximations} are focused on defining the space $\mathcal{Y}_X$ and proving the necessary properties to achieve the desired conclusion about its Calkin algebra. Section \ref{easy space} concerns the definition of a ``first stage'' $\mathcal{X}$ of the final direct sum $\mathcal{Y}_X = (\sum_k\oplus X_k)_\mathrm{AH}^X$ that does not involve the Bourgain-Delbaen-Argyros-Haydon part. Although $X$ is finitely block represented in $\mathcal{X}$ every normalized block sequence has a subsequence that is equivalent to the unit vector basis of $c_0$. We state most properties of $\mathcal{X}$ without proof because we don't evoke them directly. Section \ref{definition space section} is devoted to precisely defining the space $\mathcal{Y}_X = (\sum_k\oplus X_k)_\mathrm{AH}^X$ and determining its most fundamental properties. In Section \ref{section calkin algebra proof} we prove that the Calkin algebra of $\mathcal{Y}_X$ is $\R I\oplus\mathcal{K}_\mathrm{diag}(X)$. This is done by assuming that \eqref{equation intro} holds, the proof of which is the  objective of Sections \ref{section ris}-\ref{Diagonal plus compact approximations}. In Section  \ref{section ris} rapidly increasing sequences (RIS) are defined. Here, RIS have the additional property that the $X$ part in the $X$-Bourgain-Delbaen-Argyros-Haydon sum completely vanishes on them allowing them to be treated in the same manner as in \cite{AH} and \cite{Z}. In the same section the basic inequality is proved and it is shown that operators that vanish on RIS are horizontally compact. In Section \ref{section s-p-hc} it is proved that bounded operators on the space are scalars multiples of the idenity plus a horizontally compact operator and in Section \ref{Diagonal plus compact approximations} \eqref{equation intro} is finally proved. Section \ref{section remarks and problems} contains several remarks and open problems.

\section{The spaces $\mathcal{L}_\mathrm{diag}(X)$ and $\R I\oplus\mathcal{K}_\mathrm{diag}(X)$}
\label{long section diagonal algebras}
A sequence $(e_i)_i$ in a Banach space $X$ is called a Schauder basis of $X$ if every element can be represented uniquely as $x = \sum_{i=1}^\infty a_ie_i$ where the convergence is in the norm topology. Then, the natural projections $P_n(\sum_{i=1}^\infty a_ie_i) = \sum_{i=1}^na_ie_i$, $n\in\N$, are uniformly bounded by some $C\geq 1$ which is called the monotone constant of $(e_i)_i$. The sequence $(e_i^*)_i$ in $X^*$ defined by $e_i^*(e_j) = \de_{i,j}$ is called the biorthogonal sequence of $(e_i)_i$. For each $x\in X$ we define the support of $x$ to be the subset of $\N$, $\supp(x) = \{i: e_i^*(x)\neq 0\}$ and the range of $x$ to be the smallest interval of $\N$ containing the support of $x$. A sequence $(x_k)_k$ in $X$ is called a block sequence if the supports of the corresponding vectors are successive subsets of $\N$.

Given a Banach space $X$ with a Schauder basis $(e_i)_i$ a bounded linear operator $T:X\to X$ is called diagonal if for every $i\neq j\in\N$ we have $e_j^*(Te_i) = 0$. We denote the subspace of $\mathcal{L}(X)$ consisting of all diagonal operators by $\mathcal{L}_\mathrm{diag}(X)$. The simplest diagonal operator is the identity $I$. The ``building blocks of diagonal operators'' may be considered the operators $(e_i\otimes e_i)_i$ defined by $e_i^*\otimes e_i(x) = e_i^*(x)e_i$. Every diagonal operator $T$ then can be written as
\begin{equation}
\label{sot representation of diagonal operators}
T = \mathrm{SOT}-\sum_{i=1}^\infty e_i^*(Te_i)e_i^*\otimes e_i.
\end{equation}
We denote the space of diagonal operators that are simultaneously compact by $\mathcal{K}_\mathrm{diag}(X)$. That is, $\mathcal{K}_\mathrm{diag}(X) = \mathcal{L}_\mathrm{diag}(X)\cap\mathcal{K}(X)$. It is straightforward to check that a diagonal operator $T$ is compact if and only if the convergence of the sum in \eqref{sot representation of diagonal operators} is in the norm topology. It follows that $(e_i^*\otimes e_i)_i$ is a Schauder basis for $\mathcal{K}_\mathrm{diag}(X)$. It is worth pointing out that $\mathcal{L}_\mathrm{diag}(X)$ is isomorphic to the dual of the space $V$ spanned by the biorthogonal  sequence of $(e_i^*\otimes e_i)_i$ in $(\mathcal{K}_\mathrm{diag}(X))^*$.

We are particularly interested in the subspace of $\mathcal{L}_\mathrm{diag}(X)$ consisting of all operators of the form $T = \la I + K$, with $K\in\mathcal{K}_\mathrm{diag}(X)$. We naturally denote this space by $\R I\oplus \mathcal{K}_\mathrm{diag}(X)$. This space endowed with operator composition is a commutative unital Banach algebra. An operator $T$ in $\mathcal{L}_\mathrm{diag}(X)$ is in $\R I\oplus \mathcal{K}_\mathrm{diag}(X)$ if and only if
\begin{equation}
\label{scalar plus compact diagonal characterization}
\lim_m\lim_n\left\|\sum_{i=m}^n\left(e_i^*(Te_i) - e_m^*(Te_m)\right)e_i^*\otimes e_i\right\| = 0.
\end{equation}

It is a well known and easy to prove fact that if $(e_i)_i$ is unconditional then $\mathcal{L}_\mathrm{diag}(X)$ is naturally isomorphic to $\ell_\infty$. Then, $\mathcal{K}_\mathrm{diag}(X)$ is naturally isomorphic to $c_0$ (endowed with the unit vector basis) and $\R I\oplus \mathcal{K}_\mathrm{diag}(X)$ is naturally isomorphic to $c$, the space of convergent real sequences. As we will see in several examples if the basis is not unconditional then more interesting things may occur.

\subsection{Ideals of $\R I\oplus\mathcal{K}_\mathrm{diag}(X)$}
We observe that the ideals in $\R I\oplus\mathcal{K}_\mathrm{diag}(X)$ are the same as in the space $c = C(\omega)$. Here, $\omega$ is the first infinite ordinal number.
\begin{prp}
\label{these are the ideals}
Let $X$ be a Banach space with a Schauder basis. For every $T\in\R I\oplus\mathcal{K}_\mathrm{diag}(X)$ and $i\in\N$ define $\la_{T,i} = e_i^*(Te_i)$ and also define $\la_{T,\omega} = \lim_i\la_{T,i}$ (which exists by \eqref{scalar plus compact diagonal characterization}). For every closed subset $L$ of $[1,\omega]$ define
\[\mathcal{A}_L = \left\{T\in\R I\oplus\mathcal{K}_\mathrm{diag}(X): \la_{T,\kappa} = 0 \text{ for all } \kappa\in L\right\}.\]
Then the closed ideals of $\R I\oplus\mathcal{K}_\mathrm{diag}(X)$ are precisely the spaces $\mathcal{A}_L$ for all closed subsets $L$ of $[1,\omega]$.
\end{prp}

\begin{proof}
We start by making some elementary observations that follow from the fact that $(e_i^*\otimes e_i)_i$ is a Schauder basis of $\mathcal{K}_\mathrm{diag}(X)$. If $\omega\in L$ then $\mathcal{A}_L$ is the closed linear span of $(e_i^*\otimes e_i^*)_{i\in\N\setminus L}$. If $\omega\notin L$ then $L$ has an upper bound and $\mathcal{A}_L$ is the closed linear span of $P_{\N\setminus L} = I - \sum_{i\in L}e_i^*\otimes e_i$ together with $(e_i^*\otimes e_i^*)_{i\in\N\setminus L}$.

Given a closed ideal $\mathcal{A}$ define $L = \{\kappa\in[1,\omega]: \la_{T,\kappa} = 0$ for all $T\in\mathcal{A}\}$. Clearly, $L$ is closed and $\mathcal{A}\subset\mathcal{A}_L$. By the fact that $\mathcal{A}$ is an ideal is easy to see that $e_i^*\otimes e_i$ is in $\mathcal{A}$ for all $i\in\N\setminus L$. In the case $\omega\in L$ this easily yields $\mathcal{A}_L = [(e_i^*\otimes e_i)_{i\in\N\setminus L}] \subset \mathcal{A}$, i.e., $\mathcal{A} = \mathcal{A}_L$. For the case $\omega\notin L$ it is sufficient to show that $P_{\N\setminus L} = I - \sum_{i\in L}e_i^*\otimes e_i$ is in $\mathcal{A}$. To that end, let $T$ be any element in $\mathcal{A}$ with $\la_{T,\omega} \neq 0$. Then, by \eqref{scalar plus compact diagonal characterization},
\[
\begin{split}
T &= \la_{T,\omega}I + \sum_{i=1}^\infty(\la_{T,i} - \la_{T,\omega})e_i^*\otimes e_i\\
&= \la_{T,\omega}I - \sum_{i\in L}\la_{T,\omega}e_i^*\otimes e_i + \sum_{i\in\N\setminus L}^\infty(\la_{T,i} - \la_{T,\omega})e_i^*\otimes e_i\\
& = \la_{T,\omega}P_{\N\setminus L} + S,
\end{split}
\]
where $S = \sum_{i\in\N\setminus L}^\infty(\la_{T,i} - \la_{T,\omega})e_i^*\otimes e_i$. As $e_i^*\otimes e_i$ is in $\mathcal{A}$ for all $i\in\N\setminus L$ it follows that $S\in\mathcal{A}$ which yields the desired conclusion.
\end{proof}

\subsection{Initial examples of spaces $\R I\oplus \mathcal{K}_\mathrm{diag}(X)$}
An important result for explicitly describing the space $\mathcal{L}_\mathrm{diag}(X)$ is the following.

\begin{thm}\cite[Theorem 1.1]{ADT}
\label{ADT main thm}
Let $X$ be a Banach space with a Schauder basis $(e_i)_i$. The following are equivalent.
\begin{itemize}
\item[(i)] The map $e_i^*\to e_i^*\otimes e_i$ extends to an isomorphism between $X^*$ and $\mathcal{L}_\mathrm{diag}(X)$.
\item[(ii)]
\begin{itemize}
\item[(a)] The basis $(e_i)_i$ dominates the summing basis of $c_0$.
\item[(b)] The space $X^*$ is submultiplicative, that is, there exists $C$ so that for all sequences of scalars $(a_i)_{i=1}^n$ and $(b_i)_{i=1}^n$ we have
\[\left\|\sum_{i=1}^na_ib_ie_i^*\right\|\leq C\left\|\sum_{i=1}^na_ie_i^*\right\|\left\|\sum_{i=1}^nb_ie_i^*\right\|\]
\end{itemize}
\end{itemize}
\end{thm}

The following, fairly immediate, Corollary is sometimes more convenient.
\begin{cor}
\label{same space banach algebra}
Let $X$ be a Banach space with a Schauder basis for which there exist positive constants $C_1$, $C_2$ so that
\begin{itemize}

\item[(i)] for all $n\in\N$ $\|\sum_{i=1}^ne_i\| \leq C_1$ and
\item[(ii)] for all sequences of scalars $(a_i)_{i=1}^n$, $(b_i)_{i=1}^n$
\[\left\|\sum_{i=1}^na_ib_ie_i\right\| \leq C_2\left\|\sum_{i=1}^na_ie_i\right\|\left\|\sum_{i=1}^nb_ie_i\right\|.\]
\end{itemize}
Then if $Y = [(e_i^*)_i]$ the space $\mathbb{R}I\oplus\mathcal{K}_\mathrm{diag}(Y)$ is isomorphic as a Banach algebra to the unitization $\R e_\omega\oplus X$ of $X$ via the map $I\mapsto e_\omega$ and for all $i\in\N$ $e_i^{**}\otimes e_i^*\to e_i$.
\end{cor}

Clearly, the sequence $(e_\omega,e_1,e_2,\ldots)$ forms a Schauder basis of $\R e_\omega\oplus X$. The following result can be easily obtained by combining Proposition \ref{these are the ideals} and Corollary \ref{same space banach algebra}.

\begin{cor}
\label{ideals in that case too}
Let $X$ be a Banach space with a Schauder basis that satisfies the assumptions of Corollary \ref{same space banach algebra}. The ideals of $\R e_\omega\oplus X$ are precisely the spaces $\mathcal{A}_L= \cap_{\kappa\in L}\ker e_\kappa^*$ for all closed subsets $L$ of $[1,\omega]$.
\end{cor}

We mention now some examples of spaces to which this Theorem can be applied without the requirement of a lot of other theory. Some of these examples are from \cite{ADT}. The subsequent subsections of this section are devoted to providing more examples to which this theorem can be applied.

\begin{exa}
\label{variation norm}
As it is explained in \cite[Example 2.7]{ADT}, if $X = c$ endowed with the monotone summing basis $(s_n)_n$ then $\mathcal{L}_\mathrm{diag}(X)$ is isometric to $\ell_1(\N)$. In this case $(s_n^*)_n$ spans a space of co-dimension one in $X^*$ which, by Theorem \ref{ADT main thm}, implies that $\R I\oplus\mathcal{K}_\mathrm{diag}(X) = \mathcal{L}_\mathrm{diag}(X) \equiv \ell_1$. To be more precise, $\mathcal{L}_\mathrm{diag}(X)$ isometrically coincides with the Banach algebra $\mathrm{bv}_1$ of all scalar sequences of bounded variation equipped with coordinate wise multiplication and the norm $\|(a_i)_i\| = \sum_i|a_i-a_{i+1}| + \lim_i|a_i|$.
\end{exa}

\begin{exa}
One of the main results of \cite{ADT} is Theorem 1.4 which states that there exists a  hereditarily indecomposable Banach space $X_\mathrm{ADT}$ that is quasi-reflexive of order one, it has a hereditarily indecomposable dual, and it has a Schauder basis $(e_i)_i$ that satisfies the assumptions of Theorem \ref{ADT main thm}. This means that $\mathcal{L}_\mathrm{diag}(X_\mathrm{ADT}) \equiv X^*_\mathrm{ADT}$, which is hereditarily indecomposable (in fact, as it is stated in \cite{ADT}, this identification is isometric). Also, by \cite[Theorem 14 (iii)]{ADT} we directly conclude  $\R I\oplus \mathcal{K}_\mathrm{diag}(X_\mathrm{ADT}) = \mathcal{L}_\mathrm{diag}(X_\mathrm{ADT})\equiv X^*_\mathrm{ADT}$ (this follows from the quasi-reflexivity of $X_\mathrm{ADT}$).
\end{exa}

Recall that every Banach space $X$ with a 1-unconditional basis is a Banach algebra when endowed with coordinate-wise multiplication.
\begin{prp}
\label{all unconditional containing c0}
Let $X$ be a Banach space with a normalized 1-unconditional basis $(x_i)_i$. Then there exists a Banach space $Y$ with a Schauder basis so that the space $\mathbb{R}I\oplus\mathcal{K}_\mathrm{diag}(Y)$ contains an ideal $\widetilde X$ that is isomorphic as a Banach algebra to $X$ and a subalgebra $\mathcal{A}$ that is isomorphic as a Banach algebra to $C(\omega)$ so that $\mathbb{R}I\oplus\mathcal{K}_\mathrm{diag}(Y) = \widetilde X\oplus \mathcal{A}$.
\end{prp}

\begin{proof}
Let $(s_i)_i$ denote the monotone basis of $c$, i.e. $\|\sum _ia_is_i\| = \sup_n|\sum_{i=1}^na_i|$.We define a norm on $c_{00}(\N)$ so that for any scalar sequence $(a_i)_i$ that is eventually zero we have
\begin{equation}
\label{sum with summing}
\left\|\sum_{i=1}^\infty a_ie_i\right\| = \max\left\{\left\|\sum_{i=1}^\infty a_{2i}x_{i}\right\|,\left\|\sum_{i=1}^\infty a_is_i\right\|\right\}
\end{equation}
and let $W$ denote its completion. The sequence $(d_i)_i$ defined by $d_1 = e_1$ and for all $i\in\N$ $d_{i+1} = e_{i+1} - e_i$ forms a Schauder basis for $W$.
The norm on the sequence $(d_i)_i$ is given by
\[\left\|\sum_{i=1}^\infty a_id_i\right\| = \max\left\{\left\|\sum_{i=1}^\infty \left(a_{2i}-a_{2i+1}\right)x_{i}\right\|,\sup_i|a_1-a_{i}|\right\}\]
 A standard argument using the triangle inequality yields that $W$ endowed with $(d_i)_i$ satisfies the assumptions of Corollary \ref{same space banach algebra}.  The above formula also immediately yields that the sequence $(d_{2i})_i$ is equivalent to $(x_i)_i$ and by Corollary \ref{ideals in that case too} the space $\widetilde X = [(d_{2i})_i]$ is an ideal of $W$. It is also straightforward to check that the map $Q:W\to W$ defined by $Q(\sum_ia_id_i) = \sum_{i}(a_{2i}-a_{2i+1})d_{2i}$ defines a bounded linear projection onto $\widetilde X$. Note that the sequence $(d_1)^\frown(d_{2i}+d_{2i+1})_i$ is equivalent to the unit vector basis of $c_0$, the space $\bar{\mathcal{A}}= \R d_1\oplus[(d_{2i}+d_{2i+1})_i]$ is closed under multiplication, and it is isomorphic as a Banach algebra to $c_0$ endowed with coordinate-wise multiplication with respect to its unit vector basis. It is also easy to see that $\bar{\mathcal{A}}$ is complementary to $\widetilde X$. Finally, take $\widetilde X$ and $\mathcal{A} = \R e_\omega\oplus \bar{\mathcal{A}}$ as subspaces of $\R e_\omega\oplus W$ which, by Corollary \ref{same space banach algebra}, coincides with $\mathbb{R}I\oplus\mathcal{K}_\mathrm{diag}(Y)$ for $Y  = [(d_i^*)_i]$.
\end{proof}

The following lemma allows to dualize certain spaces and obtain Banach algerbas.

\begin{lem}
\label{dualizing unconditional plus submultiplicative}
Let $X$ and $Y$ be Banach spaces with normalized Schauder bases $(x_i)_i$ and $(y_i)$ respectively and assume that  $(x_i)_i$ is unconditional and that $(y_i^*)_i$ is submultiplicative. If $M = \{m_1<m_2<\cdots\}$, $N = \{n_1<n_2<\cdots\}$ are subsets of $\N$ with $M\cup N = \N$ define a norm on $c_{00}(\N)$ so that for any sequence of scalars $(a_i)_i$ that is eventually zero 
\begin{equation*}
\left\|\sum_{i=1}^\infty a_ie_i\right\| = \max\left\{\left\|\sum_{i=1}^\infty a_{m_i}x_i\right\|,\left\|\sum_{i=1}^
\infty a_{n_i}y_i\right\|\right\}
\end{equation*}
and let $W$ denote the completion of $c_{00}(\N)$ with this norm. Then the sequence $(e_i^*)_i$ in $W^*$ is submultiplicative.
\end{lem}

\begin{proof}
By renorming our spaces we may assume that $(y_i^*)_i$ is bimonotone and $1$-submultiplicative and that $(x_i)_i$ is 1-unconditional. We will show that $(e_i^*)_i$ is 1-submultiplicative. Define the isometric embedding $T:W\to (X\oplus Y)_\infty$ given by
\[Te_i = (x,y)\text{ where }x = \left\{\begin{array}{ll}x_j&:\text{ if }i=m_j\\0&:\text{ otherwise}\end{array}\right.\text{ and }y = \left\{\begin{array}{ll}y_j&:\text{ if }i=n_j\\0&:\text{ otherwise}\end{array}\right.\]
and observe that $T^*:(X^*\oplus Y^*)_1\to W^*$ is the $w^*$-continuous map given by $T^*(x_i^*,0) = e^*_{m_i}$ and $T^*(0,y_i^*) = e^*_{n_i}$ for all $n\in\N$.

Let $u_1^*$ and $_2^*$ be normalized elements in $W^*$. We will show that the element $u_1^*u_2^* = w^*$-$\sum_iw_1^*(e_1)w_2^*(e_i)e_i^*$ is well defined and it has norm at most one. By the Hahn-Banach theorem there exist elements $x_1^* = w^*$-$\sum_i\la_ix_i^*$ and $y_1^* = w^*$-$\sum_i\mu_iy_i^*$ with $\|x_1^*\|+\|y_1^*\| = 1$ so that $u_1^* = T^*(x_1^*,y_1^*) = w^*$-$\sum_i\la_ie_{m_i}^* + w^*$-$\sum_i\mu_ie_{n_i}^*$ and similarly there exist $x_2^* = w^*$-$\sum_i\xi_ix_i^*$ and $y_2^* = w^*$-$\sum_i\zeta_iy_i^*$ with $\|x_2^*\|+\|y_2^*\| = 1$ so that $u_2^* = T^*(x_2^*,y_2^*) = w^*$-$\sum_i\xi_ie_{m_i}^* + w^*$-$\sum_i\zeta_ie_{n_i}^*$.

Define for each $i\in\N$ the scalars $\tilde\zeta_i =\left\{\begin{array}{ll}\zeta_j&:\text{ if }m_i=n_j\text{ for some }j\\0&:\text{ otherwise}\end{array}\right.$ and $\tilde\mu_i =\left\{\begin{array}{ll}\mu_j&:\text{ if }m_i=n_j\text{ for some }j\\0&:\text{ otherwise}\end{array}\right.$. Since $\sup_i|\tilde\zeta_i|\leq \|y^*_2\|$ and $\sup_i|\tilde\mu_i|\leq \|y^*_2\|$  by unconditionality we obtain that $x_3^* = w^*$-$\sum_i\la_i\tilde\zeta_ix_i^*$ and $x_4^* = w^*$-$\sum_i\xi_i\tilde\mu_ix_i^*$ are well defined and $\|x^*_3\|\leq\|x_1^*\|\|y_2^*\|$, $\|x_4^*\| \leq \|x_2^*\|\|y_1^*\|$. A straightforward calculation yields that if $x^* = x_1^*x_2^*+x_3^*+x_4^*$ and $y^* = y_1^*y_2^*$ then
\[
\begin{split}
T^*(x^*,y^*) &= \left(w^*\text{-}\sum_i\la_ie_{m_i}^* + w^*\text{-}\sum_i\mu_ie_{n_i}^*\right)\left(w^*\text{-}\sum_i\xi_ie_{m_i}^* + w^*\text{-}\sum_i\zeta_ie_{n_i}^*\right)\\
&=u_1^*u_2^*
\end{split}
\]
and hence $\|u_1^*u^*_2\|\leq \|x_1^*\|\|x_2^*\|+\|x_1^*\|\|y_2^*\|+\|x_2^*\|\|y_1^*\| + \|y_1^*\|\|y_2^*\| = 1$.
\end{proof}

\begin{prp}
\label{all unconditional containing ell1}
Let $X$ be a Banach space with a normalized 1-unconditional basis $(x_i)_i$. Then there exists a Banach space $Y$ with a Schauder basis so that the space $\mathbb{R}I\oplus\mathcal{K}_\mathrm{diag}(Y)$ contains an ideal $\widetilde X$ that is isomorphic as a Banach algebra to $X$ and a subalgebra $\mathcal{A}$ that is isomorphic as a Banach algebra to $\mathrm{bv}_1$ (see Example \ref{variation norm}) so that $\mathbb{R}I\oplus\mathcal{K}_\mathrm{diag}(Y) = \widetilde X\oplus \mathcal{A}$.
\end{prp}

\begin{proof}
For notational purposes we will prove the result for the space $Z = [(x_i^*)_i]$ instead of $X$.  This is clearly sufficient by duality. Take $c_{00}(\N)$ with the norm described by \eqref{sum with summing} from the proof of Proposition \ref{all unconditional containing c0} and its completion $W$. By Lemma \ref{dualizing unconditional plus submultiplicative} the sequence $(e_i^*)_i$ is submultiplicative and since $(e_i)_i$ dominates the summing basis of $c_0$ by Theorem \ref{ADT main thm} the space $\R I\oplus \mathcal{K}_\mathrm{diag}(W)$ ca be naturally identified with $\R e^*_\omega\oplus[(e_i^*)]$, where $e^*_\omega = w^*$-$\sum_ie_i^*$.

By the proof of Proposition \ref{all unconditional containing c0} the operator $Q:W\to W$ given by $Q(\sum_ia_ie_i) = \sum_ia_{2i}(e_{2i} - e_{2i-1})$ is a bounded projection, $(e_{2i}-e_{2i-1})_i$ onto $\tilde X = [(e_{2i}-e_{2i-1})_i]$, $\mathrm{ker}(Q) = [(e_{2i-1})_i]$, $(e_{2i}-e_{2i-1})_i$ is equivalent to $(x_i)_i$, and $(e_{2i-1})_i$ is equivalent to the summing basis if $c_0$. Then by duality $Q^*|_{\R e_\omega^*\oplus[(e_i^*)]}$ is a projection onto $(e_{2i}^*)_i$, $\ker(Q^*|_{\R e_\omega^*\oplus[(e_i^*)]}) = \R e_\omega^*\oplus[(e_{2i-1}^*+e_{2i}^*)_i]$, $(e_{2i}^*)_i$ is equivalent to $(x_i^*)_i$, and $(e_{2i-1}^*+e_{2i}^*)_i$ is equivalent to the difference basis of $\ell_1$. Setting $\tilde Z = [(e_{2i}^*)_i]$ and $\mathcal{A} = \R e_\omega^*\oplus[((e_{2i-1}^*+e_{2i}^*)_i)]$ the conclusion follows.
\end{proof}

\subsection{Jamesifying unconditional sequences}
We discuss the jamesification of a Banach space with an unconditional basis and its dual. The classical such example is the jamesification $J$ of $\ell_2$ by R. C. James in \cite{J} (hence also the term jamesification).  The purpose is to study these spaces and their duals as Banach algebras of diagonal operators. We recall the definition of the jamesification of a Schauder basic sequence from \cite[Page 21]{BHO}.

\begin{dfn}[\cite{BHO}]
Let $X$ be a Banach space with a normalized Schauder basis $(x_i)_i$. We define a norm on $c_{00}(\N)$ as follows: for every sequence of scalars $(a_i)_i$ that is eventually zero we set
\begin{equation}
\label{jamesification original definition}
\left\|\sum_{i=1}^\infty a_ie_i\right\| = \sup\left\{\left\|\sum_{n=1}^\infty\left(\sum_{i=k_n}^{m_{n}}a_i\right)x_{k_n}\right\|: 1\leq k_1 \leq m_1 < k_2\leq m_2 <\cdots\right\}.
\end{equation}
We denote the completion of $c_{00}(\N)$ with this norm by $J(X)$ and we call it the ``jamesification'' of $(x_i)_i$. We call the sequence $(e_i)_i$ the unit vector basis of $J(X)$.
\end{dfn}
Of course, the space $J(X)$ depends on the basis $(x_i)_i$ of $X$. In the sequel the basis $(x_i)_i$ used to define $J(X)$ will be specified or clear from the context.

\subsection{The spaces $J(X)^*$ and $J(X)$ as Banach algebras of diagonal operators.}
It was observed by A. D. Andrew and W. L. Green in \cite{AG} that $J$ can be viewed as a Banach algebra.  We apply Theorem \ref{ADT main thm} to them to show that for any space $X$ with an unconditional basis the spaces $J(X)$ and $J(X)^*$ may be viewed as Banach algebras of diagonal operators.

\begin{prp}
\label{Jstar is submultiplicative}
Let $X$ be a Banach space with a normalized and $1$-unconditional basis $(x_i)_i$ and let $(e_i)_i$ denote the Schauder basis of its Jamesification $J(X)$. The following hold.
\begin{itemize}
\item[(i)] The sequence $(e_i)_i$ is a normalized and monotone Schauder basis of $J(X)$.
\item[(ii)] For any sequence of scalars $(a_i)_{i=1}^n$ we have $|\sum_{i=1}^na_i| \leq \|\sum_{i=1}^na_ie_i\|$.
\item[(iii)] For any sequences of scalars $(a_i)_{i=1}^n$, $(b_i)_{i=1}^n$ we have
\begin{equation*}
\left\|\sum_{i=1}^na_ib_ie_i^*\right\|\leq 2\left\|\sum_{i=1}^na_ie_i^*\right\|\left\|\sum_{i=1}^nb_ie_i^*\right\|.
\end{equation*}
\end{itemize}
\end{prp}

\begin{proof}
Statements (i) and (ii) easily follow from \eqref{jamesification original definition} and unconditionality. We will use the equivalence of statements (2) and (3) of \cite[Theorem 2.4]{ADT}. For (iii)  it is sufficient to find a 1-norming set $K$ in $J(X)^*$ that is contained in the linear span of $(e_i^*)_i$ and satisfies $K\cdot K\subset 2B_{J(X)^*}$. The will prove that the desired set is
\[
\begin{split}
K = \left\{\sum_{n=1}^\infty c_n\left(\sum_{i=k_n}^{m_{n}}e_i^*\right):\right.& (c_n)_{n=1}^\infty\in c_{00}(\N),\; k_1\leq m_1<k_2\leq m_2<\cdots\\
&\left.\text{and } \left\|\sum_{n=1}^\infty c_nx_{k_n}^*\right\|\leq 1\right\}
\end{split}
\]
By \eqref{jamesification original definition} it is easy to see that $K$ is a 1-norming set. To show that $K\cdot K\subset 2 B_{J(X)^*}$ let $f = \sum_{n=1}^\infty a_n(\sum_{i=k_n}^{d_{n}}e_i^*)$ and $f = \sum_{m=1}^\infty b_n(\sum_{i=l_m}^{t_{m}}e_i^*)$ be in $K$. Then, one can check that $fg$ has the form
$$fg = \sum_{r=1}^\infty c_r\left(\sum_{i=q_r}^{p_{r}}e_i^*\right)$$
where each $q_r$ is either some $k_n$ or some $l_m$. Define the sets $R_f = \{r\in\N: q_r=k_{n_r}$ for some $n_r\in\N\}$ and $R_g = \{r\in\N: q_r = m_{l_r}$ for some $l_r\in\N\}\setminus R_f$ and the functionals
\[
fg|_1 = \sum_{r\in R_f}c_r\left(\sum_{i=q_r = k_{n_r}}^{p_r}e_i^*\right) \text{ and } fg|_2 = \sum_{r\in R_g}c_r\left(\sum_{i=q_r = l_{m_r}}^{p_r}e_i^*\right).
\]
Clearly, $fg = fg|_1 + fg|_2$. Observe that for each $r\in R_f$ there is $\bar m_r\in\N$ so that $c_r = a_{k_{n_r}}b_{l_{\bar m_r}}$. Unconditionality of the basis $(x_i)_i$ yields that
$$\left\|\sum_{r\in R_f}c_rx_{q_r}^*\right\| = \left\|\sum_{r\in R_f}a_{k_{n_r}}b_{l_{\bar m_r}}x_{k_{n_r}}^*\right\| \leq \sup_n|b_n|\left\|\sum_{n=1}^\infty a_kx_{n_k}^*\right\| \leq 1.$$
This implies that $fg|_1\in K$ and similarly it follows that $fg|_2$ is in $K$ as well which yields the conclusion.
\end{proof}

A space that displays similar qualities to the jamesification of a space is a certain subspace of its dual defined below.
\begin{dfn}
\label{define scriptjstar}
Let $X$ be Banach space with a normalized 1-unconditional basis $(x_i)_i$. Denote by $s:J(X)\to \R$ the norm-one linear functional defined by $s(\sum_ia_ie_i) = \sum_ia_i$ and denote by $\mathcal{J}_*(X)$ the subspace of $J(X)^*$
\[\mathcal{J}_*(X) = \R s\oplus [(e^*_i)_i].\]
Denote by $(v_i)_i$ the sequence in $\mathcal{J}_*(X)$ given by $v_1 = s = w^*$-$\sum_{j=1}^\infty e_i^*$, and for all $i\in\N$, $v_{i+1} = s - \sum_{j=1}^ie_j^* = w^*$-$\sum_{j=i+1}^\infty e_j^*$.
\end{dfn}
The closed linear span of $(v_i)_i$ is clearly $\mathcal{J}_*(X)$ and as we shall see later it is a normalized and monotone Schauder basis as well (see Proposition \ref{basic properties scriptjstar} (i)). The space $\mathcal{J}_*(X)$ is qualitatively similar to the space $J([(x_i^*)_i])$ however they do not coincide in general.

\begin{rmk}
\label{when is it the dual I wonder}
If $X$ has a normalized 1-unconditional basis it is implied by \cite[Theorem 2.2 (2) and Theorem 4.1 (2)]{BHO} that if $\ell_1$ does not embed into $X$ then $\mathcal{J}_*(X) = J(X)^*$.
\end{rmk}

\begin{rmk}
\label{scriptjstar submultiplicative}
If $X$ is a Banach space with a 1-unconditional basis $(x_i)_i$ then the sequence $(e_i^*)_i$ in $J(X)^*$ is submultiplicative with the element $s = w^*$-$\sum_{i=1}^\infty e_i^*$ acting as a multiplicative identity. Hence the space $\mathcal{J}_*(X)$ is submultiplicative (not with $(v_i)_i$) and $s$ acts as an identity on it. In addition, $v_iv_j = v_{\max\{i,j\}}$ for all $i$, $j\in\N$.
\end{rmk}

\begin{prp}
\label{basic properties scriptjstar}
If $X$ has a normalized 1-unconditional basis $(x_i)_i$ then the basis $(v_i)_i$ of $\mathcal{J}_*(X)$ satisfies the following properties.
\begin{itemize}

\item[(i)] The basis $(v_i)_i$ is normalized and monotone.

\item[(ii)] For any scalars $(a_i)_{i=1}^n$ we have $|\sum_{i=1}^na_i| \leq \|\sum_{i=1}^na_iv_i\|$.


\item[(iii)] The unit ball of $\mathcal{J}_*(X)$ is a 1-norming set for $J(X)$, hence $J(X)$ is naturally isometric to a subspace of $(\mathcal{J}_*(X))^*$ with the identification $v_1^* = e_1$, and for all $i\in\N$, $v_{i+1}^* = e_{i+1} - e_i$. In particular, the closed linear span of $[(v_i^*)_i]$ is isometrically isomorphic to $J(X)$.

\item[(iv)] For any sequence of scalars $(a_i)_i$ that is eventually zero we have
\begin{equation}
\label{variation form}
\left\|\sum_{i=1}^\infty a_iv_i^*\right\| = \sup\left\{\left\|\sum_{n=1}^\infty\left(a_{k_n} - a_{m_n}\right)x_{k_n}\right\|: 1\leq k_1 < m_1\leq k_2<m_2\leq\cdots\right\}.
\end{equation}
\item[(v)] For any sequences of scalars $(a_i)_{i=1}^n$ and $(b_i)_{i=1}^n$ we have
\begin{equation}
\label{james variation submultiplicative}
\left\|\sum_{i=1}^\infty a_ib_iv_i^*\right\| \leq 2\left\|\sum_{i=1}^\infty a_iv_i^*\right\|\left\|\sum_{i=1}^\infty b_iv_i^*\right\|.
\end{equation}
\end{itemize}

\end{prp}

\begin{proof}
The second statement is a consequence of the fact that $(e_i)_i$ is normalized. For (iii) observe that the linear span of $(e_i^*)_i$ is a subset of $\mathcal{J}_*(X)$ and use the monotonicity of $(e_i)_i$. Then  (iv) follows from (iii) and \eqref{jamesification original definition}. Furthermore, (iv) implies that $(v_i^*)_i$ is monotone and \eqref{jamesification original definition} yields that $\|v_i\| = 1$ for all $i\in\N$, i.e. (i) holds. Finally, (v) follows from 1-unconditionality of $(x_k)_k$ and the triangle inequality applied to (iv).
\end{proof}

\begin{rmk}
\label{J(X) submultiplicative}
If $X$ is a Banach space with a  normalized 1-unconditional basis $(x_i)_i$ then the sequence $(v_i^*)_i$ is also a monotone (but not necessarily normalized) basis of $J(X)$ that satisfies \eqref{variation form}. Also, $J(X)$ with this basis satisfies \eqref{james variation submultiplicative}. We call $(v_i^*)_i$ the difference basis of $J(X)$. Hence, $J(X)$ is submultiplicative when endowed with pointwise multiplication with respect to this basis. Note that $e_ie_j = e_{\min\{i,j\}}$ for all $i$, $j\in\N$ and hence $(e_i)_i$ is an approximate identity  however $J(X)$ does not contain a multiplicative identity. If we identify $J(X)$ with a subspace of $(\mathcal{J}_*(X))^*$ then the element $e_\omega= w^*$-$\sum_{i=1}^\infty v_i^* = w^*$-$\lim_ie_i$ acts as a multiplicative identity on $J(X)$. We view the subspace $\R e_\omega\oplus J(X)$ of $(\mathcal{J}_*(X))^*$ as the unitization of $J(X)$.
\end{rmk}

\begin{prp}
\label{james spaces as diagonal operators}
Let $X$ be a Banach space with a 1-unconditional basis $(x_i)_i$. The following hold.
\begin{itemize}

\item[(i)] $\mathcal{L}_\mathrm{diag}(J(X)) \equiv J(X)^*$ and $\R I\oplus \mathcal{K}_\mathrm{diag}(X) \equiv \mathcal{J}_*(X) = \R s\oplus[(e_i^*)_i]$ and

\item[(ii)] $\mathcal{L}_\mathrm{diag}(\mathcal{J}_*(X)) \equiv \mathcal{J}_*(X)^*$ and $\mathcal{K}_\mathrm{diag}(\mathcal{J}_*(X)) \equiv J(X)$.

\end{itemize}
More precisely, the map $A:\mathcal{L}_\mathrm{diag}(J(X)) \to J(X)^*$ defined by
\[A\Big(\mathrm{SOT}\text{-}\sum_{i=1}^\infty\la_i e_i^*\otimes e_i\Big) = w^*\text{-}\sum_{i=1}^\infty\la_i e_i^*\]
is an onto isomorphism with $\|A\|\|A^{-1}\|\leq 2$ and the image of $\R I\oplus \mathcal{K}_\mathrm{diag}(J(X))$ under $A$ is $\mathcal{J}_*(X)$. Also, the map $B:\mathcal{L}_\mathrm{diag}(\mathcal{J}_*(X)) \to \mathcal{J}_*(X)^*$ defined by
\[B\Big(\mathrm{SOT}\text{-}\sum_{i=1}^\infty\la_i v_i^*\otimes v_i\Big) = w^*\text{-}\sum_{i=1}^\infty\la_i v_i^*\]
is an onto isomorphism with $\|B\|\|B^{-1}\|\leq 2$ and the image of $\R\oplus \mathcal{K}_\mathrm{diag}(\mathcal{J}_*(X))$ is $\R e_\omega\oplus J(X)$.
\end{prp}

\begin{proof}
Item (i) follows readily from Proposition \ref{Jstar is submultiplicative} and \cite[Theorem 2.4]{ADT} and the definition of $\mathcal{J}_*(X)$. Item (ii) follows from Proposition \ref{basic properties scriptjstar} and \cite[Theorem 2.4]{ADT} as well.
\end{proof}

\begin{rmk}
\label{james spaces are Calkin}
Taking $\ell_p$, $1<p<\infty$, with the unit vector basis yields that the spaces $\R e_\omega\oplus J_p$ and $J_p^*$, $1<p<\infty$, are of the form $\R I\oplus\mathcal{K}_\mathrm{diag}(Y)$.
\end{rmk}

\begin{cor}
\label{quasi-reflexive any order}
For every $n\in\N$ there is a space $Y$ with a basis so that $\R I\oplus\mathcal{K}_\mathrm{diag}(Y)$ is quasi-reflexive of order $n$.
\end{cor}

\begin{proof}
For $n=2$ take, e.g., the space the space $J = J(\ell_2)$ with the basis $(e_i)_i$ and $Y = J\oplus J$ endowed with the basis $(w_i)_i$ so that $w_{2i-1} = (v_i^*,0)$ and $w_{2i} = (0,v_i^*)$. It is straightforward to check that this basis satisfied the assumptions of Corollary \ref{same space banach algebra} and hence $\mathcal{K}_\mathrm{diag}(Y)$ is isomorphic to $J\oplus J$ which is quasi-reflexive of order two. Adding one dimension does not alter this fact. Of course, this works for any $n\in\N$.
\end{proof}

The following demonstrates some additional examples of spaces and their duals that can be viewed as spaces of the form $\R I\oplus\mathcal{K}_\mathrm{diag}(Y)$. In particular, it applies to a large class of spaces with spreading bases, e.g., to the diagonal of $\ell_q\oplus J_p$ for $1<q<p<\infty$. It follows from \cite{AMS} that this Corollary does not apply to all spaces with conditional spreading bases.
\begin{cor}
\label{this gives many conditional spreading}
Let $X$ and $Y$ be Banach spaces with normalized 1-unconditional bases $(x_i)_i$ and $(y_i)_i$ respectively. Denote by $(\tilde e_i)_i$ the unit vector basis of $J(X)$, define a norm on $c_{00}(\N)$ so that for any sequence of scalars $(a_i)_i$ that is eventually zero
\[\left\|\sum_{i=1}^\infty a_ie_i\right\| = \max\left\{\left\|\sum_{i=1}^\infty a_iy_i\right\|, \left\|\sum_{i=1}^\infty a_i\tilde e_i\right\|\right\},\]
and let $Z$ denote the completion of $c_{00}(\N)$ with respect to this norm.
Then the sequence $(e_i^*)_i$ in $Z^*$ is submultiplicative, $\mathcal{L}_\mathrm{diag}(Z)\equiv Z^*$ and $\R I\oplus \mathcal{K}_\mathrm{diag}(Z) \equiv \R s\oplus [(e_i^*)_i]$. Also, if we define the basis $(d_i)_i$ of $Z$ by setting $d_1 = e_1$ and $d_{i+1} = e_{i+1}-e_i$ for $i\in\N$  then $Z$ endowed with this basis is submultiplicative and $\mathcal{K}_\mathrm{diag}(\R s\oplus[(e_i^*)_i]) \equiv Z$.
\end{cor}

\begin{proof}
The part about $\mathcal{L}_\mathrm{diag}(Z)$ and $Z^*$ can be obtained by combining Lemma \ref{dualizing unconditional plus submultiplicative} with Proposition \ref{james spaces as diagonal operators}. The fact that $(d_i)_i$ is submultiplicative follows from the triangle inequality and Proposition \ref{basic properties scriptjstar} (v) and the fact that $d_i^* = w^*$-$\sum_{j=i}^\infty e_i^*$ and hence $\R s\oplus[(e_i^*)] = [(d_i^*)_i]$.
\end{proof}

\subsection{Spaces with right (or left) dominant unconditional bases as complemented ideals of spaces of compact diagonal operators} We prove that if $X$ has a right dominant unconditional basis then $X$ embeds as a complemented ideal in the space $\mathcal{L}_\mathrm{diag}(\mathcal{J}_*(X\oplus X))$. We also show that if $X$ has a left dominant unconditional basis $(x_i)_i$ then $X$ embeds as a complemented ideal in the space $\mathcal{K}_\mathrm{diag}(J([(x_i^*)_i]\oplus[(x_i^*)_i]))$. This will be used in the next subsection to embed reflexive spaces with unconditional bases as ideals into quasi-reflexive algebras of diagonal operators.

We recall the following Definition from \cite[Page 22]{BHO}.

\begin{dfn}
\label{def right dominant}
An unconditional basis $(x_i)_i$ of a Banach space $X$ is said to be  $C$-right dominant, for some constant $C>0$, if for all $1\leq k_1\leq m_1 < k_2\leq m_2 <\cdots$ and any sequence of scalars $(a_i)_i$ that is eventually zero we have
\begin{equation}
\label{right dominance equation}
\left\|\sum_{i=1}^\infty a_{m_i}x_{k_i}\right\| \leq C\left\|\sum_{i=1}^\infty a_{m_i}x_{m_i}\right\|.
\end{equation}
We say that $(x_i)_i$ is right dominant if it is $C$-right dominant for some $C>0$. Reversing the inequality in \eqref{right dominance equation} we obtain the definition of a $1/C$-left dominant sequence.
\end{dfn}
Examples of right dominant sequences are the bases of Tsirelson space from \cite{FJ}, Tsirelson's dual from \cite{T}, Schreier space from \cite{S}, and any subsymmetric sequence such as $\ell_p$ spaces or Schlumprecht space \cite{Sch}. These examples are left dominant as well and the only case in which this is not straightforward is Tsirelson space where it follows from \cite[Proposition 6]{CJS}.

The property of being right dominant can be reformulated as follows. A sequence $(x_i)_i$ is right dominant if there exists a constant $C$ so that for every strictly increasing sequence of natural numbers $(k_i)_i$ and every sequence $(m_i)_i$ with $m_i\in[k_i,k_{i+1})$  we have that $(x_{k_i})_i$ is $C$-dominated by $(x_{m_i})_i$.

\begin{rmk}
\label{don't have to be so strict about inequalities}
If $(x_i)_i$ is suppression unconditional and $C$-right dominant, $(a_i)_i$ is a sequence of scalars that is eventually zero and $1\leq k_1 < m_1 \leq k_2 < m_2 \leq k_3< m_3\leq\cdots$ then by splitting the even and odd terms of the sequence we get
\begin{equation*}
\begin{split}
\left\|\sum_{i=1}^\infty a_ix_{k_i}\right\| &\leq \left\|\sum_{i=1}^\infty a_{2i}x_{k_{2i}}\right\| + \left\|\sum_{i=1}^\infty a_{2i-1}x_{k_{2i-1}}\right\|\\
&\leq C\left\|\sum_{i=1}^\infty a_{2i}x_{m_{2i}}\right\| + C\left\|\sum_{i=1}^\infty a_{2i-1}x_{m_{2i-1}}\right\|\leq 2C\left\|\sum_{i=1}^\infty a_ix_{m_i}\right\|.
\end{split}
\end{equation*}
\end{rmk}

The following result builds on \cite[Proposition 2.3]{BHO} and, in a sense, it is a slight generalization of it by removing an extra assumption.

\begin{lem}
\label{complemented subspace even parts}
Let $X$ be a Banach space with a normalized 1-unconditional basis $(x_i)_i$ that is right dominant.
The sequence $(e_{2i}-e_{2i-1})_i$ in $J(X)$ is equivalent to $(x_{2i})_i$ and it spans a complemented subspace $\widetilde X_{2\N}$ of $J(X)$ via the map $Q:J(X)\to J(X)$ defined by $Qx = \sum_{i=1}^\infty e_{2i}^*(x)(e_{2i} - e_{2i-1})$. Furthermore, we have that $(I - Q)(x) = \sum_{i=1}^\infty(e_{2i-1}^*+e_{2i}^*)(x)e_{2i-1}$ and the space $J(X)_{2\N - 1} = (I-Q)[J(X)]$ is a subalgebra of $J(X)$.
\end{lem}

\begin{proof}
Let $C$ be the constant for which \eqref{right dominance equation} holds. We first show the equivalence. Note that the basis under consideration is in fact $(v_{2i}^*)_i$ and it satisfies \eqref{variation form}, which immediately yields that it dominates $(x_{2i-1})_i$ with constant one. For the inverse domination let $(a_i)_i$ be a sequence of scalars that are eventually zero and all odd entries are zero as well. Then, there exist $k_1<m_1\leq k_2 < m_2 \leq\cdots$ so that
\[\left\|\sum_{i=1}^\infty a_{2i}v_{2i}^*\right\| =  \left\|\sum_{i=1}^\infty a_{i}v_{i}^*\right\| = \left\|\sum_{n=1}^\infty\left(a_{k_n} - a_{m_n}\right)x_{k_n}\right\|.\]
Denote the quantity in the above equation by $\lambda$ and define
\[K = \left\{n: k_n\text{ is even }\right\},\; M = \left\{n: m_n \text{ is even}\right\}.\]
We then calculate 
\begin{equation}
\label{equation one in complemented subspace even parts}
\begin{split}
\la &= \left\|\sum_{n\in K\cap M}\left(a_{k_n} - a_{m_n}\right)x_{k_n} + \sum_{n\in K\setminus M}a_{k_n}x_{k_n} - \sum_{n\in M\setminus K}a_{m_n}x_{k_n}\right\|  \\
& = \left\|\sum_{n\in K}a_{k_n}x_{k_n} - \sum_{n\in M}a_{m_n}x_{k_n}\right\| \leq \left\|\sum_{n\in K} a_{k_n}x_{k_n}\right\| + \left\|\sum_{n\in M}a_{m_n}x_{k_n}\right\|\\
& \leq \left\|\sum_{n\in K} a_{k_n}x_{k_n}\right\| + 2C\left\|\sum_{n\in M}a_{m_n}x_{m_n}\right\|\text{ (by Remark \ref{don't have to be so strict about inequalities})}.
\end{split}
\end{equation}
By unconditionality we obtain $\la \leq (1+2C)\|\sum_{i=1}^\infty a_{2i}x_{2i}\|$.

For the complementation of the sequence $(v_{2i}^*)_i$ it is enough to show that the map $Q = \mathrm{SOT}-\sum_{i=1}^\infty(v_{2i}-v_{2i+1})\otimes v_{2i}^*$ defines a bounded linear map on $J(X)$. To that end, let $(a_i)_i$ be a sequence of scalars that is eventually zero and set $\mu = \|\sum_{i=1}^\infty a_i u_i^*\|$ as well as $\nu = \|\sum_{i=1}^\infty (a_{2i} - a_{2i+1})v_{2i}^*\|$. Then there are $1\leq k_1 < m_1\leq k_2 < m_2\leq\cdots$ so that if $K = \{n: k_n\text{ is even }\}$ and $M = \{n: m_n \text{ is even}\}$ then
\[\begin{split}
\nu &= \left\|\sum_{n\in K\cap M}\left(a_{k_n} - a_{k_n+1} - a_{m_n}+a_{m_n+1}\right)x_{k_n}\right.\\
&\left. \phantom{AA}+ \sum_{n\in K\setminus M}\left(a_{k_n} - a_{k_n+1}\right)x_{k_n} - \sum_{n\in M\setminus K}\left(a_{m_n} - a_{m_n+1}\right)x_{k_n}\right\|\\
&=\left\|\sum_{n\in K}\left(a_{k_n} - a_{k_n+1}\right)x_{k_n} - \sum_{n\in M}\left(a_{m_n} - a_{m_n+1}\right)x_{k_n}\right\|
\end{split}\]
Repeating the same argument as earlier yields $\nu \leq (1+2C)\|\sum_{i=1}^\infty (a_i-a_{i+1})x_i\| \leq (1+2C)\mu$.

To see that the space $J(X)_{2\N - 1}$ is a subalgebra note that it is spanned by the vectors $(e_{2i-1})_i$. It follows that the sequence $v_1^*$, $v_2^*+v_3^*$, $v_4^*+v_5^*$,\ldots is a basis of $J(X)_{2\N - 1}$ which  clearly yields that $J(X)_{2\N - 1}$ is closed under multiplication.
\end{proof}

\begin{lem}
\label{doubling keeps right dominance}
Let $X$ be a Banach space with a normalized 1-unconditional $C$-right dominant Schauder basis $(x_i)_i$ and let $W = (X\oplus X)_\infty$ endowed with the basis $(w_i)_i$ where $(x_i,0) = w_{2i-1}$ and $(0,x_i) = w_{2i}$ for all $i\in\N$. Then $(w_i)_i$ $5C$-is right dominant.
\end{lem}

\begin{proof}
Let $(k_i)_i$ be strictly increasing, let $(m_i)_i$ be a sequence with $k_i\leq m_i <k_{i+1}$, and let $(a_i)_i$ be a sequence of scalars that is eventually zero. Define
\[\la = \left\|\sum_ia_iw_{k_i}\right\|,\; \mu = \left\|\sum_ia_iw_{m_i}\right\|,\;K = \{i: k_i \text{ even}\}, \text{ and } M = \{i: m_i \text{ even}\}\]
If we set $A = (K\cap M)\cup(K^c\cap M^c)$ then the sets $A$, $K\setminus M$, and $M\setminus K$ form a partition of $\N$. By using that $(w_{2i-1})_i$ and $(w_{2i})_i$ are isometrically equivalent to each other and that they are both $C$-right dominant we obtain the following
\begin{equation*}
\begin{split}
\left\|\sum_{i\in A}a_iw_{k_i}\right\| &\leq C \left\|\sum_{i\in A}a_iw_{m_i}\right\|,\\
\left\|\sum_{i\in K\setminus M}a_iw_{k_i}\right\| &= \left\|\sum_{i\in K\setminus M}a_iw_{k_i-1}\right\| \leq 2C\left\|\sum_{i\in K\setminus M}a_i w_{m_i}\right\| \text{ (by Remark \ref{don't have to be so strict about inequalities})},\\
\left\|\sum_{i\in M\setminus K}a_iw_{k_i}\right\| & = \left\|\sum_{i\in M\setminus K}a_iw_{k_i+1}\right\| \leq C\left\|\sum_{i\in M\setminus K}a_iw_{m_i}\right\|,
\end{split}
\end{equation*}
where in the last inequality we used that for $i\in M\setminus K$ we have $k_i+1\leq m_i <k_{i+1}+1$. We use unconditionality to conclude $\la \leq 5C\mu$.
\end{proof}


The statement of the following result is somewhat lengthy however most of the details are necessary in the sequel.

\begin{prp}
\label{doubling the space gives something clean}
Let $X$ be a Banach space with a right dominant normalized 1-unconditional basis $(x_i)_i$. Let $Y$ denote the space $[(x_i^*)_i]$ and let $W$ denote the space $(X\oplus X)_\infty$ endowed with the basis $(w_i)_i$ where $(x_i,0) = w_{2i-1}$ and $(0,x_i) = w_{2i}$ for all $i\in\N$. Denote by $(e_i)_i$ the unit vector basis of $J(X)$ and by $(\bar e_i)_i$ the unit vector basis of $J(W)$. Also denote by $s:J(X)\to \R$ and $\bar s:J(W)\to\R$ the corresponding summing functionals. The following hold.
\begin{itemize}
\item[(i)] The sequence $(\bar e_{2i} - \bar e_{2i-1})_i$ is equivalent to $(x_i)_i$ and complemented in $J(W)$  via the projection $Qx = \sum_{i=1}^\infty \bar e_{2i}^*(x)(\bar e_{2i} - \bar e_{2i-1})$.
\item[(ii)] The sequence $(\bar e_{2i-1})_i$ is equivalent to $(e_i)_i$ and complemented in $J(W)$ via the projection $(I-Q)(x) = \sum_{i=1}^\infty(\bar e_{2i-1}^*(x) + \bar e_{2i}^*(x))\bar e_{2i-1}$.
\end{itemize}
In particular, the space $\widetilde X = [(\bar e_{2i} - \bar e_{2i-1})_i]$ is an ideal of $J(W)$ that is isomorphic as a Banach algebra to $X$, the space $\widetilde{J(X}) = [(\bar e_{2i-1})_i]$ is a subalgebra of $J(W)$ that is isomorphic as a Banach algebra to $J(X)$, and $J(W) = \widetilde X\oplus \widetilde{J(X})$.
\begin{itemize}
\item[(iii)] The sequence $(\bar e_{2i}^*)_i$ is equivalent to $(x_i^*)_i$ and complemented in $\mathcal{J}_*(W)$ via the projection $R(f) = \sum_{i=1}^\infty (f(\bar e_{2i}) - f(\bar e_{2i-1}))\bar e_{2i}^*$.
\item[(iv)] The sequence $(\bar s)^\frown(\bar e_{2i}^*+\bar e_{2i-1}^*)_i$ is equivalent to $(s)^\frown(e_i^*)_i$ and complemented in $\mathcal{J}_*(W)$ via the map $S:\mathcal{J}_*(W)\to\mathcal{J}_*(W)$ that is defined as follows. $S(\bar s) = \bar s$ and $S|_{[(\bar e_i^*)_i]}(f) = \sum_{i=1}^*f(\bar e_{2i-1})(\bar e_{2i-1}^* + \bar e_{2i}^*)$.
\end{itemize}
In particular, the space $\widetilde Y = [(\bar e_{2i}^*)_i]$ is an ideal of $\mathcal{J}_*(W)$ that is isomorphic as a Banach algebra to $Y$, the space $\widetilde{\mathcal{J}_*(}Y) = \mathbb{R}\bar s\oplus[(\bar e^*_{2i-1}+\bar e^*_{2i})_i]$ is a subalgebra of $\mathcal{J}_*(W)$ that is isomorphic as a Banach algebra to $\mathcal{J}_*(Y)$, and $\mathcal{J}_*(W) = \widetilde Y\oplus \widetilde{\mathcal{J}_*(}Y)$.
\end{prp}

\begin{proof}
By Lemma \ref{doubling keeps right dominance} the sequence $(w_i)_i$ is right dominant. Then, Lemma \ref{complemented subspace even parts} basically contains statement (i) whereas to obtain statement (ii) the only thing missing is that $(\bar e_{2i-1})_i$ is equivalent to $(e_i)_i$. This follows easily from \eqref{jamesification original definition} and the right dominance of $(w_i)_i$. The ``in particular'' part under statement (ii) is also contained in Lemma \ref{complemented subspace even parts}. For statement (iii) merely observe that $R = Q^*|_{\mathcal{J}_*(W)}$ whereas for statement (iv) observe that $S = I - R = (I - Q)^*|_{\mathcal{J}_*(W)}$. The remaining part of the statement is fairly straightforward.
\end{proof}

We can tidy up the statement of Proposition \ref{doubling the space gives something clean} by combining it with Proposition \ref{james spaces as diagonal operators} to obtain the following neat corollaries, which apply to a large class of spaces, e.g. those mentioned after Definition \ref{def right dominant} (spaces with subsymmetric bases, Schreier space, Tsirelson space and its dual). Here, $E = \mathcal{J}_*(X\oplus X)$.

\begin{cor}
\label{what you can get from right dominant}
Let $X$ be a Banach space with a right dominant normalized 1-unconditional basis $(x_i)_i$. Then there exists a Banach space $E$ with a Schauder basis so that the space $\mathcal{K}_\mathrm{diag}(E)$ contains an ideal $\widetilde X$ that is isomorphic to $X$ as a Banach algebra and a subalgebra $\widetilde{J(X})$ that is isomorphic to $J(X)$ so that $\mathcal{K}_\mathrm{diag}(E) = \widetilde X\oplus\widetilde{J(X})$.
\end{cor}

Here, $F = [(x_i^*)]\oplus[(x_i^*)_i]$ which by duality and Lemma \ref{doubling keeps right dominance} is right dominant.
\begin{cor}
\label{what you can get from left dominant}
Let $X$ be a Banach space with a left dominant normalized 1-unconditional basis $(x_i)_i$ and let $Y = [(x_i^*)_i]$. Then there exists a Banach space $F$ with a Schauder basis so that the space $\R I\oplus \mathcal{K}_\mathrm{diag}(F)$ contains an ideal $\widetilde X$ that is isomorphic to $X$ as a Banach algebra and a subalgebra $\widetilde{\mathcal{J}_*(Y})$ that is isomorphic to $\mathcal{J}_*(Y)$ as a Banach algebra so that $\R I\oplus\mathcal{K}_\mathrm{diag}(F) = \widetilde X\oplus\widetilde{\mathcal{J}_*(Y})$.
\end{cor}

\begin{rmk}
Note that in Corollary \ref{what you can get from left dominant} if $X$ is additionally reflexive then $Y = X^*$ and by Remark \ref{when is it the dual I wonder} $\mathcal{J}_*(Y) = \mathcal{J}_*(X^*) = J(X^*)^*$ and hence $\widetilde{\mathcal{J}_*(Y})$ is isomorphic as a Banach algebra to $J(X^*)^*$.
\end{rmk}

\subsection{Spaces with unconditional bases plus Jamesified Tsirelson spaces as spaces with compact diagonal operators} We will show that whenever a space $X$ with an unconditional basis does not contain $c_0$ then there exists an appropriate Tsirelson space $T_\mathcal{M}^{1/2}$ so that $X\oplus J(T_\mathcal{M}^{1/2})^*$ is the Calkin algebra of some space and that whenever  a space $X$ with an unconditional basis does not contain $\ell_1$ then there exists an appropriate Tsirelson space $T_\mathcal{M}^{1/2}$ so that $X\oplus J(T_\mathcal{M}^{1/2})$ is the Calkin algebra of some space. We first discuss some basic notions around Tsirelson spaces.

A a collection $\mathcal{M}$ of finite subsets of $\N$ is called compact if it is compact when naturally identified with a subset of $2^\N$. We say that a collection $(E_i)_{i=1}^n$ of finite subsets of $\N$ is $\mathcal{M}$-admissible if there is a set $\{m_1,\ldots,m_n\}\in \mathcal{M}$ so that $m_1\leq E_1 < m_2\leq E_2<\cdots<m_n\leq E_n$. Given a compact  collection $\mathcal{M}$ and $0<\theta<1$ the Tsirelson space $T_\mathcal{M}^\theta$, defined in \cite{AD}, is the completion of $c_{00}(\N)$ with the uniquely defined norm satisfying the implicit formula
\begin{equation*}
\left\|\sum_{i=1}^\infty a_it_i\right\| = \max\left\{\sup_i|a_i|,\theta\sup\sum_{k=1}^n\left\|\sum_{i\in E_k}a_it_i\right\|: (E_k)_{k=1}^n \text{ is }\mathcal{M}\text{-admissible}\right\}
\end{equation*}
The sequence $(t_i)_i$ forms a 1-unconditional Schauder basis for $T_\mathcal{M}^\theta$. The classical Tsirelson space from \cite{Tsi} (or to be more precise, its dual described in \cite{FJ}) is the space $T = T_\mathcal{S}^{1/2}$ where $\mathcal{S} = \{F\subset\N: \#F\leq\min(F)\}\cup\{\emptyset\}$. It is stated in \cite[Theorem 1]{AD} that the space $T_\mathcal{M}^\theta$ is reflexive whenever $\mathrm{CB}(\mathcal{M})(\theta/(1 + \theta)) > 1$, where $\mathrm{CB}(\mathcal{M})$ denotes the Cantor-Bendixson index of $\mathcal{M}$. In particular, if $\mathrm{CB}(\mathcal{M})$ is infinite then the space $T_\mathcal{M}^\theta$ is reflexive (see also \cite[1.1. Proposition]{AD2}). A special type of collections of finite subsets of $\N$ are the so called regular families. A collection $\mathcal{M}$ is called regular if it is compact, for every $A\in\mathcal{M}$ it contains all  $B\subset A$, and whenever $\{k_1,\ldots,k_n\}\in\mathcal{M}$ and $k_1\leq m_1,\ldots, k_n\leq m_n$ then $\{m_1,\ldots,m_n\}\in\mathcal{M}$. Whenever the family $\mathcal{M}$ is regular it is not very hard to see that the basis $(t_i)_i$ of $T_\mathcal{M}^\theta$ is right dominant.

There are a series of results, starting with \cite{OSZ}, about controlling the norm of certain spaces via norms of Tsirelson spaces. We use some estimates from \cite{OSZ} to prove one such result that we need later. We point out that we will not require the full statement. Here, $Sz(X)$ denotes the Szlenk index of $X$ (see \cite{Sz}).

\begin{prp}
\label{shrinking dominated by tsirelson}
Let $X$ be a Banach space with a normalized monotone shrinking basis $(x_i)_i$. Then there exists a regular family $\mathcal{M}$ so that $(x_i)_i$ satisfies subsequential $15$-$T_\mathcal{M}^{1/2}$ estimates. That is, every normalized block sequence $(y_i)_i$ of $(x_i)_i$ with $k_i = \min\supp(y_i)$ for all $i\in\N$ is $15$-dominated by $(t_{k_i})_i$. Here, $(t_i)_i$ denotes the basis of $T_\mathcal{M}^{1/2}$.
\end{prp}

\begin{proof}
By \cite[Theorem 18]{OSZ} (we use the statement of the theorem for $X=Z$) there exist $1 =m_0<m_1<m_2<\cdots$ and an ordinal number $\al < Sz(X)$ so that for any $1 \leq s_0 < s_1<\cdots$ and any block sequence $(y_i)_i$ of $(x_i)_i$ with $\supp(y_i)\in(m_{s_{i-1}},m_{s_i}]$  is 5-dominated by $(t_{m_{s_{i-1}}})_i$, where $(t_i)_i$ is the basis of the space $T_{\mathcal{F}_\al}^{1/2}$. Here, $\mathcal{F}_\al$ denotes the fine Schreier family that has Cantor-Bendixson index $\alpha + 1$ (see, e.g., \cite[Page 71]{OSZ}). By passing to an appropriate subsequence of $(m_i)_i$ we may assume that $(m_i - m_{i-1})_i$ is nondecreasing. Define
\[
\begin{split}
\mathcal{M} = \left\{\vphantom{\cup_{i=1}^nA_i}\right.&\cup_{i=1}^nA_i:  A_1<\cdots <A_n\text{ and }\#A_i \leq m_{d_i}-m_{d_{i-1}},\\
&\left.\vphantom{\cup_{i=1}^nA_i}\text{where }d_i = \min\{i:\min(A_i) \leq m_{d_i}\}\right\}
\end{split}
\]
The family $\mathcal{M}$ is regular. This follows from the regularity of $\mathcal{F}_\al$ and the fact that $(m_i - m_{i-1})_i$ is nondecreasing. It also follows that $\mathrm{CB}(\mathcal{M}) \leq \omega\alpha + 1$.

We will show that $\mathcal{M}$ is the desired family. For any sequence of scalars $(a_j)_j$ that is eventually zero we have:
\begin{equation}
\label{norming sets proves this got no time for that}
\left\|\sum_{i=1}^\infty\left(\sum_{j=m_{i-1}+1}^{m_i}|a_j|\right)t_{m_{i-1}}\right\|_{T_{\mathcal{F}_\al}^{1/2}}\leq\left\|\sum_{i=1}^\infty\sum_{j=m_{i-1}+1}^{m_i}a_jt_j\right\|_{T_\mathcal{M}^{1/2}}
\end{equation}
 A way to see this is to use norming sets. Take the standard norming set $K$ of $T_{\mathcal{F}_\al}^{1/2}$ and write it as $K = \cup_{s=0}^\infty K_s$ (see, e.g., the proof of \cite[1.1. Proposition]{AD2}) and show by induction on $s$ that for every  $f$ in $K_s$ with $\supp(f)\subset\{m_j:j\in\N\cup\{0\}\}$ there is $g\in B^*_{T_\mathcal{M}^{1/2}}$ for any $i\in\N$ and $j\in(m_{i-1},m_i]$ we have $g(e_j) = f(e_{m_{i-1}})$.
We leave the details to the reader.

Let $(x_j)_{j=1}^n$ be a block sequence in $X$ and for each $j$ set $k_j = \min\supp(x_j)$ and also set $B_j = \{i:\ran(x_j)\cap(m_{i-1},m_i]\neq\emptyset\}$. The sets $B_j$ are intervals of $\N$ clearly satisfying $\max(B_j)\leq\min(B_{j+1})$. This easily implies that we may define a partition $C_0$, $C_1$, $C_2$ of $\{1,\ldots,n\}$ so that for that for each $\e=0,1,2$ and each $j_1,j_2\in C_\e$ we have that $B_{j_1}$ and $B_{j_2}$ either are singletons that coincide or they are disjoint. 
For $i\in\N$ and $\e=0,1,2$ the set $A_i^\e = \{j\in C_\e:\min(B_j) = m_{i}\}$. Set $D_\e = \{i:A_i^\e\neq \emptyset\}$ and for $i\in D_\e$ deifne $y_i^\e = \sum_{j\in A_i^\e}x_j$. Because of the property of $C_\e$ we obtain that $(y^\e_i/\|y^\e_i\|)_{i\in D_\e}$ is 5-dominated by $(t_{m_{i-1}})_{i\in D_\e}$. We calculate
\[
\begin{split}
\left\|\sum_{j\in C_\e}x_j\right\| &= \left\|\sum_{i\in D_\e}y_i^\e\right\| \leq 5\left\|\sum_{i\in D_\e}\|y_i^\e\|t_{m_{i-1}}\right\|_{T_{\mathcal{F}_\al}^{1/2}} =  5\left\|\sum_{i\in D_\e}\left\|\sum_{j\in A_i^\e}x_j\right\|t_{m_{i-1}}\right\|_{T_{\mathcal{F}_\al}^{1/2}}\\
&\leq 5\left\|\sum_{i\in D_\e}\left(\sum_{j\in A_i^\e}\left\|x_j\right\|\right)t_{m_{i-1}}\right\|_{T_{\mathcal{F}_\al}^{1/2}}\text{(triangle ineq. and uncond.)}\\
& \leq 5\left\|\sum_{i\in D_\e}\sum_{j\in A^\e_i}\left\|x_j\right\|t_{k_j}\right\|_{T_\mathcal{M}^{1/2}}\text{(by \eqref{norming sets proves this got no time for that})}\\
& = 5\left\|\sum_{j\in C_\e}\|x_j\|t_{k_j}\right\|_{T_\mathcal{M}^{1/2}} \leq 5\left\|\sum_{j=1}^n\|x_j\|t_{k_j}\right\|_{T_\mathcal{M}^{1/2}}\text{ (by uncond.)}
\end{split}
\]
The conclusion follows by adding the estimates for $C_0$, $C_1$, and $C_2$.
\end{proof}

\begin{rmk}
The proof of Proposition \ref{shrinking dominated by tsirelson} and \cite[Theorem 6.2]{C} yield the following. If $\mathrm{Sz}(X) = \omega$ then $\mathrm{CB}(\mathcal{M}) < \omega^2$ whereas if $\mathrm{Sz}(X) > \omega$ then $\mathrm{CB}(\mathcal{M}) < \mathrm{Sz}(X)$.
\end{rmk}

The following says that spaces with an unconditional basis that do not contain $c_0$ are embedded as complemented ideals into quasi-reflexive spaces of the form $\R I\oplus\mathcal{K}_\mathrm{diag}(Y)$.

\begin{prp}
\label{unconditional withtout c0 plus jamesification of tsirelson}
Let $X$ be a Banach space with a normalized and 1-unconditional basis $(x_i)_i$ that does not contain an isomorphic copy of $c_0$. Then there exists a regular family $\mathcal{M}$, and a Banach space $Y$ with a Schauder basis so that the space $\R I\oplus\mathcal{K}_\mathrm{diag}(Y)$ contains a complemented ideal $\widetilde X$ that is isomorphic as a Banach algebra to $X$ and a subalgebra $\mathcal{A}$ that is isomorphic as a Banach algebra to $J(T_\mathcal{M}^{1/2})^*$ so that $\R I\oplus\mathcal{K}_\mathrm{diag}(Y) = \widetilde X\oplus \mathcal{A}$.
\end{prp}

\begin{proof}
Take $\mathcal{M}$ given by Proposition \ref{shrinking dominated by tsirelson} applied to the space $[(x_i^*)_i]$.
This means that for any sequence of scalars $(a_i)_i$ that is eventually zero we have
\begin{equation}
\label{dominates tsirelson dual}
\left\|\sum_{i=1}^\infty a_ix_i\right\| \geq \frac{1}{15}\left\|\sum_{i=1}^\infty a_it_i^*\right\|
\end{equation}
where $(t_i^*)_i$ is the basis of $(T_\mathcal{M}^{1/2})^*$. Let $Z = [(t_i)_i]\oplus[(t_i)_i] = T_\mathcal{M}^{1/2}\oplus T_\mathcal{M}^{1/2}$ endowed with the basis $(z_i)_i$ as in the statement of Lemma \ref{doubling keeps right dominance} which is right dominant, let $(e_i)_i$ denote the unit vector basis of $J(Z)$ and let $(v_i)_i$ denote the basis of the space $\mathcal{J}_*(Z) = J(Z)^*$(equality follows from Remark \ref{when is it the dual I wonder}). Also denote by $(\bar e_i)_i$ the unit vector basis of $J(T_\mathcal{M}^{1/2})$,

Set $w_i = v_1 - v_{i+1} = \sum_{j=1}^ie_i^*$.  We define a norm on $c_{00}(\N)$ given by
\begin{equation}
\label{definition norm of this space for reflexive and tsirelson}
\left\|\sum_{i=1}^\infty a_i\tilde e_i\right\| = \max\left\{\left\|\sum_{i=1}^\infty a_{2i-1}x_{i}\right\|,\left\|\sum_{i=1}^\infty a_iw_{i}\right\|\right\}
\end{equation}
and denote by $W$ the completion of $c_{00}(\N)$ with respect to this norm. The sequence $(d_i)_i$ given by $d_1 = \tilde e_1$ and for $i\in\N$ $d_{i+1} = \tilde e_{i+1} -\tilde e_i$ is a Schauder basis of $W$ and it satisfies the formula
\[
\left\|\sum_{i=1}^\infty a_id_i\right\| = \max\left\{\left\|\sum_{i=1}^\infty(a_{2i-1} - a_{2i})x_i\right\|,\left\|\sum_{i=1}^\infty a_{i}e_i^*\right\|\right\}.\]
It follows that $W$ endowed with $(d_i)_i$ satisfies the assumptions of Corollary \ref{same space banach algebra}. Hence, it is enough to show the desired decomposition for the unitization of $W$. Furthermore
\[
\left\|\sum_{i=1}^\infty a_id_{2i}\right\| = \max\left\{\left\|\sum_{i=1}^\infty a_{i}x_i\right\|,\left\|\sum_{i=1}^\infty a_{i}e_{2i}^*\right\|\right\}.\]
By Proposition \ref{doubling the space gives something clean} (iii) the sequence $(e_{2i}^*)_i$ is equivalent to $(t_i^*)_i$, and by \eqref{dominates tsirelson dual} we are able to conclude that $(d_{2i})_i$ is equivalent to $(x_i)_i$. In fact, by Proposition \ref{doubling the space gives something clean} (iii) the map $R(\sum_ia_i e_i^*) = \sum_i(a_{2i}-a_{2i-1})e_{2i}^*$ is a bounded linear projection which easily implies that $\tilde R(\sum_ia_i d_i) = \sum_i(a_{2i}-a_{2i-1})d_{2i}$ is a bounded linear projection onto the ideal $\widetilde X = [(d_{2i})]$ the kernel of which is the subalgebra $\tilde {\mathcal{A}} = [(d_{2i}+d_{2i-1})_i] = [(\tilde e_{2i})_i]$. Clearly, by \eqref{definition norm of this space for reflexive and tsirelson}, the sequence $(\tilde e_{2i})_i$ is equivalent to the sequence $(w_{2i})_i$ which implies that $(d_{2i}+d_{2i-1})_i$ is equivalent to $(e_{2i}^*+e^*_{2i-1})$ which by Proposition \ref{doubling the space gives something clean} (iv) is equivalent $(\bar e_i^*)_i$. Setting $\mathcal{A} = \R e_\omega\oplus \tilde{\mathcal{A}}$ in the unitization of $W$ concludes the proof.
\end{proof}

The proof of the following result uses the proof of Proposition \ref{unconditional withtout c0 plus jamesification of tsirelson} and Lemma \ref{dualizing unconditional plus submultiplicative}. We omit its proof because it is very similar to the proof of Proposition \ref{all unconditional containing ell1}.
\begin{prp}
\label{unconditional withtout ell1 plus jamesification of tsirelson}
Let $X$ be a Banach space with a normalized 1-unconditional basis that does not contain an isomorphic copy of $\ell_1$. Then there exists a regular family $\mathcal{M}$, and a Banach space $Y$ with a Schauder basis so that the space $\R I\oplus\mathcal{K}_\mathrm{diag}(Y)$ contains a complemented ideal $\widetilde X$ that is isomorphic as a Banach algebra to $X$ and a subalgebra $\mathcal{A}$ that is isomorphic as a Banach algebra to $\R e_\omega\oplus J(T_\mathcal{M}^{1/2})$ so that $\R I\oplus\mathcal{K}_\mathrm{diag}(Y) = \widetilde X\oplus \mathcal{A}$.\end{prp}


\begin{rmk}
As the proof of Proposition \ref{unconditional withtout c0 plus jamesification of tsirelson} clearly indicates, if the basis of $X$ dominates the unit vector basis of $\ell_q$, for some $1< q <\infty$, then in the conclusion of Proposition \ref{unconditional withtout c0 plus jamesification of tsirelson} $\mathcal{A}$ may be taken to be isomorphic as a Banach algebra to $J_p^* = J(\ell_p)^*$ where $1/p+1/q = 1$. This happens, e.g., if the space $X$ has non-trivial co-type. Similarly, if the basis of $X$ is dominated by the unit vector basis of $\ell_p$ for some $1<p<\infty$ then in the conclusion of Proposition \ref{unconditional withtout ell1 plus jamesification of tsirelson} $\mathcal{A}$ may be taken to be isomorphic as a Banach algebra to $J_p = J(\ell_p)$. This happens when, e.g., the space $X$ has non-trivial type.
\end{rmk}

\section{Control on diagonal operators via horizontally block finite representability}
\label{easy space}
Given a Banach space $X$ with a normalized Schauder basis $(e_k)_k$ and a sequence of Banach spaces $(X_k)_{k=1}^\infty$, all having normalized Schauder bases, we shall define a type of direct sum $\mathcal{X} = (\sum\oplus X_k)^{X}_\mathrm{utc}$ the outside norm of which is a mixture of $c_0$ alongside finite, but arbitrarily large, pieces of the basis of the space $X$. This sum is fairly simply defined and it has the property that the sequence of projection operators $(I_k)_k$, each on the space $X_k$ within the sum, is equivalent (in the operator norm), to the sequence of diagonal operators $(e_k^*\otimes e_k)_k$ in $\mathcal{L}_\mathrm{diag}(X)$. We will use this space $\mathcal{X}$ in the sequel and it exhibits certain key properties given to it by its components. We will use these components in Section \ref{definition space section} to define the more complicated space.

\subsection{The definition of $\mathcal{X} = (\sum\oplus X_k)^{X}_\mathrm{utc}$}
\label{the easy section}
Let us fix a Banach space $X$ with a normalized Schauder basis $(e_i)_{i=1}^\infty$ 
with bimonotone constant $A_0$. That is, for each interval $E$ of $\N$ the natural projection on the  vectors $(e_{k})_{k\in E}$ has norm at most $A_0$. Let us also fix a sequence of Banach spaces $(X_k)_{k=1}^\infty$ each one of which has a normalized Schauder basis $(t_{k,i})_{i=1}^\infty$. We do not, yet, make any additional assumption on the bases $(t_{k,i})_{i=1}^\infty$.

Let us denote by $c_{00}(X_k)_k$ the vector space of all sequence $(x_k)_k$ where $x_k\in X_k$ for all $k\in\N$ and only finitely many entries are non-zero. For each $k\in\N$ we naturally identify a vector $x\in X_k$ with the sequence $(x_k)_k$ the $k$'th entry of which is $x$ and all other entries are $0$. Similarly define $c_{00}(X^*_k)_k$ and make the same identification. We define two subset of  $c_{00}(X^*_k)_k$, namely
\begin{equation*}
\begin{split}
\mathcal{G}^{\mathrm{utc}} &= \left\{\sum_{k=1}^{i_0}a_kt^*_{k,i_0}:\;i_0\in\N \text{ and } \exists \;(a_k)_{k=1}^\infty\text{ such that }\left\|w^*\!\text{-}\!\sum_{k=1}^\infty a_ke_k^*\right\|_{X^*}\leq 1\right\} \text{ and }\\
\mathcal{G}^{\mathrm{utc}}_0 &= \mathcal{G}^{\mathrm{utc}}\cup \left(\cup_{k=1}^\infty\left(\frac{1}{A_0} B_{X_k^*}\right)\right) \text{ (each } B_{X_k^*} \text{ is viewed in } c_{00}(Y^*_k)_k).
\end{split}
\end{equation*}

\begin{rmk}
The term utc stems from the fact that if we view each element $x^* = (x_k^*)_k$ of $\mathcal{G}^{\mathrm{utc}}$ as a matrix $((a_{k,i})_{k=1}^\infty)_{i=1}^\infty$ where each $x_k^* = w^*$-$\sum_ia_{k,i}t_{k,i}^*$, then this matrix has non-zero entries in only one column $i_0$ and only above the diagonal.
\end{rmk}

For each $x = (x_k)_k$ and $x^* = (x_k^*)_k$ in $c_{00}(X_k)_k$ and $c_{00}(X_k^*)_k$ respectively we define $x^*(x) = \sum_{k=1}^\infty x_k^*(x_k)$. We now define a norm for $x = (x_k)_k$ in $c_{00}(X_k)$:
\begin{equation*}
\|x\|_{\mathcal{X}} = \sup\left\{x^*(x):x^*\in\mathcal{G}^{\mathrm{utc}}_0\right\}.
\end{equation*}
We set $\mathcal{X} = (\sum\oplus X_k)^{X}_\mathrm{utc}$ to be the completion of $c_{00}(X_k)$ endowed with this norm. We also denote for each $n\in\N$, for later use, by $\mathcal{X}_n = (\sum_{k=1}^n\oplus X_k)^{X}_\mathrm{utc}$  the subspace of $\mathcal{X}$ that consists of all $x = (x_k)_k$ with $x_k = 0$ for all $k>n$.

For each vector $x = (x_k)_k$ in $\mathcal{X}$ we define $\supp(x) = \{k:x_k\neq 0\}$. We also define for each $k\in\N$
\begin{equation}
\label{k-support}
\supp_{k}(x) = \{i\in\N: t_{k,i}^*(x_k)\neq0\} = \supp_{(t_{k,i})_i}(x_k).
\end{equation}
We list some facts about the space $\mathcal{X}$.

\begin{prp}
\label{basic facts about simple space}
The space  $\mathcal{X} = (\sum\oplus X_k)^{X}_\mathrm{utc}$ satisfies the following.
\begin{itemize}

\item[(i)] The sequence $(X_k)_k$ forms a shrinking Schauder decomposition of the space with bimonotone constant $A_0$. Hence, for each $k\in\N$ we may define the natural projection $I_k$ the image of which is $X_k$.

\item[(ii)] For every $i_0\in\N$ and sequence of scalars $(a_k)_{k=1}^{i_0}$ we have
\begin{equation*}
\left\|\sum_{k=1}^{i_0}a_kt_{k,i_0}\right\|_{\mathcal{X}} = \left\|\sum_{k=1}^{i_0}a_ke_k\right\|_X. 
\end{equation*}
In particular, $X$ is finitely representable in $\mathcal{X}$.

\item[(iii)] For any vectors $y,$ $w$ in $\mathcal{X}$ so that $\max\supp(y) < \min\supp(w)$ with the property that for all $k\in\N$ the set $\supp_k(y)$ is finite  (i.e., if $y=(x_k)_k$ then each $x_k$ has finite support) and satisfy the condition
$$\max_{k\in\N}\{\max(\supp_k(y))\} < \min\supp(w)$$
we have $\|y + w\|_\mathcal{X} = \max\{\|y\|_\mathcal{X},\|w\|_{\mathcal{X}}\}$. In particular, every normalized block sequence $(w_k)_k$ in $\mathcal{X}$ has, for every $\e>0$, a subsequence that is $(1+\e)$-equivalent to the unit vector basis of $c_0$.

\item[(iv)] For all sequences of scalars $(a_k)_{k=1}^\infty$ we have
\begin{equation*}
\left\|\mathrm{SOT}\text{-}\sum_{k=1}^\infty a_kI_k\right\|= \left\|\mathrm{SOT}\text{-}\sum_{k=1}^\infty a_ke_k^*\otimes e_k\right\|.
\end{equation*}
In particular, $(I_k)_k$ is isometrically equivalent to $(e_k^*\otimes e_k)_k$.
\end{itemize}
\end{prp}

The first two statements are fairly straightforward (except perhaps the shrinking property in (i) which follows from (iii)). We do not explicitly use statements (iii) or (iv) so we do  not include a proof of any of them. We note that (iii) follows from the ``utc'' condition in $\mathcal{G}^{\mathrm{utc}}$. The proof of (iv) is a simplified version of the proof of \cite[Theorem 2.4]{ADT} and later on we prove something similar to this for a more complicated space (see Propositions \ref{lower calkin} and \ref{upper calkin}).

\begin{rmk}
\label{pesky detail}
If for fixed $k_0\in\N$ we take the  natural identification $\mathrm{id}_{k_0}$ of $X_{k_0}$ with a subspace of $\mathcal{X} = (\sum_k\oplus X_k)_X^\mathrm{utc}$ then $\mathrm{id}_{k_0}$ is not necessarily an isometry. In fact, for each $x\in X_{k_0}$ we have $(1/A_0)\|x\| \leq \|\mathrm{id}_{k_0}(x)\| \leq \sup_i\|t_{k_0,i}^*\|\|x\|$. The upper bound comes from the set $\mathcal{G}^\mathrm{utc}$.
\end{rmk}

\begin{rmk}
The assumption that $(t_{k,i})_i$ is a Schauder basis of $X_k$ is not entirely necessary. It is, e.g, sufficient if for each $k$ there is a complemented subspace $W_k$ of $X_k$ so that $(t_{k,i})_i$ is a Schauder basis of $W_k$. In general, what is required is that for each $X_k$ there is a meaningful notion of support with respect to $(t_{k,i})_i$ in the sense that there exists a bounded sequence $(t_{k,i}^*)_i$ in $X_k^*$ that is orthogonal to $(t_{k,i})_i$ so that each vector $w$ can be approximated by a sequence of vectors $(w_j)_j$ so that, for each $j\in\N$, the set $\{i:t_{k,i}^*(w_j)\neq 0\}$ is finite. A bounded Markushevich basis is sufficient as well.
\end{rmk}

\section{An $X$-Bourgain-Delbaen-Argyros-Haydon direct sum of spaces}
\label{definition space section}

In \cite{Z} D. Zisimopoulou defined a Bourgain-Delbaen direct sum $(\sum\oplus X_n)_{\mathrm{AH}}$ of a sequence of separable Banach spaces where the outside norm is based on the Argyros-Haydon construction from \cite{AH}. One of the most important features of this construction is that, under certain assumptions, every bounded linear operator on this space is a multiple of the identity plus a horizontaly compact operator (see Definition \ref{def hor-co}). This is used in \cite{MPZ} where an appropriate choice of the sequence $(X_n)_n$ yields a space with a $C(\omega)$ Calkin Algebra. A careful iteration of this procedure is also implemented in that paper and this leads for each countable compactum $K$, to a space having $C(K)$ as a Calkin Algebra.  In this section we modify the construction of Zisimopoulou by adding the space $(\sum\oplus X_k)^{X}_\mathrm{utc}$ as an ingredient. More precisely, we will define a direct sum of the spaces $X_k$, $k\in\N$ where the outside norm is a mixture of a Bourgain-Delbaen-Argyros-Haydon sum with a ``utc'' sum.  The purpose of this is to obtain a space $\mathcal{Y}_X$ that has the space $\R I\oplus\mathcal{K}_\mathrm{diag}(X)$ as a Calkin Algebra instead of $C(\omega)$. The definition of this space $\mathcal{Y}_X$ is long and technical, however, it is similar to \cite{Z}, and we present it rather comprehensively. We shall shorten the proof of some of the properties of the resulting space, whenever they are almost word for word applicable in the present case, by referring the reader to the appropriate proof in the appropriate paper.

\subsection{Determining the shape of the space $\mathcal{Y}_X$}
We fix for the rest of this paper a Banach space $X$ with a normalized Schauder basis $(e_i)_{i=1}^\infty$ with bimonotone constant $A_0$ and a sequence of Banach spaces $(X_k)_{k=1}^\infty$, each one with a Schauder basis $(t_{k,i})_{i=1}^\infty$ and we set $\mathcal{X} = (\sum\oplus X_k)^{X}_\mathrm{utc}$. We will specify the spaces $X_k$ eventually but this is not yet important. The space $\mathcal{Y}_X$ is going to be a subspace of the following ``large'' space
\begin{equation}
\label{big space}
\mathcal{Z}_{X}^{\infty} = \left(\left(\sum_{k=1}^\infty\oplus X_k\right)^{X}_\mathrm{utc}\!\!\!\!\!\!\oplus\ell_\infty(\Gamma)\right)_\infty \equiv \left(\left(\sum_{k=1}^\infty\oplus X_k\right)^{X}_\mathrm{utc}\!\!\!\!\!\!\oplus\left(\sum_{k=1}^\infty\oplus\ell_\infty(\Delta_k)\right)_\infty\right)_\infty,
\end{equation}
where $\Gamma$ is a countable set that is the union of a collection of pairwise disjoint finite sets $\Delta_k$, $k\in\N$, that will be determined later. The identification in \eqref{big space} is done in the obvious way and each element $z$ in $\mathcal{Z}_X^\infty$ can be represented in the form  $z = (x_k,y_k)_{k=1}^\infty$, where $x_k\in X_k$ and $y_k\in\ell_\infty(\Delta_k)$, for each $k\in\N$. Furthermore, for each $n\in\N$ set $\Ga_n = \cup_{k=1}^n\De_k$ and  
\begin{equation}
\label{initial parts of big space}
\mathcal{Z}_{X}^n = \left(\left(\sum_{k=1}^n\oplus X_k\right)^{X}_\mathrm{utc}\!\!\!\!\!\!\oplus\ell_\infty(\Ga_n)\right)_\infty \equiv \left(\left(\sum_{k=1}^n\oplus X_k\right)^{X}_\mathrm{utc}\!\!\!\!\!\!\oplus\left(\sum_{k=1}^n\oplus\ell_\infty(\Delta_k)\right)_\infty\right)_\infty,
\end{equation}
i.e., the space of all $z = (x_k,y_k)_{k=1}^\infty\in\mathcal{Z}_X^\infty$ with $x_k = 0$ and $y_k = 0$ for all $k>n$. We naturally represent each such $z$ in $\mathcal{Z}_X^n$ by $(x_k,y_k)_{k=1}^n$. For each $n\in\N$ we define $R_n:\mathcal{Z}_X^\infty\to \mathcal{Z}_X^n$ to be the restriction onto the $n$ first coordinates, i.e., $R_n(x_k,y_k)_{k=1}^\infty = (x_k,y_k)_{k=1}^n$ (this vector may also be represented as an infinite sequence by adding zeros to the tail). We point out that $\|R_n\| \leq A_0$ ($R_n$ may fail to have norm one) because of Proposition \ref{basic facts about simple space} (i).

Before defining the embedding of $\mathcal{Y}_X$ into $\mathcal{Z}_X^\infty$ we discuss certain ingredients on the latter space. We assume the existence of these ingredient and we don't define them precisely until later, however, in the end this reduces to the definition of the sets $\Delta_k$, $k\in\N$. Let us assume that for each $n\in\N$ and $\gamma\in\De_{n+1}$ we have fixed a bounded linear functional $c_\ga^*:\mathcal{Z}_X^n\to\R$. Define for $n\in\N$ the bounded linear extension operator $i_{n,n+1}:\mathcal{Z}_X^n\to\mathcal{Z}_X^{n+1}$ given by
\begin{equation*}
\begin{split}
i_{n,n+1}(x_k,y_k)_{k=1}^n = (\tilde x_k,\tilde y_k)_{k=1}^{n+1}, \text{ where } \tilde x_k = x_k, \tilde y_k = y_k \text{ for } 1\leq k\leq n \text{ and}\\
\tilde x_{n+1} = 0, \tilde y_{n+1} = (c_\ga^*((x_k,y_k)_{k=1}^n))_{\ga\in\De_{n+1}}.
\end{split}
\end{equation*}

This naturally defines for $m< n\in\N$ the bounded linear extension operator $i_{m,n}:\mathcal{Z}_X^m\to\mathcal{Z}_X^{n}$ given by $i_{m,n} = i_{n-1,n}\circ i_{n-2,n-1}\circ\cdots\circ i_{m,m+1}$. For $n=m$ we take $i_{n,n}$ to be the identity on $\mathcal{Z}_X^n$. The following are easy to see.

\begin{rmk}
\label{partial extensions remark}
For all $l\leq m\leq n$ we have
\begin{equation*}
i_{l,n} = i_{m,n}\circ i_{l,m} = i_{m,n}\circ R_m\circ i_{l,n}.
\end{equation*}
Moreover, for  $m\leq n\in\N$ and $z = (x_k,y_k)_{k=1}^m\in\mathcal{Z}_X^m$ if  $i_{m,n}(z) = (\tilde x_k,\tilde y_k)_{k=1}^n$ then
\begin{itemize}
 \item[(i)] $\tilde x_k = x_k$ and $\tilde y_k = y_k$ for $1\leq k\leq m$,
 \item[(ii)] $\tilde x_k = 0$ for $m<k\leq n$, and
 \item[(iii)] for each $m<k\leq n$ and $\ga\in\De_k$ we have $e_\ga^*(i_{m,n}(\tilde y_k)) = c_\ga^*(i_{m,k-1}(z))$,
\end{itemize}
where $e_\ga^*$ denote the coordinate functionals on $\ell_\infty(\Delta_k)$.
\end{rmk}

For each $m\leq n$ we define the bounded linear operator $P_m^{(n)}:\mathcal{Z}_X^\infty\to \mathcal{Z}_X^n$
\begin{equation*}
P_m^{(n)} = i_{m,n}\circ R_m,
\end{equation*}
which may be also viewed as an operator on $\mathcal{Z}_X^n$. One can easily check the following.
\begin{rmk}
Let $l,m\leq n\in\N$. Then
\label{partial projections remark}
\begin{itemize}
\item[(i) ] $P_m^{(n)}$ is a projection,
\item[(ii)] $P^{(n)}_lP_{m}^{(n)} = P^{(n)}_{\min\{l,m\}}$, and
\item[(iii)] $P_n^{n}[\mathcal{Z}_X^\infty] = R_n[\mathcal{Z}_X^\infty] = \mathcal{Z}_X^n$.
\end{itemize}
\end{rmk}

We define for $l< m\leq n$ the bounded linear projection $P_{(l,m]}^{(n)} = P_{m}^{(n)} - P_{l}^{(n)}$. We will now make an additional assumption on the form of the functionals $(c_\ga^*)_{\ga\in\De_{n+1}}$, $n\in\N$. Recall that the sequence $(X_k)_k$ is a Schauder decomposition of the space $(\sum\oplus X_k)^{X}_\mathrm{utc}$ with bimonotone constant $A_0$. Let us fix $0<\beta_0 < 1/A_0$ and assume that for every $n\in\N$ and $\ga\in\De_{n+1}$ there are $\beta\in[-\beta_0,\beta_0]$ and $b^*$ in the unit ball of $(\mathcal{Z}_X^n)^*$ so that
\begin{subequations}
\begin{align}
c_\ga^* &= e_\eta^* + \beta b^*\circ P^{(n)}_{(m,n]}, \text{ where } 1\leq m<n \text{ and } \eta\in\De_m,\text{ or}\label{BD form a}\\
c_\ga^* &= \beta b^*\circ P_n^{(n)}.\label{BD form b}
\end{align}
\end{subequations}
\begin{rmk}
\label{only dependence on the past}
It is important to note that whether properties \eqref{BD form a} and \eqref{BD form b} of the functionals $c_\ga^*$, $\ga\in\De_{n+1}$ are satisfied is witnessed on the space $\mathcal{Z}_X^n$ and it does not depend on the entire space $\mathcal{Z}_X^\infty$.
\end{rmk}
Although the proof of the following is identical to that of \cite[Proposition 5.1]{Z}, we include a short description of it for completeness. We fix
\begin{equation}
\label{extension constant}
C_0 = 1+2\beta_0A_0/(1-\beta_0A_0)
\end{equation}
throughout the rest of the paper.

\begin{rmk}
\label{almost isometric}
If $\beta_0$ is chosen sufficiently close to zero then $C_0$ can be desirably close to one.
\end{rmk}

\begin{lem}
\label{uniformly bounded partial projections}
Let us fix $n\in\N$ and assume that for all $1\leq m\leq n$ and $\ga\in\De_{m+1}$ the functional $c_\ga^*$ satisfies \eqref{BD form a} or \eqref{BD form b}. Then $\|i_{l,m}\|\leq C_0$ (see \eqref{extension constant}) for all $1\leq l\leq m\leq n+1$.
\end{lem}

\begin{proof}
Fix $l\leq n$ and prove the statement by induction on $l\leq m\leq n+1$. For $l=m$ the map $i_{l,l}$ is the identity and there is nothing to prove. Assume that the conclusion holds for all $l\leq d\leq m$ for some $l\leq m \leq n$, which also implies that $\|P^{(m)}_{(d,m]}\|\leq A_0 + C_0A_0$ for all $l\leq d\leq m$. Let $z = (x_k,y_k)_{k=1}^l\in\mathcal{Z}_X^l$ with $\|z\|\leq 1$. If $i_{l,m+1}(z) = (\tilde x_k,\tilde y_k)_{k=1}^{m+1}$ then by Remark \ref{partial extensions remark} we can deduce that
\begin{equation*}
\|i_{l,m+1}(z)\| = \max\{\|i_{l,m}(z)\|,\|\tilde y_{m+1}\|\} \leq \max\left\{C_0,\max_{\ga\in\De_{m+1}}|c_\ga^*(i_{l,m}(z))|\right\}
\end{equation*}
To complete the proof fix $\ga\in\De_{m+1}$. If $c_\ga^*$ satisfies \eqref{BD form b} then one can check that $|c_\ga^*(i_{l,m}(z))|\leq \beta_0 C_0 < C_0$. Otherwise, there are $1\leq d<m$, $\eta\in\De_d$, $\beta\in[\beta_0,\beta_0]$, and $b^*$ in the unit ball of $(\mathcal{Z}_X^{m})^*$ with $c_\ga^* = e_\eta^* + \beta b^*\circ P^{(m)}_{(d,m]}$. If $d\leq l$ then it follows that $e_\eta^*(i_{l,m}(z)) = e_\ga^*(z)$ therefore $|c_\ga^*(i_{l,m}(z))| \leq 1 + \beta_0(A_0 + C_0A_0) = C_0$. If $l<d<m$, then it can be seen that $P^{(m)}_{(d,m]}(i_{l,m}(z)) = 0$, i.e., $|c_\ga^*(i_{l,m}(z))| = |e_\eta^*(i_{l,m}(z))| \leq C_0$.
\end{proof}

Lemma \ref{uniformly bounded partial projections} and Remark \ref{partial projections remark} allows us to define for each $n\in\N$ the extension operator $i_n:\mathcal{Z}_X^n\to\mathcal{Z}_X^\infty$ with $i_n(x) = \lim_mi_{n,m}(x)$. This extension is well defined due to Remark \ref{partial extensions remark} (ii). Let us restate  Remark \ref{partial extensions remark} in the language of the new extension operators.

\begin{rmk}
\label{infinite extensions remark}
For all $n\in\N$ we have
\begin{equation}
\label{boundedness and compatibility}
 \|i_n\| \leq C_0 \text{ and for all }m\leq n \text{ we have } i_m = i_n\circ R_n\circ i_m.
\end{equation}
Moreover, for  $n\in\N$ and $z = (x_k,y_k)_{k=1}^n\in\mathcal{Z}_X^n$ if  $i_{n}(z) = (\tilde x_k,\tilde y_k)_{k=1}^\infty$ then
\begin{itemize}
 \item[(i)] $\tilde x_k = x_k$ and $\tilde y_k = y_k$ for $1\leq k\leq n$,
 \item[(ii)] $\tilde x_k = 0$ for all $k>n$, and
 \item[(iii)] for each $k>n$ and $\ga\in\De_k$ we have $e_\ga^*(i_{n}(\tilde y_k)) = c_\ga^*(i_{n,k-1}(z))$,
\end{itemize}
where $e_\ga^*$ denote the coordinate functionals on $\ell_\infty(\Delta_k)$. By Lemma \ref {uniformly bounded partial projections} we also have
\begin{equation}
\label{extensions are embeddings}
\|z\| \leq \|i_{n}(z)\| \leq C_0\|z\|.
\end{equation}
\end{rmk}

We now define for each $n\in\N$ the space $Y_X^n = i_n[\mathcal{Z}_X^n]$ and for all $n\in\N$ we define the bounded linear operator $P_n = i_n\circ R_n$.
\begin{rmk}
\label{Schauder decomposition}
The following hold.
\begin{itemize}
 \item[(i)] The space $Y_X^n$ is $C_0$-isomorphic to $\mathcal{Z}_X^n$, for all $n\in\N$ via the map $i_n$ the inverse of which is the map $R_n$ (by \eqref{extensions are embeddings}),
 \item[(ii)] for all $n\in\N$ we have that $P_n:\mathcal{Z}_X^\infty\to Y^n_X$ is a bounded linear projection with $\|P_n\| \leq A_0C_0$, and
 \item[(iii)] for all $m\leq n\in\N$ we have $P_mP_n = P_nP_m = P_m$ (this follows from \eqref{boundedness and compatibility}).
\end{itemize}
We may therefore define for $m\leq n\in\N$ the projection $P_{(m,n]} = P_n - P_m$ which has norm at most $2C_0A_0$. We also write $P_{(0,n]} = P_n$.
\end{rmk}
It follows that the sequence of spaces $(Y_X^n)_n$ is increasing with respect to inclusion. We define $\mathcal{Y}_X$ to be the closure (in the norm topology) of $\cup_nY_X^n$ in the space $\mathcal{Z}_X^\infty$. The space $\mathcal{Y}_X$ admits a Schauder decomposition $(Z_n)_n$ with associated projections $(P_n)_n$. That is, $P_1[\mathcal{Y}_X] = Z_1$ and for $n\geq 2$ $P_{(n-1,n]}[\mathcal{Y}_X] = Z_n$. Using \eqref{boundedness and compatibility} it is not hard to see that in fact
\begin{equation*}
Z_n = i_n[(X_n\oplus\ell_\infty(\Delta_n))_\infty],
\end{equation*}
where $(X_n\oplus\ell_\infty(\Delta_n))_\infty$ is viewed as a subspace of $\mathcal{Z}_X^n$ in the natural way. Hence, we may write $\mathcal{Y}_X = \sum\oplus Z_n$ and if $z\in\mathcal{Y}_X$ then $z = \sum_{n=1}^\infty P_{\{n\}}z$, where  $P_{\{n\}} = P_{(n-1,n]}$. We can define the set
\begin{equation*}
\supp_\mathrm{BD}(z) = \{n\in\N:P_{\{n\}}z\neq 0\}
\end{equation*}
and we also denote by $\ran_\mathrm{BD}(z)$ the smallest interval of $\N$ containing $\supp_\mathrm{BD}(z)$. For $0<m\leq n$ it can be seen that
\begin{equation}
\label{block vectors live in such things}
P_{(m,n]}[\mathcal{Y}_X] = \sum_{k=m+1}^n\oplus Z_k = i_n\left[\mathcal{Z}_X^{(m,n]}\right],
\end{equation}
where
\begin{equation*}
\mathcal{Z}_X^{(m,n]} = \left(\left(\sum_{k=m+1}^n\oplus X_k\right)_{\mathrm{utc}}^X\!\!\!\!\!\!\oplus \left(\sum_{k=m+1}^n\oplus\ell_\infty(\Delta_k)\right)_\infty\right)_\infty 
\end{equation*}
is viewed as a subspace of $\mathcal{Z}_X^n$ in the natural way.

\begin{rmk}
\label{how things are normed}
As $\mathcal{Y}_X$ is a subspace of $\mathcal{Z}_X^\infty$, every $z\in\mathcal{Y}_X$ is of the form $z = (x_k,y_k)_{k=1}^\infty$ and $\|z\| = \max\{\|(x_k)\|_\mathcal{X},\|(y_k)_k\|_{\ell_\infty(\Ga)}\}$. By setting $x = (x_k)_k\in\mathcal{X}$ and $y = (y_k)_k\in\ell_\infty(\Ga)$ we obtain
\begin{equation}
\label{norm triad}
\begin{split}
\|x\| = \max&\left\{\sup\{|x^*(x)|: x^*\in \mathcal{G}^\mathrm{utc}\},\right.\\
&\;\;\sup\{|x^*(x_k)|: x^*\in (1/A_0)B_{X^*_k}: k\in\N\},\\
&\;\;\left.\sup\{|e_\ga^*(y)|: \ga\in\Ga\}\right\}.
\end{split}
\end{equation}
\end{rmk}

\subsection{The precise definition of the space $\mathcal{Y}_X$}
We have established the general form of the space $\mathcal{Y}_X$. To abide by the assumptions of Lemma \eqref{uniformly bounded partial projections} we have to fix $0<\beta_0<1/A_0$ and choose a sequence of disjoint finite sets $(\De_n)_n$ and bounded linear functionals $(c_\ga^*)_{n\in\De_{n+1}}$ defined on $\mathcal{Z}_X^n$ so that \eqref{BD form a} and \eqref{BD form b} are satisfied. Crucially, by Remark \ref{only dependence on the past}, we are allowed to choose by induction on $n$ the set $\De_n$ and the corresponding functionals $(c_\ga^*)_{\ga\in\De_n}$ as they act each time only on what has been defined so far.  

First, we fix a pair of increasing sequences of natural numbers $(\tilde m_j,\tilde n_j)_j$ that satisfy \cite[Assumption 2.3]{AH}, namely
\begin{equation}
\label{mjnjproperties}
\text{(1) } \tilde m_1\geq 4,\text{ (2) } \tilde m_{j+1}\geq \tilde m_j^2, \text{ (3) } \tilde n_1 \geq \tilde m_1^2, \text{ and  (4) } \tilde n_{j+1}\geq(16\tilde n_j)^{\log_2(\tilde m_{j+1})}.
\end{equation}
Let us next choose an infinite sequence of pairwise disjoint infinite subsets of the natural numbers $(L_k)_{k=0}^\infty$. For each $k\in\N$ we define $X_k$ to be the Argyros-Haydon space defined in \cite[Section 10.2]{AH} using the sequence $(\tilde m_j,\tilde n_j)_{j\in L_k}$. Each of these spaces has the ``scalar-plus-compact'' property (i.e., every $T:X_k\to X_k$ is a compact perturbation of a scalar operator) and for $k\neq m$ every bounded linear operator $T:X_k\to X_m$ is compact (see \cite[Theorem 10.4]{AH}). We will use the set $L_0 = \{\ell^0_1<\ell^0_2<\cdots\}$ to define the outside norm of the direct sum. Henceforth we write $\tilde m_{\ell^0_j} = m_j$ and $\tilde n_{\ell^0_j} = n_j$. We make the assumption that $\beta_0 = 1/m_{1} < 1/A_0$.
 
Let us choose for each $k\in\N$ a $1$-norming countable and symmetric subset $\tilde F_k$ of the unit ball of $X_k^*$, and let $F_k^n$ be the symmetric subset of $(1/A_0)\tilde F_k$ set consisting of the first $n$-elements of $(1/A_0)\tilde F_k$ and their negatives. Each set $F_k^n$ may be naturally identified with a subset of the unit ball of $((\sum\oplus X_k)^X_{\mathrm{utc}})^*$. For each $n\in\N$ define $K_n = \cup_{k=1}^n F_k^n$. Let us also take the set $\mathcal{G}^{\mathrm{utc}}$ (see the beginning of Section \ref{the easy section}) which is a symmetric and separable subset of the unit ball of $((\sum\oplus X_k)^X_{\mathrm{utc}})^*$. Choose an increasing sequence of finite symmetric subsets $(\mathcal{G}^\mathrm{utc}_n)_n$ the union of which is dense in $\mathcal{G}^\mathrm{utc}$.

We are prepared to inductively define the sets $(\De_n)_n$ and the corresponding functionals. Set $\De_1 = \{0\}$. There is no need to define $c_\ga^*$ for $\ga\in\De_1$. Assume that we have defined the sets $\De_1,\ldots,\De_n$ and the families of functionals $(c_\ga^*)_{\ga\in\De_k}$, $1\leq k\leq n$. Having defined these elements means having defined the space $\mathcal{Z}_X^n$ (see \eqref{initial parts of big space}). We also assume that to each $\ga\in\Ga_n = \cup_{k=1}^n\De_k$ we have assigned a natural number $\sigma(\ga)$, so that $\max_{\ga\in\De_k}\sigma(\ga)<\min_{\ga\in\De_{k+1}}\sigma(\ga)$ for $1\leq k\leq n-1$ and the map $\sigma:\Ga_n\to \N$ is injective. Furthermore, assume that to each $\ga\in\Ga_n\setminus\Ga_1$ we have a assigned a positive number $\we(\ga) = 1/m_j$, for some $j\in\N$, and a natural number $\ag(\ga) = a$ with $1\leq a\leq n_j$. For each $1\leq k\leq n$, if $\ga\in\De_k$ we also write $\ra(\ga) = k$. 

Set $N_{n+1} = 2^{n}(\#\Ga_n)$ and let $B_{n}$ \label{rational convex combinations} be the set of all linear combination $\sum_{\eta\in\Ga_n}a_\eta e_\eta^*$ with $\sum_{\eta\in\Ga_n}|a_\eta|\leq 1$ and $a_\eta$ is a rational number with denominator dividing $N_{n+1}!$. Set $A_n = K_n\cup\mathcal{G}^{\mathrm{utc}}_n\cup B_n$. The set $\De_{n+1}$ is the union of the following four finite sets consisting of triples and quadruples.
\begin{subequations}
\begin{align}
\De_{n+1}^{\mathrm{Even}_0} =&\left\{(n+1,1/m_{2j},b^*): 2j\leq n+1, b^*\in A_n\right\}\label{even0}\\
\De_{n+1}^{\mathrm{Odd}_0} =&\left\{(n+1,1/m_{2j-1},\eta): 2j-1\leq n+1, \eta\in\Ga_n\text{ with}\right.\label{odd0}\\
&\left.\;\;\we(\eta) = 1/m_{4i-2}<(1/n_{2j-1})^2\right\}\nonumber\\
\De_{n+1}^{\mathrm{Even}_1} =&\left\{(n+1,\xi,1/m_{2j},b^*): \xi\in\Ga_n\text{ with }\we(\xi) = 1/m_{2j},\right.\label{even1}\\
&\left.\;\; \ag(\xi) < n_{2j}, b^*\in A_n\right\}\nonumber\\
\De_{n+1}^{\mathrm{Odd}_1} =&\left\{(n+1,\xi,1/m_{2j-1},\eta): \xi\in\Ga_n\text{ with } \we(\xi) = 1/m_{2j-1}\text{ and}\right.\label{odd1}\\
&\;\; \ag(\xi) < n_{2j-1}, \eta\in\Ga_n\text{ with } \ra(\xi)<\ra(\eta)\text{ and}\nonumber\\
&\left.\;\;\we(\eta) = 1/m_{4\sigma(\xi)}\right\}\nonumber
\end{align}
\end{subequations}
We define for each $\ga\in\De_{n+1}$ the corresponding linear functional $c_\ga^*$. Note that each $e_\eta^*$ for $\eta\in\Ga_n$ and each $b^*$ for $b^*\in A_n$ act as a linear functional on $\mathcal{Z}_X^n$ in the natural way. We set
\begin{align*}
c_\ga^* &= \frac{1}{m_{2j}}b^*\circ P_{(0,n]}^n,\;\ag(\ga) = 1,\text{ and }\we(\ga) = 1/m_{2j}, \text{ if }\ga\in\De_{n+1}^{\mathrm{Even}_0}\\
c_\ga^* &= \frac{1}{m_{2j-1}}e_\eta^*\circ P_{(0,n]}^n,\;\ag(\ga) = 1,\text{ and }\we(\ga) = 1/m_{2j-1}, \text{ if }\ga\in\De_{n+1}^{\mathrm{Odd}_0}\\
c_\ga^* &= e_\xi^* + \frac{1}{m_{2j}}b^*\circ P_{(p,n]}^n,\;\ag(\ga) = \ag(\xi)+1,\text{ and }\we(\ga) = 1/m_{2j},\\
&\quad\text{ if }\ga\in\De_{n+1}^{\mathrm{Even}_1}\text{ and }p = \ra(\xi)\\
c_\ga^* &= e_\xi^* + \frac{1}{m_{2j-1}}e_\eta^*\circ P^n_{(p,n]},\;\ag(\ga) = \ag(\xi)+1,\text{ and }\we(\ga) = 1/m_{2j-1},\\
&\quad\text{ if }\ga\in\De_{n+1}^{\mathrm{Odd}_1}\text{ and }p = \ra(\xi)
\end{align*}
Finally, we extend the definition of the function $\sigma$ to on the set $\De_{n+1}$, so that it remains one-to-one injective on $\Ga_{n+1}$ and $\max_{\ga\in\De_n}\sigma(\ga)<\min_{\ga\in\De_{n+1}}\sigma(\ga)$.
\begin{rmk}
Comparing the definition of this section to the definition presented in \cite[Section 4]{Z}, modulo perhaps certain convexity conditions, the key addition is that we allow the functionals $b^*$ to be chosen from the set $\mathcal{G}^\mathrm{utc}_n$ as well. 
\end{rmk}

\subsection{Basic properties of $\mathcal{Y}_X$}\label{basic properties of space}

For every $n\in\N$ and $\ga\in\De_{n+1}$ the functional $c_\ga^*$ is defined on $\mathcal{Z}_X^n$. We extend its domain to the whole space $\mathcal{Z}_X^\infty$ and hence also to the space $\mathcal{Y}_X$ by taking $c_\ga^*\circ R_n$ and we denote this ``new'' functional by $c_\ga^*$ as well. We also set $c_\ga^* = 0$ if $\ga\in\De_1$. We then define for each $\ga\in\Ga$ the functional $d_\ga^* = e_\ga^* - c_\ga^*$, which is defined on $\mathcal{Z}_X^\infty$ and hence also on $\mathcal{Y}_X$. An important fact is that if $\ga\in\De_n$, then $d_\ga^* = e_\ga^*\circ P_{\{n\}}$.

The following result provides what is called the evaluation analysis of a coordinate functional $e_\ga^*$. It is also used in \cite{Z} (Proposition 5.3) however it appeared earlier in \cite[Proposition 4.5]{AH} where a proof may be found.

\begin{prp}
\label{evanalysis}
Let $n\in\N$ and $\ga\in\De_{n+1}$ with $\we(\ga) = 1/m_j$ and $\ag(\ga) = a\leq n_j$. Then there exist $0=p_0<p_1<\cdots<p_a = n+1$, elements $\xi_1,\xi_2,\ldots,\xi_a$ of weight $1/m_j$ with $\xi_i\in\De_{p_i}$ and $\xi_a = \ga$, and functionals $b_i^*\in A_{p_i-1}$ (see paragraph before \eqref{even0}) so that
\begin{equation}
\label{eanalysis eq}
e_\ga^* = \sum_{i=1}^ad_{\xi_i}^* + \frac{1}{m_j}\sum_{i=1}^ab_i^*\circ P_{(p_{i-1},p_i)}. 
\end{equation}
The sequence $(p_i,\xi_i,b_i^*)_{i=1}^a$ is called the evaluation analysis of $\ga$.
\end{prp}

The following lemma allows to conversely build functionals with a certain prescribed evaluation analysis, provided that certain mild conditions are satisfied. For a proof see \cite[Proposition 4.7]{AH}

\begin{lem}
\label{constructing gammas}
Let $j\in\N$, $1\leq a\leq n_{2j}$, $0 = p_0 < p_1<\cdots<p_a$ with $2j\leq p_1$, and $b_i^*\in A_{p_i-1}$. Then, there exist $\xi_i\in\De_{p_i+1}$ for $1\leq i\leq a$ and a $\ga\in\Ga$ with $\we(\ga) = 2j$ and evaluation analysis $(p_i,\xi_i,b_i^*)_{i=1}^a$. 
\end{lem}

\section{The Calkin algebra of $\mathcal{Y}_X$}
\label{section calkin algebra proof}
In this section we will assume a result from the sequel to prove the desired description of the Calkin algebra of $\mathcal{Y}_X$, namely that it is isomorphic as a Banach algebra to $\R I\oplus\mathcal{K}_\mathrm{diag}(X)$. For $n\in\N$ we define a bounded linear operator $I_n:\mathcal{Z}_X\to\mathcal{Z}_X$ as follows. If $z$ is in $\mathcal{Y}_X$ and $z = (x_k,y_k)_{k=1}^\infty$ is its representation in $\mathcal{Z}^\infty_X$ set $R_{\{n\},0}z = (\tilde x_y, \tilde y_k)_{k=1}^n$ with $\tilde y_k = 0$ for $1\leq k\leq n$, $\tilde x_k = 0$, for $1\leq k<n-1$ and $\tilde x_n = x_n$. Note that $R_{\{n\},0}:\mathcal{Y}_X\to \mathcal{Z}_X^n$ is a bounded linear operator with norm at most $A_0$. We then set $I_n = i_n\circ R_{\{n\},0}$, i.e. $I_n(z) = i_n((0,0),\ldots,(x_n,0))$. The map $I_n$ is a projection of norm at most $A_0C_0$ and its image is $(A_0\sup_i\|t_{n,i}^*\|)$-isomorphic to the space $X_n$. Furthermore, the map $A_{n}:\mathcal{Y}_X\to\mathcal{Y}_X$ defined by $A_{n}z = i_n((0,0),\ldots,(0,y_n))$ is a finite rank operator and hence compact. It is important to note that $I_n = P_{\{n\}} - A_{n}$, that is, $I_n$ is a compact perturbation of $P_{\{n\}}$. The following cannot be proved yet and requires some work. We use it in this section to describe the Calkin Algebra of $\mathcal{Y}_X$ and postpone its proof until much later.

\begin{thm}[Theorem \ref{diagonal plus compact are dense}]
\label{pretty important part of this result}
For every bounded linear operator $T:\mathcal{Y}_X\to\mathcal{Y}_X$ there exists a sequence of real numbers $(a_k)_{k=0}^\infty$ and a sequence of compact operators $(K_n)_n$ so that
$$T = \lim_n\left(a_0 I + \sum_{k=1}^na_kI_k + K_n\right),$$
where the limit is taken in the operator norm.
\end{thm}

\begin{cor}
\label{they generate all of it}
If we denote by $[I]$ and $[I_n]$ the equivalence class of $I$ and $I_n$ respectively in $\mathcal{C}al(\mathcal{Y}_X)$, then the linear span of $\{[I]\}\cup\{[I_n]: n\in\N\}$ is dense in the space $\mathcal{C}al(\mathcal{Y}_X)$.
\end{cor}

We now proceed to make several estimates as to how the norm on the linear span $\{[I]\}\cup\{[I_n]: n\in\N\}$ compares to the norm of the linear span $\{I\}\cup\{e_i^*\otimes e_i: i\in\N\}$ in $\R I\oplus\mathcal{K}_\mathrm{diag}(X)$.

\begin{lem}
\label{norm of this vector is such and such}
Let $i_0\leq i\in\N$ and $(b_k)_{k=1}^{i_0}$ be a sequence of scalars. If
\[u = (x_k,y_k)_{k=1}^{i_0}\in \left(\left(\sum_{k=1}^{i_0}\oplus X_k\right)^X_\mathrm{utc}\oplus\left(\sum_{k=1}^{i_0}\oplus\ell_\infty\left(\De_k\right)\right)_\infty\right)_\infty\]
with $x_k = b_kt_{k,{i}}$ for $1\leq k\leq  {i_0}$ and $z = i_{i_0}(u)$ (which by \eqref{block vectors live in such things} is in $\mathcal{Y}_X$) then we have that
\begin{equation*}
 \left\|\sum_{k=1}^{i_0}b_ke_k\right\|_X \leq \|z\| \leq C_0\max\left\{\left\|\sum_{k=1}^{i_0}b_ke_k\right\|_X, \max_{1\leq k\leq i_0}\|y_k\|_\infty\right\}.
\end{equation*}
In particular, the space $X$ is $C_0$-crudely finitely representable in $\mathcal{Y}_X$.
\end{lem}

\begin{proof}
By \eqref{extensions are embeddings} we have $\|u\| \leq \|z\| \leq C_0\|u\|$ hence it is sufficient to show that $\|\sum_{k=1}^{i_0}b_ke_k\|_X\leq \|u\| \leq \max\{\|\sum_{k=1}^{i_0}b_ke_k\|_X,\max_{1\leq k\leq i_0}\|y_k\|_\infty\}$. We define the vector $x = (\la_kt_{k,{i}})_{k=1}^{i_0}\in (\sum_{k=1}^{i_0}\oplus X_k)^X_\mathrm{utc}$. Then \eqref{norm triad} yields
$$\|u\| = \max\left\{\max_{1\leq k\leq {i_0}}(1/A_0)|\la_k|,\sup_{x^*\in\mathcal{G}^\mathrm{utc}}|x^*(x)|,\max_{1\leq k\leq i_0}\|y_k\|_\infty\right\}$$
From the fact that the basis of $X$ has bimonotone constant $A_0$ we obtain that $\max_{1\leq k\leq {i_0}}|\la_k| \leq A_0\|\sum_{k=1}^{i_0}b_ke_k\|_X$, whereas from the definition of $\mathcal{G}^\mathrm{utc}$ we obtain
\begin{equation*}
\begin{split}
\sup_{x^*\in\mathcal{G}^\mathrm{utc}}|x^*(x)| &= \sup\left\{\left|\sum_{k=1}^{i_0}a_kb_k\right|: (a_k)_{k=1}^\infty\text{ is such that }\left\|w^*\!\text{-}\!\sum_{k=1}^\infty a_ke_k^*\right\|_X\leq 1\right\}\\
&=\left\|\sum_{k=1}^{i_0}b_ke_k\right\|_X.
\end{split}
\end{equation*}
The conclusion immediately follows from the above equations.
\end{proof}

\begin{lem}
\label{lower calkin}
Let $n\in\N$, $(a_k)_{k=0}^{n}$ be a sequence of scalars, and $T:\mathcal{Y}_X\to\mathcal{Y}_X$ be the bounded linear operator $T =a_0I + \sum_{k=1}^na_kI_k$. Then, for every compact operator $K:\mathcal{Y}_X\to\mathcal{Y}_X$ we have
\begin{equation*}
\|T - K\| \geq \frac{1}{C_0}\left\|a_0 I_X + \sum_{k=1}^na_k e_k^*\otimes e_k\right\|_{\mathcal{L}(X)}. 
\end{equation*}
\end{lem}

\begin{proof}
Let us for the moment fix $i_0\in\N$ with $i_0\geq n+1$ and $(b_k)_{k=1}^{i_0}$ so that $\|\sum_{k=1}^{i_0}b_ke_k\|_X\leq  1$. For $i\geq i_0$ define the vector $x_i$ exactly as in the statement of Lemma \ref{norm of this vector is such and such}. We observe that the sequence $(x_i)_{i\geq i_0}$ is weakly null. Recall that for $1\leq k\leq i_0$ the basis $(t_{k,i})_i$ of the space $X_k$ is shrinking (although this is not explicitly stated, it follows easily from the proof of \cite[Proposition 5.2]{AH}). The natural image of $X_k$ in the space $((\sum_{k=1}^{i_0}\oplus X_k)_X^\mathrm{utc}\oplus(\sum_{k=1}^{i_0}\oplus\ell_\infty(\De_k))_\infty)_\infty$ is a ($A_0\sup_{i}\|t_{k,i}^*\|$)-embedding (see Remark \ref{pesky detail}) and the map $i_{i_0}$ is a $C_0$-embedding as well (see \eqref{extensions are embeddings}). If we consider the natural image $(\tilde t_{k,i}^{i_0})_i$ of the sequence $(t_{k,i})_i$ in the aforementioned space then sequence $(i_{i_0}(\tilde t_{k,i}^{i_0}))_{i}$ is weakly null. As $x_i = \sum_{k=1}^{i_0}b_ki_{i_0}(\tilde t_{k,i}^{i_0})$ we have that $(x_i)_i$ is weakly null. This means that $(Kx_i)_i$ converges to zero in norm, i.e. $\liminf_i\|Tx_i - Kx_i\| = \liminf_i\|Tx_i\|$. We combine this with Lemma \ref{norm of this vector is such and such}, according to which $\|x_i\| \leq C_0$ for all $i\geq i_0$, to deduce
\begin{equation}
 \|T - K\| \geq \frac{1}{C_0}\liminf_i\|Tx_i\|.
\end{equation}
For $i\geq i_0$ the vector $Tx_i$ has the form $i_{i_0}(x_k,y_k)_{k=1}^{i_0}$ where $x_k = (a_0 + a_k)b_kt_{k,i}$, for $1\leq k \leq n$ and $x_k = a_0b_kt_{k,i}$ for $n+1 \leq k\leq i_0$. By Lemma \ref{norm of this vector is such and such} we obtain
\begin{equation*}
\begin{split}
 \|Tx_i\| &\geq \left\|\sum_{k=1}^n(a_0 + a_k)b_ke_k + \sum_{k=n+1}^{i_0}a_0b_ke_k\right\|_X\\
 &= \left\|\left(a_0I_X+\sum_{k=1}^na_ke_k^*\otimes e_k\right)\left(\sum_{k=1}^{i_0}b_ke_k\right)\right\|_X
\end{split}
\end{equation*}
Taking a supremum over all $i_0\geq n+1$ and $(b_k)_{k=1}^{i_0}$ as in the first line of the proof yields the desired result. 
\end{proof}

The following is proved using Proposition \ref{evanalysis} in an identical manner as in the proof of \cite[Lemma 1.7]{MPZ} .

\begin{lem}
\label{some convenient form}
Let $n,q\in\N$ with $q>n$ and $\ga\in\De_q$ with $\we(\ga) = 1/m_j$ for some $j\in\N$. Consider the functional $g:\mathcal{Y}_X\to\mathbb{R}$ with $g = e_\ga^*\circ P_{[1,n]}$. Then, one of the following holds:
\begin{itemize}
 \item[(i)] $g = 0$.
 \item[(ii)] There are $p_1\in\N$ with $p_1\leq n$ and $b^*$ in the unit ball of the space $((\sum_{k=1}^n\oplus X_k)_\mathrm{utc}^X\oplus (\sum_{k=1}^n\oplus\ell_\infty(\De_k))_\infty)_\infty^*$ so that $g = (1/m_j)b^*\circ P_{(p_1,n]}$.
 \item[(iii)] There are $p_1\in\N$ with $p_1\leq n$, $\ga'\in \De_{p_1}$, and $b^*$ as above so that $g = e_{\ga'}^* + (1/m_j)b^*\circ P_{(p_1,n]}$.
\end{itemize}

\end{lem}

\begin{prp}
\label{upper calkin}
Let $n\in\N$, $(a_k)_{k=0}^{n}$ be a sequence of scalars, and $T:\mathcal{Y}_X\to\mathcal{Y}_X$ be the bounded linear operator $T =a_0I + \sum_{k=1}^na_kI_k$. Then, there exists a compact operator $K:\mathcal{Y}_X\to\mathcal{Y}_X$ so that
\begin{equation*}
\left\|T - K\right\| \leq  (2C_0^2 - C_0)\left\|a_0 I_X + \sum_{k=1}^na_k e_k^*\otimes e_k\right\|_{\mathcal{L}(X)}.
\end{equation*}
\end{prp}

\begin{proof}
Let us set
$$\left\|a_0 I_X +\sum_{k=1}^na_k e_k^*\otimes e_k\right\|_{\mathcal{L}(X)} = M.$$
Set $a_k = 0$ for all $k>n$. This way, for all $x = \sum_{k=1}^\infty c_ke_k$ in $X$ we have $(a_0 I_X + \sum_{k=1}^na_k e_k^*\otimes e_k)x = \sum_{k=1}^\infty (a_0 + a_k)c_ke_k$. Let $A:\mathcal{Y}_X\to\mathcal{Y}_X$ be defined as follows: if $x = (x_k,y_k)_{k=1}^\infty$ then $Ax = i_n((0,y_k)_{k=1}^n)$. This is a finite rank operator and hence it is compact. We define $K = A\sum_{k=1}^na_kI_k$. To show that $K$ satisfies the conclusion let $x = (x_k,y_k)_{k=1}^\infty$ be an element in the unit ball of $\mathcal{Y}_X$. Note that
\begin{equation}
\label{trivial observations used in this proof}
 \left(\sum_{k=1}^na_kI_k - K\right)x = i_n(a_kx_k,0)_{k=1}^n \text{ and } a_0Ix = (a_0x_k,a_0y_k)_{k=1}^\infty
\end{equation}
This means that $(T-K)x = ((a_0+a_k)x_k,w_k)_{k=1}^\infty$, where $w_k = a_0y_k$, if $1\leq k\leq n$ and $w_k$ are some vectors in $\ell_\infty(\De_k)$ for $k>n$.

As $\|x\| \leq 1$ we conclude that for every $x^*\in\mathcal{G}^\mathrm{utc}$ we have $|x^*(x)|\leq 1$. This means that for all $i_0\in\N$ and for all $(b_k)_{k=1}^\infty$ for which $\|w^*\!\text{-}\!\sum_{k=1}^\infty b_ke_k^*\|_{X^*} \leq 1$ we have $|\sum_{k=1}^{i_0}b_kt_{k,i_0}^*(x_k)|\leq 1$, i.e. for all $i_0\in\N$ we have
\begin{equation}
\label{yeah yeah}
\left\|\sum_{k=1}^{i_0}t_{k,i_0}^*(x_k)e_k\right\|_X \leq 1.
\end{equation}
This yields that for any $(b_k)_{k=1}^\infty$ with $\|w^*\!\text{-}\!\sum_{k=1}^\infty b_ke_k^*\|_{X^*} \leq 1$ and every $i_0\in\N$ we have
\begin{equation}
\label{utc estimate}
\begin{split}
\left|\sum_{k=1}^{i_0}b_k(a_0 + a_k)t_{k,i_0}^*(x_k)\right| &\leq \left\|\sum_{k=1}^{i_0}(a_k+a_0)t_{k,i_0}^*(x_k)e_k\right\|_X\\
&\leq M\left\|\sum_{k=1}^{i_0}t_{k,i_0}^*(x_k)e_k\right\|_X\leq M \text{ by \eqref{yeah yeah}}.
\end{split}
\end{equation}
The above easily implies that for all $x^*\in\mathcal{G}^\mathrm{utc}$ we have $$\left|x^*\left(\left(T-K\right)x\right)\right| \leq M.$$
Similarly, for all $k\in\N$ and $x_k^*\in(1/A_0)B_{X_k^*}$ we have $|x_k^*(x_k)|\leq 1$, i.e.
\begin{equation}
\label{coordinatewise estimate}
\begin{split}
\left|x_k^*\left(\left(T-K\right)x\right)\right| &=  \left|x_k^*\left(\left(a_0+a_k\right)x_k\right)\right| \leq |a_0 + a_k|=\left\|\left(a_0 + a_k\right)e_k\right\|_X\\
&= \left\|\left(a_0 I_X + \sum_{k=1}^na_k e_k^*\otimes e_k\right)e_k\right\|_X\leq M.
\end{split}
\end{equation}
To obtain the desired conclusion it remains to show that for any $\ga\in\Ga$ we have
$$\left|e_\ga^*\left(\left(T-K\right)x\right)\right| \leq (2C_0^2 - C_0)M.$$
For $\ga\in\Ga$ with $\ra(\ga) = q \leq  n$ we have
\begin{equation*}
\left|e_\ga^*\left(\left(T-K\right)x\right)\right| = \left|e_\ga^*(a_0y_q)\right| \leq |a_0| = \left\|\left(a_0 I_X + \sum_{k=1}^na_k e_k^*\otimes e_k\right)e_{n+1}\right\|_X\leq M.
\end{equation*}
We now consider the case in which $\ga\in\Ga$ with $\ra(\ga) = q>n$. Set $u = (a_kx_k,0)_{k=1}^n\in \mathcal{Z}_X^n$. Arguments identical to that used in obtaining \eqref{utc estimate} and \eqref{coordinatewise estimate} yield that
\begin{equation*}
 \|u\| \leq \left\|\sum_{k=1}^na_ke_k^*\otimes e_k\right\|_{\mathcal{L}(X)} \leq 2M.
\end{equation*}
We observe that if $m\leq n$ then
\begin{equation*}
\left\|P_{(m,n]} i_n(u)\right\| = \left\|i_n(u) - i_mR_m(u)\right\| \leq (C_0+C_0A_0)\|u\| \leq 4C_0A_0M.
\end{equation*}
Also observe that for all $\ga'$ with $\ra(\ga') \leq n$ we have $e_\ga^*(i_n(u)) = 0$. We combine the above two facts with Lemma \ref{some convenient form} to obtain
$$|e_\ga^*(i_n(u))| \leq \beta_04C_0A_0M.$$
Finally,
\begin{equation*}
\begin{split}
\left|e_\ga^*\left(\left(T-K\right)x\right)\right|& = \left|e_\ga^*(a_0x) + e_\ga^*\left(i_n(u)\right)\right| \text{ (from \eqref{trivial observations used in this proof})}\\
& \leq |a_0| + 4\beta_0C_0A_0M \leq M + 4\beta_0C_0A_0M = (1+4\beta_0C_0A_0)M\\
&\leq C_0(1+4\beta_0A_0)M= C_0\left(C_0 - \frac{2\beta_0A_0}{1-\beta_0A_0} + 4\beta_0C_0A_0\right)M\\
&\leq C_0\left(C_0 + \frac{2\beta_0A_0}{1-\beta_0A_0}\right)M = C_0(C_0+C_0 - 1)M = (2C_0^2 - C_0)M.
\end{split}
\end{equation*}
\end{proof}

We are ready to prove the main theorem of this paper, before proceeding with the proof of Theorem \ref{pretty important part of this result} (otherwise known as Theorem \ref{diagonal plus compact are dense}). Note that, as it was pointed out in Remark \ref{almost isometric}, we can modify our construction so that $C_0$ is arbitrarily close to one.

\begin{thm}
\label{main result}
The space $\mathcal{C}al(\mathcal{Y}_X)$ is isomorphic, as a Banach algebra, to the space $\R I\oplus\mathcal{K}_\mathrm{diag}(X)$. Furthermore, the isomorphism $\Phi$ witnessing this fact satisfies $\|\Phi^{-1}\|\|\Phi\| \leq 2C_0^3 - C_0^2$.
\end{thm}

\begin{proof}
If we denote by $[I]$ and $[I_n]$ the equivalence class of $I$ and $I_n$ respectively in $\mathcal{C}al(\mathcal{Y}_X)$, then by Lemma \ref{lower calkin} and Proposition \ref{upper calkin} we have that for any sequence of scalars $(a_k)_{k=0}^n$ we have
\begin{equation*}
\begin{split}
\frac{1}{C_0}\left\|a_0I_X + \sum_{k=1}^n a_ke_k^*\otimes e_k\right\|_{\mathcal{L}(X)} &\leq \left\|a_0[I] + \sum_{k=1}^n a_k[I_k]\right\|\\
&\leq (2C_0^2-C_0)\left\|a_0I_X + \sum_{k=1}^n a_ke_k^*\otimes e_k\right\|_{\mathcal{L}(X)}.
\end{split}
\end{equation*}
That is, the space $\overline{\langle \{[I]\}\cup\{[I_n]: n\in\N\}\rangle}$ is $C_0(2C_0^2-C_0)$-isomorphic to $\R I\oplus\mathcal{K}_\mathrm{diag}(X)$. By Corollary \ref{they generate all of it} the space $\mathcal{C}al(\mathcal{Y}_X)$ is $(2C_0^3-C_0^2)$-isomorphic to $\R I\oplus\mathcal{K}_\mathrm{diag}(X)$.
\end{proof}

\section{Rapidly increasing sequences}
\label{section ris}
The notions of rapidly increasing sequences (RIS) and the basic inequality can by now be considered standard tools used in most HI, and related, constructions. The version of an RIS presented herein adds an extra condition that eliminates the influence of the $(\sum\oplus X_k)_X^\mathrm{utc}$ part in the definition of the space $\mathcal{Y}_X$ and keeps only the Bourgain-Delbaen part. Thus, RIS sequences can be treated identically as those in \cite{Z}. The ``utc'' condition used when defining $(\sum\oplus X_k)_X^\mathrm{utc}$ is designed so that sufficiently many RIS sequences with this extra condition can be found in the space. Recall that every element of $\mathcal{G}^\mathrm{utc}$ is naturally identified with a functional acting on $\mathcal{Z}_X^\infty$, and hence also on $\mathcal{Y}_X$.

\begin{ntt}
If $z\in\mathcal{Y}_X$ has finite BD-support, i.e. $\max\ran_\mathrm{BD}(z) = n$, then there is $u\in\mathcal{Z}_X^{n}$ with $z = i_n(u)$. If $u = (x_k,y_k)_{k=1}^n$ each vector $x_k$ is in the space $X_k$, which has a Schauder basis $(t_{k,i})_{i=1}^\infty$ and hence we may define the set $\supp_k(z) = \supp_{(t_{k,i})}(x_k)$ (see also \eqref{k-support}). If $\supp_k(z)$ is finite for $1\leq k\leq n$, then we say that $z$ has coordinate-wise finite supports and we may define the quantity
$$\max\supp_\mathrm{cw}(z) = \max\{\max\supp_k(z):1\leq k\leq n\}.$$
\end{ntt}

\begin{dfn}
\label{RIS}
Let $C>0$ and $(j_n)_n$ be a strictly increasing sequence of natural numbers. A block sequence (which may be either finite or infinite) $(z_n)_n$ of elements with coordinate-wise finite supports is called a $C$-rapidly increasing sequence (or a $C$-RIS) if
\begin{itemize}
 \item[(i)] $\|z_n\| \leq C$ for all $n$,
 \item[(ii)] $j_{n+1} > \max\ran_\mathrm{BD}(z_n)$,
 \item[(iii)] $|e_\ga^*(z_n)| \leq C/m_i$ for all $\ga\in\Ga$ with $\we(\ga) = 1/m_i$ and $i<j_n$, and
 \item[(iv)] for all $z^*\in\mathcal{G}^\mathrm{utc}$ there exists at most one $n$ for which $z^*(z_n)\neq 0$.
\end{itemize}
If we need to be specific about the sequence $(j_k)_k$ in the above definition we shall say that $(z_k)_k$ is a $(C,(j_k)_k)$-RIS.
\end{dfn}

\begin{lem}
\label{yeah well one could have included this in the definition of RIS but I guess that would have been just lazy}
Let $(z_n)_n$ be a sequence in $\mathcal{Y}_X$ that satisfies item (iv) of Definition \ref{RIS}. Then for every $z^*\in\mathcal{G}^{\mathrm{utc}}$ and $s\in\N$ there exists at most one $n\in\N$ so that $z^*(P_{(s,\infty)}z_n)\neq 0$.
\end{lem}

\begin{proof}
Let $s\in\N$ and $z^* = \sum_{k=1}^{i_0}a_kt^*_{k,i_0}$ with $(a_k)$ is such that $\|w^*-\sum_{k=1}^\infty a_ke_k^*\|_X \leq 1$. If $s>i_0$ it follows that $P_{(s,\infty)}^*z^* = 0$ and there is nothing to prove. Otherwise, $s\leq i_0$ and it follows that $P_{(s,\infty)}^*z^* = \sum_{k=s}^{i_0}a_kt^*_{k,i_0}$. Recall that the basis $(e_i)_i$ of the space $X$ has bimonotone constant $A_0$ which means that $\|\sum_{k=s}^{i_0}(a_k/A_0)e_k^*\|_X\leq 1$, i.e. $\tilde z^* = (1/A_0)P_{(s,\infty)}^*z^*$ is in $\mathcal{G}^\mathrm{utc}$. This easily implies the desired conclusion.
\end{proof}

\subsection{The basic inequality}
Let us denote by $(t_k^*)_k$ the unit vector basis of $c_{00}(\N)$. Given a sequence of natural numbers $(l_j)_j$ and a sequence of positive real numbers $(\theta_j)_j$ We define $W[(\mathcal{A}_{l_j},\theta_j)_j]$ to be the smallest subset $W$ of $c_{00}(\N)$ with the following properties:
\begin{itemize}
 \item[(1)] $\pm t_k^*\in W$ for all $k\in\N$,
 \item[(2)] for all $j\in\N$, $n\leq l_j$ and successive vectors $f_1, f_2,\ldots,f_n$ in $W$ the vector
 \begin{equation}
 \label{weighted mixed-T}
 f = \theta_j\sum_{k=1}^nf_k
 \end{equation}
 is in $W$ as well. 
\end{itemize}
The elements $\pm t_k^*$, $k\in\N$ will be referred to as type 0 elements of $W[(\mathcal{A}_{l_j},\theta_j)_j]$. If an element of $W[(\mathcal{A}_{l_j},\theta_j)_j]$ is as in \eqref{weighted mixed-T} then we say that it is of type 1 and it has weight $\theta_j$. We also use the notation $(t_k)_k$ for the unit vector basis of $c_{00}(\N)$ and for $f$ in $W[(\mathcal{A}_{l_j},\theta_j)_j]$ and $x\in c_{00}(\N)$ we define $f(x)$ to be the usual inner product $\langle f,x\rangle$ on $c_{00}(\N)$. The following can be found in \cite[Proposition 2.5]{AH} 

\begin{prp}
\label{auxiliary space estimate}
If $j_0\in\N$ and $f\in W[(\mathcal{A}_{4n_j},1/m_j)_j]$ is an element of weight $1/m_h$, then
\begin{equation}
\label{general estimate}
\left|f\left(\frac{1}{n_{j_0}}\sum_{l=1}^{n_{j_0}}t_l\right)\right|\leq\left\{
\begin{array}{ll}
2/(m_hm_{j_0}),&\text{if }h<j_0,\\
1/m_h,&\text{if }h\geq j_0.
\end{array}\right.
\end{equation}
If additionally $f\in W[(\mathcal{A}_{4n_j},1/m_j)_{j\neq j_0}]$ then
\begin{equation}
\label{more careful estimate}
\left|f\left(\frac{1}{n_{j_0}}\sum_{l=1}^{n_{j_0}}t_l\right)\right|\leq\left\{
\begin{array}{ll}
2/(m_h(m_{j_0})^2),&\text{if }h<j_0,\\
1/m_h,&\text{if }h> j_0.
\end{array}\right.
\end{equation}
\end{prp}

The proof of the following is practically identical to the proof of \cite[Proposition 5.4]{AH}. This is because on RIS functionals from $\mathcal{G}^\mathrm{utc}$ act in a $c_0$ way and hence they do not contribute to the norm of linear combinations of an RIS. We describe the proof for completeness.

\begin{prp}[Basic inequality]
\label{basic inequality}
Let $(z_k)_{k\in I}$ be a $C$-RIS, where $I$ is an interval of $\N$, let $(\la_k)_{k\in I}$ be real numbers, let $s$ be a natural number and let $\ga$ be an element of $\Ga$. Then there exist $k_0\in I$ and $g\in W[(\mathcal{A}_{3n_j},1/m_j)_j]$ such that
\begin{itemize}
 \item[(1)] either $g = 0$, or $\we(g) = \we(\ga)$ and $\supp(g)\subset\{k\in\ I:k>k_0\}$ and
 \item[(2)]
 \begin{equation*}
  \left|e_\ga^*\left(P_{(s,\infty)}\sum_{k\in I}\la_kz_k\right)\right| \leq (3A_0C_0C)|\la_{k_0}| + (4A_0C_0C)g\left(\sum_{k\in I}|\la_k|e_k\right).
 \end{equation*}
\end{itemize}
Moreover, if $j_0$ is such that
\begin{equation*}
\left|e_\xi^*\left(\sum_{k\in J}\la_kz_k\right)\right| \leq C\max_{k\in J}|\la_k|
\end{equation*}
for all subintervals $J$ of $I$ and all $\xi\in\Ga$ of weight $1/m_{j_0}$, then we may choose $g$ to be in $W[(\mathcal{A}_{3n_j},1/m_j)_{j\neq j_0}]$.
\end{prp}

\begin{proof}
This proof is along the lines of the proof of \cite[Proposition 5.4]{AH}. The difference is that here we also have to use property (iv) of Definition \ref{RIS}. Moreover some constants are different as, according to Remark \ref{Schauder decomposition} (iii), $\|P_n\|\leq A_0C_0$ for all $n\in\N$ whereas in \cite{AH} the same quantity is bounded by two. We describe the points where this difference takes place and refer the reader to \cite{AH} for the rest of the details. We shall only consider the statement without the additional assumption, as the modification required does not differ form that presented in \cite{AH}. The proof proceeds by induction on the rank of $\ga$. The case in which $\ra(\ga) = 1$ is fairly easy. Next, consider an element $\ga$ of rank greater than one, of age $a$, and of weight $1/m_h$. Let $(j_k)_k$ be the sequence for which $(z_k)_k$ is a $(C,(j_k)_k)$-RIS and assume that there is $l\in I$ for which $j_l\leq h < j_{l+1}$, as the other cases are simpler. Arguing identically as in \cite[Proposition 5.4]{AH} we obtain that for some $k_0\leq l$
\begin{equation}
 \label{under estimate}
 e_\ga^*\Big(P_{(s,\infty)}\sum_{\substack{k\in I\\ k\leq l}}\la_kz_k\Big) \leq (3A_0C_0C)|\la_{k_0}|.
\end{equation}
Set $I' = \{k\in I: k > l\}$ and let
$$e_\ga^* = \sum_{r=1}^ad_{\xi_r}^* + (1/m_h)\sum_{r=1}^ab_r^*\circ P_{(p_{r-1},p_r)}$$
be the evaluation analysis of $\ga$. Set
\begin{equation*}
\begin{split}
I_0' &= \{k\in I': \ran_\mathrm{BD}(z_k)\text{ contains $p_r$ for some }r\}\text{ and for }1\leq r\leq a\text{ set}\\
I'_r &= \{k\in I'\setminus I_0': \ran_\mathrm{BD}(z_k)\cap (p_{r-1},p_r)\neq \emptyset\}
\end{split}
\end{equation*}
Note that $\#I_0'\leq a$. Arguing identically as in \cite{AH} we obtain
\begin{equation*}
\begin{split}
e_\ga^*\left(P_{(s,\infty)}\sum_{k\in I'}\la_k z_k\right) &\leq (4A_0C_0C)(1/m_h^{-1})\sum_{k\in I_0'}|\la_k|\\
&+ (1/m_h)\left|\sum_{r=1}^ab_r^*\circ P_{(s\vee p_{r-1},\infty)}\sum_{k\in I'_r}\la_kz_k\right|.
\end{split}
\end{equation*}
Recall that each $b_r^*\in A_{p_r-1} = K_{p_r-1}\cup \mathcal{G}^\mathrm{utc}_{p_r-1}\cup B_{p_r-1}$. For $r$ such that $b_r\in B_{p_r-1}$ we have that $b_r^*$ is a convex combination of functionals $\pm e_\eta^*$, $\eta\in\Ga_{p_r-1}$. Hence, there exists $\eta_r$ with $\ra(\eta) < p_r$ so that
\begin{equation}
\label{convex combination dominated by singleton}
\left|\sum_{r=1}^ab_r^*\circ P_{(s\vee p_{r-1},\infty)}\sum_{k\in I'_r}\la_kz_k\right| \leq \left|\sum_{r=1}^ae_{\eta_r}^*\circ P_{(s\vee p_{r-1},\infty)}\sum_{k\in I'_r}\la_kz_k\right|
\end{equation}
Applying the inductive assumption to $\eta_r$, $(z_k)_{k\in I_r'}$ and $s\vee p_{r-1}$ we obtain $k_r\in I_r'$ and $g_r\in W[(\mathcal{A}_{3n_j},1/m_j)_j]$ with $\supp(g_r) \subset \{k\in I_r': k>k_r\}$ so that
\begin{equation}
\label{the usual case}
\left|e_{\eta_r}^*\circ P_{(s\vee p_{r-1},\infty)}\sum_{k\in I'_r}\la_kz_k\right| \leq (3A_0C_0C)|\la_{k_r}| + (4A_0C_0C)g_r\left(\sum_{k\in I'_r}|\la_k|e_k\right).
\end{equation}
For $r$ such that $b_r^*\in K_{p_r-1}$ we have that $\supp_\mathrm{BD}(b_r^*)$ is a singleton whereas if $r\in \mathcal{G}^\mathrm{utc}_{p_r-1}$ by Definition \ref{RIS} (iv) and Lemma \ref{yeah well one could have included this in the definition of RIS but I guess that would have been just lazy} there is at most one $k_r$ in $I'_r$ for which $b_r^*\circ P_{(s\vee p_{r-1},\infty)}z_{k_r} \neq 0$. In either case, there is $k_r\in I_r'$ so that
\begin{equation}
\label{like this is the only place that anything is different}
\left|b_r^*\circ P_{(s\vee p_{r-1},\infty)}\sum_{k\in I'_r}\la_kz_k\right| \leq |b_r^*\circ P_{(s\vee p_{r-1},\infty)}(\la_{k_r}z_{k_r})| \leq (2A_0C_0C)|\la_{k_r}|.
\end{equation}
We may for each such $r$ set $g_r = 0$. We combine  \eqref{convex combination dominated by singleton} and \eqref{the usual case} with \eqref{like this is the only place that anything is different} to obtain that for all $r$ there is $k_r\in I'_r$ and $g_r\in W[(\mathcal{A}_{3n_j},1/m_j)_j]$ $\supp(g_r) \subset \{k\in I_r': k>k_r\}$ so that
\begin{equation}
\label{(don't) over-estimate}
 \left|b_r^*\circ P_{(s\vee p_{r-1},\infty)}\sum_{k\in I'_r}\la_kz_k\right| \leq (3A_0C_0C)|\la_{k_r}| + (4A_0C_0C)g_r\left(\sum_{k\in I'_r}|\la_k|e_k\right).
\end{equation}
If we now define $g = (1/m_h)(\sum_{k\in I_0'}t_k^* + \sum_{r=1}^a(t_{k_r}^* + g_r))$ we conclude that $g\in  W[(\mathcal{A}_{3n_j},1/m_j)_j]$ with $\supp(g) \subset \{k\in I_r': k>k_0\}$. As in \cite{AH} We combine \eqref{under estimate} with \eqref{(don't) over-estimate} to obtain
\begin{equation*}
e_\ga^*\left(\sum_{k\in I}\la_kz_k\right) \leq (3A_0C_0C)|\la_{k_0}| + (4A_0C_0C)g\left(\sum_{k\in I}|\la_k|e_k\right).
\end{equation*}
\end{proof}

A combination of \eqref{general estimate} with Proposition \ref{basic inequality} directly yields the following estimate.

\begin{cor}
\label{brute RIS estimate}
Let $(z_k)_{k=1}^{n_j}$ be a $C$-RIS in $\mathcal{Y}_X$. Then
\begin{equation*}
\left\|\frac{1}{n_j}\sum_{k=1}^{n_j}z_k\right\| \leq A_0C_0C\left(\frac{3}{n_{j}} + \frac{4}{m_j}\right).
\end{equation*}
\end{cor}

Recall that the sequence $(m_j,n_j)_j$ is in fact of the form $(\tilde m_j,\tilde n_j)_{j\in L_0}$, i.e. it is a subsequence  of the sequence $(\tilde m_j,\tilde n_j)_j$ from page \pageref{mjnjproperties}.

\begin{cor}
\label{estimate for other norming sets}
Let $j_0\in\N\setminus L_0$ and $(z_k)_{k=1}^{\tilde n_{j_0}}$ be a $C$-RIS in $\mathcal{Y}_X$. Then
\begin{equation*}
\left\|\frac{\tilde m_{j_0}}{\tilde n_{j_0}}\sum_{k=1}^{\tilde n_{j_0}}z_k\right\| \leq \frac{11A_0C_0C}{\tilde m_{j_0}}.
\end{equation*}
\end{cor}

\begin{proof}
By Proposition \ref{basic inequality} there is $g\in W[(\mathcal{A}_{3n_j},1/m_j)_j]$ with
\[\left\|\frac{\tilde m_{j_0}}{\tilde n_{j_0}}\sum_{k=1}^{\tilde n_{j_0}}z_k\right\| \leq 3A_0C_0C\frac{\tilde m_{j_0}}{\tilde n_{j_0}} + 3A_0C_0Cg\left(\frac{\tilde m_{j_0}}{\tilde n_{j_0}}\sum_{k=1}^{\tilde n_{j_0}}e_k\right).\]
Note that $g\in W[(\mathcal{A}_{4\tilde n_j},1/\tilde m_j)_{j\neq j_0}]$ and although \eqref{more careful estimate} is not formulated in terms of this set it applies for it as well. This yields $g((\tilde m_{j_0}/\tilde n_{j_0})\sum_{k=1}^{\tilde n_{j_0}}e_k) \leq 2/\tilde m^2_{j_0}$ from which the desired estimate follows.
\end{proof}

\subsection{Rapidly increasing sequences and bounded linear operators}
On this subsection we prove that whether a bounded linear operator on $\mathcal{Y}_X$ is horizontally compact (see Definition \ref{def hor-co}) is witnessed on RIS.
\begin{ntt}
\label{extensions of stuff in range}
An important conclusion of \eqref{block vectors live in such things} is that if $z\in\mathcal{Y}_X$ and $\ran_\mathrm{BD}(z) = (m,n]$ then there is $u\in\mathcal{Z}_X^{n}$ with $z = i_n(u)$, so that if $u = (x_k,y_k)_{k=1}^n$, we have $x_k = 0$ and $y_k = 0$ for $1\leq k \leq m$.
We define
$$\supp^\Ga_\mathrm{loc}(z) =  \{i\in\N:\text{ there is }\ga\in\Ga_n\text{ with }\we(\ga)=1/m_i\text{ and }e_\ga^*(z)\neq0\}.$$
This set is called the $\Ga$-local support of $z$.
\end{ntt}
 
\begin{rmk}
\label{coordinate finite support perturbation}
For every $z\in\mathcal{Y}_X$ with finite BD-support and $\e>0$ there is $\tilde u$ with coordinate-wise finite supports, $\supp_\mathrm{BD}(u) = \supp_\mathrm{BD}(\tilde u)$, and $\|u-\tilde u\| < \e$. Indeed, if $ u = i_n(x_k,y_k)_{k=1}^n$ take for each $1\leq k\leq n$ a finitely supported vector $\tilde x_k$ in $X_k$ with $\|x_k - \tilde x_k\| \leq \e/(nC_0\sup_i\|t_{k,i}^*\|)$ (see Remark \ref{pesky detail}) and set $\tilde u = i_n(\tilde x_k,y_k)_{k=1}^n$. By \eqref{boundedness and compatibility} $\|i_n\|\leq C_0$ which yields the desired estimate.
\end{rmk}

\begin{lem}
\label{they basically all saitsfy (iv)}
Let $(z_n)_n$ be a block sequence with each $z_n$ having coordinate-wise finite supports. If for all $n\in\N$ we have $\max\supp_\mathrm{cw}(z_n) < \min\supp_\mathrm{BD}(z_{n+1})$ then $(z_n)_n$ satisfies item (iv) of Definition \ref{RIS}.
\end{lem}

\begin{proof}
If $z^*\in\mathcal{G}^\mathrm{utc}$ then $z^* = \sum_{k=1}^{i_0}a_kt_{k,i_0}^*$. If for some $n\in\N$ we have $z^*(z_n)\neq 0$, then $i_0\geq \min\supp_\mathrm{BD}(z_n)$ and $i_0 \leq \max\supp_\mathrm{cw}(z_n)$. That is,
\begin{equation*}
 \min\supp_\mathrm{BD}(z_n)\leq  i_0 \leq \max\supp_\mathrm{cw}(z_n),
\end{equation*}
which can be satisfied for at most one $n\in\N$.
\end{proof}

\begin{rmk}
\label{ris sum to ris}
Lemma \ref{they basically all saitsfy (iv)} easily implies that if $(z_k)_k$ and $(w_k)_k$ are both $C$-RIS then $(z_k+w_k)_k$ has a subsequence that is $2C$-RIS.
\end{rmk}

The following is essentially the same as \cite[Lemma 5.8]{AH} but for completeness we describe a proof.
 
\begin{lem}
\label{acts like ris on foreign weights}
Let $z$ be a vector for which the set $\supp_\mathrm{BD}(z)$ is finite. For every $\ga\in\Ga$ with $\we(\ga) = 1/m_i$ and $i\notin\supp_\mathrm{loc}^\Ga(z)$ we have $|e_\ga^*(z)| \leq 2C_0\|z\|/m_i$ (see \eqref{extension constant}). 
\end{lem}

\begin{proof}
We show this is by induction on $\ra(\ga)$. If $\ra(\ga) \leq q_0 = \max\supp_\mathrm{BD}(z)$ with $\we(\ga) = m_i$ and $i\notin\supp_\mathrm{loc}^\Ga(z)$ then $e_\ga^*(z) = 0$. Assume that the result holds for all $\ga$ with $\we(\ga) = 1/m_i$ and $i\notin\supp_\mathrm{loc}^\Ga(z)$ so that $\ra(\ga)\leq n$ for some $n\geq q_0$ and let $\ga\in\Ga$ with $\we(\ga) = 1/m_i$, $i\notin\supp_\mathrm{loc}^\Ga(z)$, and $\ra(\ga) = n+1$. Then $e_\ga^* = d_\ga^* + c_\ga^*$ (see page \pageref{basic properties of space}, Subsection \ref{basic properties of space}) and $d_\ga^*(z) = 0$, i.e. $e_\ga^*(z) = c_\ga^*$. Either $c_\ga^* = e_\xi^* + (1/m_i^*)b^*\circ P_{(p,n]}$, where $\we(\xi) = 1/m_i$ and $\ra(\xi) = p$ and $b^*\in A_n$ or $c_\ga^* = (1/m_i^*)b^*\circ P_{(0,n]}$, where $b^*\in A_n$. In the second case it trivially follows that $|c_\ga^*(z)| \leq (1/m_i)\|P_{(0,n]}\|\|z\|$. In the first case, if $p\leq n$ then $e_\xi^*(z) = 0$ and similarly $|c_\ga^*(z)| \leq (1/m_i)\|P_{(0,n]}\|\|z\|$. Otherwise, if $p>n$ then $P_{(p,n]}(z) = 0$ and the result follows from the inductive assumption.
\end{proof}

We shall say that a block sequence $(z_n)_n$ has bounded local weights if there exists $i_0$ so that for all $n\in\N$ and all $\ga\in\supp_\mathrm{loc}^\Ga(z_n)$ we have $\we(\ga)^{-1}\leq m_{i_0}$. We shall say that a block sequence $(z_n)_n$ has rapidly decreasing local weights if there exists sequence of natural numbers $(j_n)_n$ so that $\lim_nj_n = \infty$ and for all $n\in\N$ and $\ga\in\supp_\mathrm{loc}^\Ga(z_n)$ we have $\we(\ga) \leq 1/m_{j_n}$.

The next result is a combination of Lemma \ref{they basically all saitsfy (iv)} and \cite[Proposition 5.10]{AH}.
\begin{prp}
\label{some stuff is ris}
Let $(z_n)_n$ be a bounded block sequence in $\mathcal{Y}_X$ so that each $z_n$ that has finite coordinate-wise supports. If $(z_n)_n$ has bounded local weights or $(z_n)_n$ has rapidly decreasing local weights, then $(z_k)_k$ has a subsequence that is an RIS.
\end{prp}

\begin{proof}
We first pass to a subsequence so that the assumption of Lemma \ref{they basically all saitsfy (iv)} is satisfied. We refer to this sequence as $(z_n)_n$ as well. If $(z_n)_n$ has bounded local weights witnessed by $m_{i_0}$ set $C = \max\{m_{i_0},2C_0\}\sup_n\|z_n\|$. Then for all $\ga\in\Ga$ with $\we(\ga) = 1/m_i$ and $i\leq i_0$ and for all $n\in\N$ we have $|e_\ga^*(z_n)| \leq \|z_n\| \leq (m_{i_0}/m_{i})\|z_n\| \leq C/m_i$. On the other hand, for $\ga\in\Ga$ with $\we(\ga) = 1/m_i$ and $i > i_0$ by Lemma \ref{acts like ris on foreign weights} we obtain $|e_\ga^*(z_n)| \leq 2C_0\|z\|/m_i\leq C/m_i$. If we define $j_n = \min\ran_\mathrm{BD}(z_n)$ then all assumptions of Definition \ref{RIS} are satisfied. 
If on the other hand the sequence $(z_n)_n$ has rapidly increasing weights witnessed by $(j_n)_n$, pass to common subsequences of $(z_n)_n$  and $(j_n)_n$, again denoted by $(z_n)_n$  and $(j_n)_n$, so that $j_{n+1}>\max\ran_\mathrm{BD}(z_n)$ for all $n\in\N$. If $C = 2C_0\sup_n\|z_n\|$ then Lemma \ref{acts like ris on foreign weights} yields that $(z_n)_n$ is a $(C,(j_n)_n)$-RIS.
\end{proof}

The following is almost identical to the proof of \cite[Proposition 5.11]{AH}. We describe a proof for completeness.

\begin{prp}
\label{horizontal compactness is witnessed on ris}
Let $Y$ be a Banach space and $T:\mathcal{Y}_X\to Y$ be a bounded linear operator. If for all RIS $(z_n)_n$ in $\mathcal{Y}_X$ we have $\lim_n\|Tz_n\| = 0$, then for every bounded block sequence $(z_n)_n$ we have $\lim_n\|Tz_n\| = 0$.
\end{prp}

\begin{proof}
We will start with an arbitrary block sequence $(z_n)_n$ and show that it has a subsequence that satisfies the conclusion. By Remark \ref{coordinate finite support perturbation} we may perturb the sequence so that all of its elements have finite coordinate-wise supports. For each $n\in\N$ there is $q_n\in\N$ and $(x_{n,k},y_{n,k})_{k=1}^{q_n}$ so that $z_n = i_{q_n}(x_{n,k},y_{n,k})_{k=1}^{q_n}$. Define for each $n,N\in\N$ and $k\leq q_n$ the following vectors in $\ell_\infty(\De_k)$
\begin{equation*}
\begin{split}
v_{n,k}^N(\ga) &=\left\{
\begin{array}{ll}
y_k(\ga),&\text{if }\we(\ga) \geq 1/m_N,\\
0,&\text{otherwise}
\end{array}
\right.\;\text{ and }\\
w_{n,k}^N(\ga) &=\left\{
\begin{array}{ll}
y_k(\ga),&\text{if }\we(\ga) < 1/m_N,\\
0,&\text{otherwise},
\end{array}
\right.
\end{split}
\end{equation*}
and for all $n,N\in\N$ define the vectors $$\chi_n^N = i_{q_n}(x_{n,k},v_{n,k}^N)_{k=1}^{q_n}\text{ and }\psi_n^N = i_{q_n}(0,w_{n,k}^N)_{k=1}^{q_n}.$$
For each $n,N\in\N$ we have $z_n = \chi_n^N + \psi_n^N$ and $\|\chi_n^N\| \leq \|i_{q_n}\| \|(x_k,v_k^N)_{k=1}^{q_n}\| \leq C_0 \|R_{q_n}z_n\| \leq C_0A_0\|z_n\|$. For fixed $N\in\N$ the sequence $(\chi_n^N)_n$ has finite coordinate-wise supports and bounded weights, i.e., by Proposition \ref{some stuff is ris} any of its subsequences has a further subsequence that is an RIS. By assumption, $\lim_n\|T\chi_n^N\| = 0$. We may therefore choose an increasing sequence of indices $(n_N)_N$ so that $\lim_N\|T\chi_{n_N}^N\| = 0$. Consider now the sequence $(\psi_{n_N}^N)_N$ which is bounded, it has finite coordinate-wise supports, and it has rapidly decreasing local weights. By Proposition \ref{some stuff is ris} any of its subsequences has a further subsequence that is an RIS and by assumption, $\lim_N\|T\psi_{n_N}^N\| = 0$. As $z_{n_N} = \chi_{n_N}^N + \psi_{n_N}^N$ we have $\lim_N\|Tz_{n_N}\| = 0$. 
\end{proof}

\begin{cor}
\label{shrinking}
The Schauder decomposition $(Z_n)_n$ of $\mathcal{Y}_X$ is shrinking. 
\end{cor}

\begin{proof}
If it were not shrinking then there would be a functional $x^*\in\mathcal{Y}_X^*$ and a normalized block sequence $(x_k)_k$ with $\lim\inf |x^*(x_k)|  > 0$. By Proposition \ref{horizontal compactness is witnessed on ris} there would exist a $C$-RIS $(z_k)_k$ and $\e>0$ with $x^*(z_k)>\e$ for all $k\in\N$. By Corollary \ref{brute RIS estimate} we would have
\begin{equation*}
\e < \left\|\frac{1}{n_j}\sum_{k=1}^{n_j}z_k\right\| \leq A_0C_0C\left(\frac{3}{n_{j}} + \frac{4}{m_j}\right)
\end{equation*}
which is absurd for $j$ sufficiently large.
\end{proof}

\section{The scalar-plus-horizontally compact property}
\label{section s-p-hc}
in this section we prove one of the most important features of this construction namely that every operator on $\mathcal{Y}_X$ is a scalar multiple of the identity plus a horizontally compact operator. The following was introduced in \cite[Definition 7.1]{Z}.
\begin{dfn}
\label{def hor-co}
Let $X$ be a Banach space with a Schauder decomposition $(Y_n)_n$. An operator $T:X\to X$ is called horizontally compact (with respect to $(Y_n)_n$) if for every bounded block sequence $(x_k)_k$ we have $\lim_k\|Tx_k\| = 0$.
\end{dfn} 

A standard argument yields that a $T:X\to X$, where $X$ has a Schauder decomposition with associated projections $(P_n)_n$, is horizontally compact precisely when $\lim_n \|TP_{(n,\infty)}\| = 0$ or equivalently $T = \lim_n TP_n$ in operator norm. If one furthermore assumes that the Schauder decomposition is shrinking, then $T$ is horizontally compact if and only if its restriction on any block subspace is compact.

 \subsection{Exact pairs and exact sequences}
The definition of exact pairs and exact sequences is based on that from \cite[Definition 6.1]{AH}. Some modification is made to take into consideration the set $\mathcal{G}^\mathrm{utc}$. Exact sequences are a delicate tool necessary to prove properties about operators in $\mathcal{Y}_X$. Similar to \cite[Theorem 1.1 (2)]{Z}, one can use these tools to shown that block sequences in $\mathcal{Y}_X$ span HI spaces. We do not prove this result as we do not use it.

\begin{dfn}
\label{def exact pair}
Let $C>0$ and $\e\in\{0,1\}$. A pair $(z,\ga)$ where $z\in\mathcal{Y}_X$ and $\ga\in\Ga$ is said to be a $(C,j,M,\e)$-exact pair if the following are satisfied:
\begin{itemize}

\item[(i)] $|d_\xi^*(z)| \leq C/m_j$ for all $\xi\in\Ga$,

\item[(ii)] $\we(\ga) = 1/m_j$,

\item[(iii)] $\|z\| \leq C$ and $e_\ga^*(z) = \e$,

\item[(iv)] for every element $\xi\in \Ga$ with $\we(\xi) \neq m_j$ we have
\begin{equation*}
|e_\xi^*(z)| \leq \left\{\begin{array}{cc}
C/m_i,&\text{if }i<j,\\
C/m_j,&\text{if }i>j.
\end{array}\right.
\end{equation*}

\item[(v)]  $\supp_\mathrm{BD}(z)$ is finite and $\max\supp_\mathrm{cw}(z) \leq M$.

\end{itemize}
\end{dfn}

The following is very similar to \cite[Lemma 6.2]{AH}. We include a proof for the sake of completeness.

\begin{lem}
\label{adding ris to exact pair}
Let $(z_k)_{k=1}^{n_{2j}}$ be $C$-RIS and assume that there are natural numbers $0 = q_0 <2j \leq q_1<\cdots<q_k$ so that $\supp_\mathrm{BD}(z_k) \subset (q_{k-1},q_k)$ and that there are $b_k\in A_{q_{k-1}}$ with $b_k(x_k) = 0$ for $k=1,\ldots,n_{2j}$. Then there exist $\zeta_k\in\De_{q_k}$ so that if $\gamma = \zeta_{n_{2j}}$, $M = \max_{1\leq k\leq n_{2j}}\max\supp_\mathrm{cw}(z_k)$, and
\begin{equation*}
z = \frac{m_{2j}}{n_{2j}}\sum_{k=1}^{n_{2j}}z_k
\end{equation*}
then $(z,\gamma)$ is a $(7C,2j,M,0)$-exact pair. Furthermore, the evaluation analysis of $\gamma$ is $(q_k,b_k^*,\zeta_k)_{k=1}^{n_{2j}}$.
\end{lem}

\begin{proof}
The existence of $\ga$ with the desired properties is a consequence of Lemma \ref{constructing gammas} whereas \eqref{eanalysis eq} yields $e_\ga^*(z) = 0$. For every $\xi\in\Ga$ the functional $d^*_\xi$ acts on at most one $z_k$ so we have $|d_\xi^*(z)| \leq \|d_\xi^*\|Cm_{2j}/n_{2j} \leq CC_0/m_{2j}$ so (i) holds. Corollary \ref{brute RIS estimate} yields $\|z\| \leq 5A_0C_0C$, so (iii) holds. Also combining Proposition \ref{basic inequality} with \eqref{general estimate} we deduce that  for every element $\xi\in \Ga$ with $\we(\xi) \neq m_j$ we have
\begin{equation*}
|e_\xi^*(z)| \leq \left\{\begin{array}{cc}
7A_0C_0C/m_i,&\text{if }i<j,\\
7A_0C_0C/m_j,&\text{if }i>j.
\end{array}\right.
\end{equation*}
Additionally, (v) clearly holds from the choice of $M$.
\end{proof}

\begin{dfn}
A sequence $(z_k)_{k=1}^{n_{2j_0-1}}$ is called a $(C,2j_0-1,\e)$-dependent sequence if there exist natural numbers $0 = p_0 < p_1<\cdots<p_{n_{2j_0-1}}$, and elements $\eta_k$, $\xi_k$ in $\Ga$ for $k=1,\ldots,n_{2j_0-2}$ so that
\begin{itemize}


\item[(i)] $(z_1,\eta_1)$ is a $(C,4j_1-2,p_1,\e)$-exact pair and $(z_k,\eta_k)$ is a $(C,4j_k,p_k,\e)$-exact pair for $k = 2,\ldots,n_{2j_0-1}$,

\item[(ii)] the element $\ga = \xi_{n_{2j_0-1}}$ has $\we(\ga) = 1/m_{2j_0-1}$ and evaluation analysis $(p_k,e_{\eta_k}^*,\xi_k)_{k=1}^{n_{2j_0-1}}$, and

\item[(iii)] $\ran_\mathrm{BD}(z_k) \subset (p_{k-1}, p_k)$, for $k=1,\ldots,n_{2j_0-1}$.

\end{itemize}
\end{dfn}

\begin{rmk}
\label{dependent remark}
If $(z_k)_{k=1}^{n_{2j_0-1}}$ is a $(C,2j_0-1,\e)$-dependent sequence then $e_\ga^*(z_k) = \e/m_{2j_0 - 1}$ for $1\leq k\leq n_{2j_0-1}$ (where $\ga = \xi_{n_{2j_0-1}}$). In addition, by the definition of functionals of odd weight for the associated  components we can deduce that $p_{k-1} < \ra(\eta_k) < p_k$, $\ra(\xi_k) = p_k$, and for $2\leq k \leq n_{2j_0-1}$ $\we(\xi_k) = 1/m_{2j_0-1}$, $\we(\eta_1) = 1/m_{4j_1-2} < 1/n^2_{2j_0-1}$ and for $2\leq k \leq n_{2j_0-1}$ $\we(\eta_{k}) = 1/m_{4j_{k}} = 1/m_{4\sigma(\xi_{k-1})}$. \end{rmk}
 
The following is an easy consequence of the definition of an exact pair, the growth condition of the function $\sigma$, as well as Lemma \ref{they basically all saitsfy (iv)}. A short proof can be found in \cite[Lemma 6.4]{AH}.
 
\begin{lem}
\label{dependent is ris}
A $(C,2j_0-1,\e)$-dependent sequence $(z_k)_{k=1}^{n_{2j_0-1}}$ is a $C$-RIS.
\end{lem}
 
The following requires some calculations however its proof is entirely identical to \cite[Lemma 6.5]{AH} so we omit it.

\begin{lem}
\label{dependent die on same weight}
Let $(z_k)_{k=1}^{n_{2j_0-1}}$ be a $(C,2j_0-1,0)$-dependent sequence, $J$ be a subinterval of $\{1,\ldots,n_{2j_0-1}\}$, and $\zeta\in\Ga$ with $\we(\zeta) = 1/m_{2j_0-1}$. Then we have
\begin{equation*}
\left|e_\zeta^*\left(\sum_{k\in J}z_k\right)\right| \leq 3C.
\end{equation*} 
\end{lem}

The following is only a minor modification of \cite[Proposition 6.6]{AH}. We include a proof for the sake of completeness.
 
\begin{prp}
\label{zero dependent estimate}
Let $(z_k)_{k=1}^{n_{2j_0-1}}$ be a $(C,2j_0-1,0)$-dependent sequence. Then we have
\begin{equation*}
\left\|\frac{1}{n_{2j_0-1}}\sum_{k=1}^{n_{2j_0-1}}z_k\right\| \leq 33A_0C_0C\frac{1}{m^2_{2j_0-1}}.
\end{equation*}
\end{prp}
 
 \begin{proof}
Set $z = 1/n_{2j_0-1}\sum_{k=1}^{n_{2j_0-1}}z_k$. We will use \eqref{norm triad}. By condition (iv) of Definition \ref{RIS} every $x^*\in\mathcal{G}^\mathrm{utc}$ acts on at most one $z_k$ so we have $|x^*(z)| \leq C/n_{2j_0-1}$. The same holds for $x^*\in\cup_k(1/A_0)B_{X_k^*}$. For $\ga\in\Ga$ with $\we(\ga) = 1/m_{2j_0-1}$ by Lemma \ref{dependent die on same weight} we have $|e_\ga^*(z)|\leq 3C/n_{2j_0-1}$. Fix $\ga\in\Ga$ with $\we(\ga)\neq 1/m_{2j_0-1}$. We will use the full statement of Proposition \ref{RIS}. By Lemmas \ref{dependent is ris} and \ref{dependent die on same weight} it follows that the sequence $(z_k)_{k=1}^{n_{2j_0-1}}$ satisfies the additional assumption of \ref{RIS} for $3C$ and $2j_0-1$. This means that there exists $g\in W[(\mathcal{A}_{3n_j},1/m_j)_{j\neq 2j_0-1}]$ with
\[|e_\ga^*(z)| \leq \frac{9A_0C_0C}{n_{2j_0-1}} + 12A_0C_0Cg\left(\frac{1}{n_{2j_0-1}}\sum_{k=1}^{n_{2j_0-1}}e_k\right).\]
We conclude the desired estimate by applying \eqref{more careful estimate}.
\end{proof}

\subsection{Scalar-plus-horizontally compact}

As the technical tool of dependent sequences has been discussed we may now proceed to proving that every operator on $\mathcal{Y}_X$ is a scalar operator plus a horizontally compact operator.

\begin{prp}
\label{blocks vanish on initial parts}
Let $T:\mathcal{Y}_X\to\mathcal{Y}_X$ be a bounded linear operator. Then for every $N$ in $\N$ and every bounded block sequence $(z_k)_k$ we have $\lim_k\|P_NTz_k\| = 0$.
\end{prp}

\begin{proof}
By Corollary \ref{shrinking} every bounded block sequence is weakly null. Also, $\sum_{k=1}^N\oplus Z_k\simeq (\sum_{k=1}^N\oplus X_k)\oplus\ell_\infty(\Gamma_N)$. Therefore, is sufficient to check that for every $k_0$, every bounded block sequence $(z_k)_k$, and every bounded linear operator $S:\mathcal{Y}_X\to X_{k_0}$ we have $\lim_k\|Sz_k\| = 0$. By Proposition \ref{horizontal compactness is witnessed on ris} is it sufficient to check this for a $C$-RIS $(z_k)_k$. If the conclusion were false, then for some $\de > 0$ $(Sz_k)_k$ would be equivalent to abounded block sequence $(w_k)_k$ in $X_{k_0}$ with $\|w_k\| > \de$ for all $k\in\N$, that is there would exist a constant $M$ so that for all scalars $(a_k)_k$ we would have
\[\left\|\sum_ka_kw_k\right\| \leq M\left\|\sum_ka_kz_k\right\|.\]
By \cite[Proposition 4.8]{AH} for any $\ell_{2j}\in L_{k_0} = \{\ell_1 < \ell_2<\cdots\}$ (see page \pageref{mjnjproperties}) we would have
$$\left\|\frac{\tilde m_{\ell_{2j}}}{\tilde n_{\ell_{2j}}}\sum_{k=1}^{ \tilde n_{\ell_{2j}}}w_k\right\| \geq \de,$$
whereas by Corollary \ref{estimate for other norming sets} we would have
$$\left\|\frac{\tilde m_{\ell_{2j}}}{\tilde n_{\ell_{2j}}}\sum_{k=1}^{ \tilde n_{\ell_{2j}}}z_k\right\| \leq \frac{11A_0C_0C}{\tilde m_{\ell_{2j}}}.$$
This would mean $\tilde m_{\ell_{2j}} \leq 11A_0C_0CM/\de$ for arbitrary $j$, which is absurd.
\end{proof}

\begin{lem}
\label{Gamma is a norming set}
For every $z\in\mathcal{Y}_X$, $N\in\N$, and $\e>0$ there exists $\ga\in\Ga$ with $\ra(\ga)\geq N$ so that
$$|e_\ga^*(z)| \geq \frac{1-\e}{m_2} \|z\|.$$
\end{lem}

\begin{proof}
Approximate $z$ by a vector with finite BD-support $\tilde z$ so that $\|z - \tilde z\| \leq \e/(2m_2)$. By remark \ref{how things are normed} we may choose a natural number $M$ so that $(1-\e/2)\|\tilde z\| \leq \max\{|b^*(w)|: b^*\in A_M\}$ (for the definition of $A_M$ see page \pageref{even0}, paragraph before \eqref{even0}). Fix $b^*_0$ achieving this maximum. Then for every $n\geq\max\{M,N-1,\max\supp_\mathrm{BD}(\tilde z)\}$ and $b^*\in A_M\subset A_n$ we have that the triple $\ga = (n+1,1/m_2,b_0^*)$ is in $\De_{n+1}^{\mathrm{Even}_0}$. It follows that $|e_\ga^*(\tilde z)| = |c_\ga^*(\tilde z)| = (1/m_2)|b_0^*(\tilde z)| \geq (1-\e/2)/m_2\|\tilde z\|$. Hence, $|e^*_\ga(z)| \geq ((1-\e)/m_2)\|z\|$.
\end{proof}

Lemma \ref{perturb the distance} is an alternative approach to what is presented in \cite[Section 7]{Z}, where an element $x$ is approached by another element with the property that the action of every functional in $\cup_nA_n$ yields a rational number. A modification of the approach from \cite{Z} would work here as well, however the factor of $1/m_2$ in item (iii) bellow would not be avoided. The reason is that the construction presented here is designed to yield a Calkin Algebra that is $(1+\e)$-isomorphic to $\R I\oplus\mathcal{K}_\mathrm{diag}(X)$.

 \begin{lem}
 \label{perturb the distance}
 Let $x$ be in $\mathcal{Y}_X$, $T:\mathcal{Y}_X\to \mathcal{Y}_X$ be a bounded linear operator and $\de = \mathrm{dist}(Tx,\R x)$. Then for every $\e>0$ and $N\in\N$ there exist $b^*\in\cup_n B_n$ (see page \pageref{rational convex combinations}), $\ga_0\in\Ga$ and $0<\theta<\e$ so that if $\tilde x = x - \theta d_{\ga_0}$
\begin{itemize}

\item[(i)] $\ra(\ga_0) \geq N$,


\item[(ii)] $b^*(\tilde x) = 0$, and

\item[(iii)] $b^*(T\tilde x) \geq \de/(5m_2)$.

\end{itemize}
 \end{lem}

\begin{proof}
Use the Hahn-Banach Theorem to find a functional $x^*\in\mathcal{Y}_X^*$ with $\|x^*\| = 1$, $x^*(Tx) = \de$ and $x^*(x) = 0$. Lemma \ref{Gamma is a norming set} and a separation theorem imply that $\overline{\mathrm{conv}}^{w^*}\{\pm m_2e_\ga^*:\ga\in\Ga\}$ contains the unit ball of $\mathcal{Y}_X^*$. This means that there are a finite set $E$ and  $f_0 = \sum_{\ga\in E}c_\ga e_\ga^*$ with $f_0(Tx) = x^*(Tx) = \de$ and $f_0(x) = x^*(x) = 0$ so that $\sum_{\ga\in E}|c_\ga| \leq 2m_2$. Choose collections of rational numbers $(c^n_\ga)_{\ga\in E}$ with $\lim_n c_\ga^n = c_\ga$ for all $\ga\in E$.

Fix $\ga_0\in \Ga$ with $\ra(\ga_0) > \max\{N, \max_{\ga\in E}\ra(\ga)\}$ and $\mathrm{dist}(d_{\ga_0},\langle \{x,Tx\}\rangle) = \eta >0$. Arguing as before find another finite subset $F$ of $\Ga$ and $g_0 = \sum_{\ga\in F}a_\ga e_\ga^*$ with $g_0(d_{\ga_0}) = \eta$, $g_0(x) = g_0(Tx) = 0$ so that $\sum_{\ga\in E}|a_\ga| \leq 2m_2$. Choose collections of rational numbers $(a^n_\ga)_{\ga\in F}$ with $\lim_n a_\ga^n = a_\ga$ for all $\ga\in F$.

Define for each $n\in\N$
$$f_n = \frac{1}{2}\left(\frac{1}{\sum_{\ga\in E}|c_\ga^n|}\sum_{\ga\in E}c_\ga^ne_\ga^* + \frac{1}{\sum_{\ga\in F}|a_\ga^n|}\sum_{\ga\in F}a_\ga^ne_\ga^*\right) \text{ and } x_n = x - \frac{f_n(x)}{f_n(d_{\ga_0})}d_{\ga_0}.$$
Observe that for each $n\in\N$ we have $f_n\in\cup_m B_m$, that $f_n(x_n) = 0$ and that $\lim_n f_n(Tx) = \de/(2\sum_{\ga\in E}|c_\ga|)\geq \de/(4m_2)$. Furthermore, observe that $f_n(d_{\ga_0}) = \sum_{\ga\in F}a_\ga^ne_\ga^*(d_{\ga_0})/(2\sum_{\ga\in F}|a^n_\ga|)$ which yields that $\lim_n f_n(d_{\ga_0}) = \eta/(2\sum_{\ga\in F}|a_\ga|) \geq\eta/(4m_2)$. The proof is concluded by setting $b^* = f_n$ and $\theta = f_n(x)/f_n(d_{\ga_0})$ for $n$ sufficiently large.
\end{proof}

\begin{rmk}
\label{perturbation yields RIS}
 Let $(z_k)_k$ be a $(C,(j_k)_k)$-RIS, $(\ga_k)_k$ be a sequence in $\Ga$ so that $\ra(\ga_k)_k$ increases to infinity, and $(\theta_k)_k$ be real numbers with $0<\theta_k< C/m_{j_k}$ for all $k\in\N$. Then the sequence $(\tilde z_k)_k = (z_k - \theta_kd_{\ga_k})_k$ has a subsequence that is a $2C$-RIS. Indeed, conditions (i), (ii), and (iii) from Definition \ref{RIS} are straightforwardly satisfied by a suitable subsequence. Condition (iv) is trivial once it is observed that for all $\ga\in\Ga$ and all $x^*\in\mathcal{G}^\mathrm{utc}$ we have $x^*(d_\ga) = 0$. 
 \end{rmk}

 \begin{prp}
 \label{distances vanish}
 Let $T:\mathcal{Y}_X\to\mathcal{Y}_X$ be a bounded linear operator. Then for every infinite RIS $(z_k)_k$ we have $\lim_k\mathrm{dist}(Tz_k,\R z_k) = 0$.
 \end{prp}
 
 \begin{proof}
 Assume that there is a $(C,(j_k)_k)$-RIS $(z_k)_k$ and $\de > 0$ so that for all $k\in\N$ we have $\mathrm{dist}(Tz_k,\R z_k) > \de$. We will show that this would mean that $T$ is unbounded.
 
 We shall first prove the following claim. For every $j$ and $N\in\N$ there exists an $(14C,2j,M,0)$-exact pair $(z,\ga)$ so that
 \begin{equation*}
 e_\ga^*(Tz) \geq \de/(6m_2) \text{ and } \min\supp_\mathrm{BD}(z) > N.
 \end{equation*}


 For every $k\in\N$ we apply Lemma \ref{perturb the distance} to find a sequence $(\tilde z_k)_k = (z_k - \theta_k d_{\ga_k})_k$, with $\ra(\ga_k)$ increasing to infinity and $0<\theta_k<C/m_{j_k}$, and $(b_k^*)_k$ in $A_{N_k}$, for some $N_k$, so that $b_k^*(\tilde z_k) = 0$ and $b_k^*(T\tilde z_k) \geq \de/(5m_2)$. By Remark \ref{perturbation yields RIS} we may assume that $(\tilde z_k)_k$ is $2C$-RIS. Define $p_0 = 0$ and $p_k = \max\{N_k,\max\ran_\mathrm{BD}(\tilde z_k)+1\}$. By passing, if it is necessary, to a subsequence we may assume $\ran_\mathrm{BD}(z_k) \subset (p_{k-1},p_k)$ for all $k\in\N$. Utilizing the weakly null property of $(T\tilde z_k)_k$ we may assume that $\sum_k\|T\tilde z_k - P_{(p_{k-1},p_k)}\tilde z_k\| <\eta$, where $\eta$ is a positive number to be determined later. We may assume that $\min\supp_\mathrm{BD}(\tilde z_1) > N$ and that $q_1>2j$. By Lemma \ref{adding ris to exact pair}, if $z = (m_{2j}/n_{2j})\sum_{k=1}^{n_{2j}}\tilde z_k$ and there are $(\zeta_k)_{k=1}^{n_{2j}}$ with $\ra(\zeta_k) = p_k$ and $\ga\in\Ga$ with $e_\ga^* = 1/m_{2j}\sum_{k=1}^{n_{2j}}b_k^*\circ P_{(p_{k-1},p_k)} + \sum_{k=1}^{n_{2j}}d_{\zeta_k}^*$ so that $(z,\ga)$ is a $(14C, 2j, M,0)$-exact pair. We calculate
 \begin{equation*}
 \begin{split}
 e_\ga^*(T z) &= \frac{1}{n_{2j}}\sum_{k=1}^{n_{2j}}b_k^*\circ P_{(p_{k-1},p_k)}(T \tilde z_k) + \sum_{k=1}^{n_{2j}}d_{\zeta_k}^*(T z)\\
 &=\frac{1}{n_{2j}}\sum_{k=1}^{n_{2j}}b_k^*(T\tilde z_k) - \frac{1}{n_{2j}}\sum_{k=1}^{n_{2j}}b_k^* (T\tilde z_k - P_{(p_{k-1},p_k)}T\tilde z_k) + \sum_{k=1}^{n_{2j}}e_{\zeta_k}^*\circ P_{\{q_k\}}(T z)\\
 &\geq \frac{\de}{5m_2} - \frac{1}{n_{2j}}\sum_{k=1}^{n_{2j}}\|T\tilde z_k - P_{(p_{k-1},p_k)}T\tilde z_k\| - \sum_{k=1}^{n_{2j}}\|P_{\{p_k\}}Tz\|\\
 &\geq  \frac{\de}{5m_2}  - \frac{\eta}{n_{2j}} - \frac{m_{2j}}{n_{2j}}\sum_{k=1}^{n_{2j}}\sum_{l=1}^{n_{2j}}\|P_{\{p_k\}}T\tilde z_m\|\\
 & \geq \frac{\de}{5m_2}  - \frac{\eta}{n_{2j}} - (2A_0C_0)\frac{m_{2j}}{n_{2_j}}\sum_{k=1}^{n_{2j}}\sum_{l=1}^{n_{2j}}\|T \tilde z_m - P_{(p_{k-1},p_k)}T\tilde z_m\|\\
 & \geq \frac{\de}{5m_2}  - \frac{\eta}{n_{2j}} - (2A_0C_0)m_{2j}\eta \geq \frac{\de}{6m_2},
 \end{split}
 \end{equation*}
 for $\eta>0$ sufficiently small.
 
 We now use the claim to construct for any given $j_0\in\N$ a $(14C,2j_0-1,0)$-dependent sequence, with associated components $0=p_1<\cdots p_{n_{2j_0}-1}$, $(\eta_k)_{k=1}^{n_{2j_0}-1}$, and $(\xi_k)_{k=1}^{n_{2j_0}-1}$ so that
 \begin{equation*}
 e_{\xi_k}^*(Tz_k) > \de/(6m_2)\text{ and }\|Tz_k - P_{(p_{k-1},p_k)}Tz_k\| < \e/n_{2j_0-1},
 \end{equation*}
 where $0<\e<\de/(84m_2m_{2j_0-1})$, for $1 \leq k \leq n_{2j_0-1}$. We start by choosing a $(14C, 4j_1-2,M,0)$-exact pair $(z_1,\eta_1)$, for $j_1$ with $m_{4i_1-2}>n_{2j_0-1}^2$, satisfying the conclusion of the claim and we also choose $p_1$ sufficiently large so that $\max\{M, \ra(\eta_1),\max\supp_{BD}(z_1)\} < p_1$ as well as $\|Tz_1 - P_{(0,p_1)}Tz_1\| < \e/n_{2j_0-1}$. Clearly, $(z_1,\eta_1)$ is also a $(14C, 4j_1-2,p_1,0)$-exact pair. Set $\xi_1 = (p_1,1/m_{2j_0-1},\eta_1)$ which is in $\De_{p_1}^{\mathrm{Odd}_0}$. If we have chosen $(z_k,\eta_k)$, $\xi_k$, $p_k$ for $1\leq k\leq a < n_{2j_0-1}$ set $j_{a+1} = \sigma(\xi_k)$ and apply the claim to find a sequence of $(14C, 2j_{a+1},M_n,0)$-exact pairs $(z^{a+1}_n,\eta_n^{a+1})_{n\in\N}$ with $p_a <\min\supp_\mathrm{BD}(z^{a+1}_n) \to  \infty$. By the weak null property of $(z^{a+1}_n)_n$ for $n$ sufficiently large we have $\|Tz^{a+1}_n - P_{(0,p_a)}Tz^{a+1}_n\| < \e/(2n_{2j_0-1})$. Set $(z_{a+1},\eta_{a+1}) = (z_n^{a+1},\eta_n^{a+1})$ and choose $p_{a+1}$ sufficiently large so that $\max\{M_n, p_a,\ra(\eta_{a+1}),\max\supp_{BD}(z_{a+1})\} < p_{a+1}$ and $\|P_{(p_{a+1},\infty)}Tz_1\| < \e/(2n_{2j_0-1})$. Set $\xi_{a+1} = (p_{a+1},\xi_a,1/m_{2j_0-1},\eta_{a+1})$ which is in $\De_{p_{a+1}}^{\mathrm{Odd}_1}$.
 
Having chosen the dependent sequence set $z = (m_{2j_0-1}/n_{2j_0-2})\sum_{k=1}^{n_{2j_0-1}}z_k$ and $\ga = \xi_{n_{2j_0-1}}$. By the analysis of $\ga$ we obtain $e_\ga^* = (1/m_{2j_0-1})\sum_{k=1}^{n_{j_0-1}}e^*_{\eta_k}\circ P_{(p_{k-1},p_k)} + \sum_{k=1}^{n_{2j_0-1}}d_{\xi_k}^*$. By Proposition \ref{zero dependent estimate} we have $\|z\| \leq 462A_0C_0C/m_{2j_0-1}$. The same calculations as the ones above yield that $e_\ga^*(Tz) \geq \de/(6m_2) - \e/{n_{2j_0-1}}- m_{2j_0-1}\e \geq \de/ \geq \de/(7m_2)$, by the choice of $\e$. This yields $\|T\| \geq \frac{\de }{3234m_2A_0C_0C}m_{2j_0-1}$. Choosing $m_{2j_0-1}$ large enough yields a conradiction. 
 \end{proof}

\begin{prp}
\label{scalar plus horizontally compact}
Let $T:\mathcal{Y}_X\to\mathcal{Y}_X$ be a bounded linear operator. Then there exists a scalar $\la$ so that $T - \la I$ is horizontally compact.
\end{prp}
 
\begin{proof}
Let $(z_k)_k$ be an arbitrary $C$-RIS, that is bounded from bellow, say by some $\e>0$. By Proposition \ref{distances vanish} we may find a scalar $\la$ and pass to a subsequence so that $\lim_k\|Tz_k - \la z_k\| = 0$.  We will show that for any bounded block sequence $(w_k)_k$ we have $\lim_k\|Tw_k - \la w_k\| = 0$. By Proposition \ref{horizontal compactness is witnessed on ris} is sufficient to verify this only for a $C'$-RIS $(w_k)_k$. If $\lim_k\|w_k\| = 0$ then this is true. Otherwise we may assume $\|w_k\| > \tilde \e$ for all $k\in\N$ and for some $\tilde \e>0$. By passing to subsequences of $(z_k)$ and $(w_k)_k$ we may assume that $\ran(z_k)$ and $\ran(w_m)$ are disjoint for all $k$ and $m$ in $\N$. By Remark \ref{ris sum to ris} the sequence $(z_k - w_k)_k$ has a $(C+C')$-RIS subsequence. Apply Lemma \ref{distances vanish} and pass to further subsequences and to find scalars $\la'$ and $\mu$ so that $\lim_k\|Tw_k - \la'w_k\| = 0$ and $\lim_k\|T(z_k - w_k) - \mu(z_k - w_k)\| = 0$. We can consequently deduce that $\lim_k\|\la z_k - \la'w_k - \mu(z_k - w_k)\| = 0$. By the fact that the ranges of $z_k$ and $w_k$ are disjoint we obtain
\begin{equation*}
\frac{1}{2A_0C_0}\max\left\{|\la - \mu|\e,|\la' - \mu|\tilde \e\right\} \leq \lim_k\left\|(\la - \mu)z_k -  (\la' - \mu)w_k\right\| = 0
\end{equation*}
which yields $\la = \la'$, as desired.
\end{proof}

 \section{Diagonal plus compact approximations}
 \label{Diagonal plus compact approximations}
 In this section we finally prove that every bounded linear operator on the the space $\mathcal{Y}_X$ can be approximated by a sequence of operators of the form diagonal (with respect to the Schauder decomposition $(Z_k)_k$) plus compact. This is the main theorem used to prove the desired property of the Calkin algebra of $\mathcal{Y}_X$.

\begin{prp}
\label{lower bound}
Let $(z_i)_i$ be a block sequence in $\mathcal{Y}_X$ (with respect to the Schauder decomposition $(Z_k)_k$). Then $(z_i)_i$ has a subsequence $(z_{i_k})_k$ with the property that for every $j\in\N$, $1\leq a\leq n_{2j}$, and every further block vectors $(w_i)_{i=1}^{a}$ of $(z_{i_k})_k$ there exists $\ga\in\Ga$ with
\begin{equation*}
\left\|\sum_{i=1}^{a}w_i\right\| \geq e_\ga^*\left(\sum_{i=1}^{a}w_i\right) \geq \frac{1}{3m_{2j}}\sum_{i=1}^{a}\|w_i\|.
\end{equation*}
\end{prp}

\begin{proof}
Let us denote by $E_i$ the support of each vector $z_i$ with respect to $(Z_k)_k$. We fix $0<\e\leq 1/3$. By remark \ref{how things are normed} we may choose for each $i_0\in\N$ an index $N_{i_0}$ so that for every $w$ in the linear span of $(z_i)_{i\leq i_0}$ satisfies
$$(1-\e)\|w\| \leq \max\{b^*(w): b^*\in A_{N_{i_0}}\}.$$
(for the definition of $A_{N_{i_0}}$ see page \pageref{even0}, paragraph before \eqref{even0}) As the sets $(A_n)_n$ are increasing we may choose the sequence $(N_i)_i$ to be strictly increasing. Pass to a subsequence $(i_k)_k$ so that $N_{i_k} + 1 < \min(E_{i_{k+1}})$ as well as $2k\leq N_{i_k}$  for all $k\in\N$.

Let now $j\in\N$, $1\leq a\leq n_{2j}$, and $(w_d)_{d=1}^a$ be block vectors of $(z_{i_k})_k$. We may assume that $a = n_{2j}$. For each $d$ let $i_{k_d}$ be the maximum of the support of the vector $w_d$ with respect to $(z_{i_k})_k$. Then there is $\tilde b_d^*\in A_{N_{i_{k_d}}}$ with $\tilde b_d^*(w_d) \geq (1-\e)\|w_d\|$. We set $b_1^* = 0$, $p_1 = N_{i_{k_j}},$ $b_2^* = \tilde b_{j+1}^*$, $p_2 = N_{i_{k_{j+1}}}$,\ldots, $b^*_{n_{2j}-j+1} = \tilde b^*_{n_{2j}}$, $p_{n_{2j}-j+1} = N_{i_{k_{n_{2j}-j+1}}}$.
Then there are $\xi_l\in\De_{p_l+1}$ for $1\leq l\leq n_{2j}-j+1$ and $\ga\in\Ga$ with evaluation analysis $(p_l,\xi_l,b_l^*)_{l=1}^{n_{2j}-j+1}$. It follows that
\begin{equation*}
\label{one estimate required in this lemma}
\left\|\sum_{d=1}^{n_{2j}}w_d\right\| \geq e_\ga^*\left(\sum_{d=1}^{n_{2j}}w_d\right) = \frac{1}{m_{2j}}\sum_{d=j+1}^{n_{2j}}\tilde b_d^*(w_d) \geq (1-\e)\frac{1}{m_{2j}}\sum_{d=j+1}^{n_{2j}}\|w_d\|.
\end{equation*}
Similarly, for $1\leq d\leq j$ there is $\ga_d\in\Ga$ with $\|\sum_{d=1}^{n_{2j}}w_d\| \geq e_{\ga_d}^*(\sum_{i=1}^{a}w_i) \geq (1-\e)/m_{2}\|w_d\|.$
Finally,
\begin{align*}
\frac{1}{m_{2j}}\sum_{i=1}^{a}\|w_i\| &\leq \frac{j}{m_{2j}}\max_{1\leq d\leq j}\|w_d\| + \frac{1}{m_{2j}}\sum_{d=j+1}^{n_{2j}}\|w_d\|\\
&\leq \frac{1}{m_2}\max_{1\leq d\leq j}\|w_d\| + \frac{1}{m_{2j}}\sum_{d=j+1}^{n_{2j}}\|w_d\|\leq \frac{2}{1-\e}\left\|\sum_{d=1}^{n_{2j}}w_d\right\|.
\end{align*}
\end{proof}

\begin{prp}
\label{you can only escape from component  compactly}
Let $T:\mathcal{Y}_X\to\mathcal{Y}_X$ be a bounded linear operator. Then for every $k\in\N$ the operator $TP_{\{k\}} - P_{\{k\}}TP_{\{k\}}$ is compact.
\end{prp}

\begin{proof}
We first observe that $P_{(0,k)}TP_{\{k\}}$ is compact. Recall $Z_k\simeq X_k\oplus\ell_\infty(\De_k)$ and $\sum_{i=1}^{k-1}\oplus Z_k\simeq (\sum_{i=1}^{k-1}\oplus X_i)\oplus\ell_\infty(\Ga_{k-1})$. We may identify $P_{(0,k)}TP_{\{k\}}$ with an operator on these spaces. As it is mentioned bellow \eqref{mjnjproperties} every operator from $X_k$ to $X_i$ with $k\neq i$ is compact. This yields that $P_{(0,k)}TP_{\{k\}}$ is compact.

We now show that $P_{(k,\infty)}TP_{\{k\}}$ is compact as well and the proof will be complete. This requires a bit more work. Omitting the finite dimensional component $\ell_\infty(\De_k)$ It is sufficient to show that every bounded linear operator $S:X_k\to\sum_{i=k+1}^\infty\oplus Z_i$ is compact. We will show that for a $C$-RIS $(x_n)_n$ in $X_k$ $\lim_n\|Sx_n\| = 0$. If this is true the \cite[Proposition 5.11]{AH} implies that $S$ is indeed compact. We start by observing that by \cite[Corollary 5.5]{AH} for any scalars $(a_n)_n$ we have $\|\sum_{n=1}^\infty a_nx_n\| \leq 10C\vertiii{\sum_{k=1}^\infty a_ne_n}$, where $\vertiii{\cdot}$ is the norm induced by $W[(\mathcal{A}_{3n_j},1/\tilde m_j)_{j\in L_k}]$. The argument used in the proof of Corollary \ref{estimate for other norming sets} yields that for any $j\in\N$ yields we have
\begin{equation}
\label{kenli}
\left\|\frac{m_{2j}}{n_{2j}}\sum_{n=1}^{n_{2j}}x_n\right\| \leq \frac{20C}{m_{2j}}.
\end{equation}
Towards a contradiction assume that $\liminf_n\|Sx_n\| > 0$. Arguing as in the first part of this proof, for every $N>k$ the operator $P_{(k,N)}S$ is compact. Recall that bounded block sequences in $X_k$ are weakly null (see \cite[Proposition 5.12]{AH}). This means that for all $N>k$ $\lim_n\|P_{(k,N)}Sx_n\| = 0$. We conclude that $(Sx_n)_n$ has a subsequence that is equivalent to a block sequence in $\mathcal{Y}_X$. By Proposition \ref{lower bound} and passing to a further subsequence there exists $\theta>0$ so that for all $j\in\N$ we have $\|\sum_{n=1}^{n_{2j}}x_n\| \geq \theta n_{2j}/m_{2j}$. Combining this with \eqref{kenli} we obtain $m_{2j} \leq 20C/\theta $ for all $j\in\N$ which is absurd.
\end{proof}

\begin{thm}
\label{diagonal plus compact are dense}
For every bounded linear operator $T:\mathcal{Y}_X\to\mathcal{Y}_X$ there exists a sequence of real numbers $(a_k)_{k=0}^\infty$ and a sequence of compact operators $(K_n)_n$ so that
$$T = \lim_n\left(a_0 I + \sum_{k=1}^na_kI_k + K_n\right),$$
where the limit is taken in the operator norm.
\end{thm}
 
\begin{proof}
Use Proposition \ref{scalar plus horizontally compact} to write $T = a_0I + S$ with $S$ horizontally compact. Recall that for each $k\in\N$ $I_k:\mathcal{Y}_X\to\mathcal{Y}_X$ is a projection whose image is isomorphic to the space $X_k$, a space that has the scalar-plus-compact-property. This means that for each $k\in\N$ the operator $I_kSI_k = a_kI_k + \tilde C_k$, where $a_k$ is a scalar and $\tilde C_k$ is compact.  Since  $I_k$ is a finite rank perturbation of $P_{\{k\}}$ we conclude that $P_{\{k\}}SP_{\{k\}} = a_kI_k + C_k$ with $C_k$ compact.  Furthermore, by Proposition \ref{you can only escape from component  compactly} for every $n\in\N$ the operator $\tilde K_n = SP_{(0,n]} - \sum_{k=1}^nP_{\{k\}}SP_{\{k\}}$ is compact. Summarizing, if we define the compact operator $K_n = \sum_{k=1}^nC_k + \tilde K_n$ then $SP_{(0,n]} = \sum_{k=1}^na_kI_k + K_n$. As $S$ is horizontally compact $SP_{(0,n]}$ converges to $S$ in operator norm. In conclusion, $\lim_n\|T - (a_0I + \sum_{k=1}^na_kI_k + K_n)\| = 0$. 
\end{proof}
 
 \begin{rmk}
 \label{ss is compact}
 Theorem \ref{diagonal plus compact are dense} easily implies that Strictly singular operators on $\mathcal{Y}_X$ are always compact.
 \end{rmk}

\section{Remarks and problems}
\label{section remarks and problems}
We conclude this paper with a section containing general remarks based on our results as well as several related open problems.

\begin{rmk}
In \cite{MPZ} for every countable compactum $K$ a Banach space $X_K$ is presented the Calkin algebra of which is isomorphic, as a Banach algebra, to $C(K)$. There exist $K_1$ and $K_2$ so that $C(K_1)$ are isomorphic as Banach spaces but not as Banach algebras. Such an example is provided by $K_1 = \omega$ and $K_2 = \omega\cdot2$. Hence, it is possible for Calkin algebras to be isomorphic to one another as Banach spaces but not as Banach algebras. There is also an additional manner of achieving this. In \cite{T} a $\mathscr{L}_\infty$-space $\mathfrak{X}_\infty$ is presented the Calkin algebra of which is isometric, as a Banach algebra, to the convolution algebra of $\ell_1(\N_0)$ (where $\N_0 = \{0\}\cup\N$). Note that this commutative Banach algebra has continuum many maximal ideals (each ideal corresponds to a number $\la\in[-1,1]$ by being the kernel of the functional $\phi_\la\in\ell_\infty(\N_0) = \ell_1(\N_0^*)$ with $\phi_\la = (1,\la,\la^2,\la^3,\ldots)$). Here, if $X$ is the space $c$ endowed with the monotone summing basis then as it follows by Example \ref{variation norm} the Calkin algebra of $\mathcal{Y}_X$ is isomorphic to $\ell_1$. However, by Corollary \ref{ideals in that case too} it has countably many maximal ideals. Therefore it is different, as a Banach algebra, to the convolution algebra of $\ell_1(\N_0)$.
\end{rmk}

\begin{rmk}
All preexisting examples of Banach spaces the Calkin algebras of which were explicitly described were $\mathscr{L}_\infty$-spaces (see \cite{T}, \cite{MPZ}, \cite{KL}). We point out that, as it was observed in \cite{Sk}, the Calkin algebra of the space $X$ from \cite{AM} is the unitization of a Banach algebra with trivial multiplication. This space is not a $\mathscr{L}_\infty$-space however the norm structure of its Calkin algebra is not explicitly described. In the present paper we use tools from the theory of $\mathscr{L}_\infty$-spaces, however the space $\mathcal{Y}_X$ is not necessarily a $\mathscr{L}_\infty$-space. For example, it can be shown using the set $\mathcal{G}^\mathrm{utc}$ that $X^*$ is crudely finitely representable in $\mathcal{Y}_X^*$. This implies that when $X^*$ does not embed in an $L_1$ space then $\mathcal{Y}_X^*$ is not a $\mathscr{L}_1$-space. Such a space is, e.g., $X = J^* = \mathcal{J}_*(\ell_2)$ the dual of which is isomorphic to $J$ which has trivial cotype. In fact, with some work, something stronger can be shown: if $\mathcal{Y}_X$ is a $\mathscr{L}_\infty$-space then so is $X$.
\end{rmk}

\begin{rmk}
The space $\mathcal{Y}_X$ has the bounded approximation property. To see this recall that $(Z_n)_n$ is a Schauder decomposition of $\mathcal{Y}_X$ and the spaces $\sum_{1=k}^n\oplus Z_k$ are uniformly isomorphic to $((\sum_{i=1}^n\oplus X_k)^X_\mathrm{utc}\oplus\ell_\infty(\Ga)_n)_\infty = (\mathcal{X}_n\oplus\ell_\infty(\Ga_n))_\infty$. Recall that each $X_k$ has a $2$-Schauder basis $(t_{k,i})_i$. Using the definition of $\mathcal{G}^\mathrm{utc}$ one can see that the finite dimensional subspaces $F_m$ spanned by $((t_{k,i})_{i=1}^m)_{k=1}^m$ are increasing, uniformly complemented in $\mathcal{X}_n$, and their union is dense $\mathcal{X}_n$.
\end{rmk}

\begin{rmk}
When a space $X$ has a finite dimensional Calkin algebra or one that is isomorphic to $\ell_1$ then $\mathcal{K}(X)$ is complemented in $\mathcal{L}(X)$. In the examples from \cite{MPZ}, that have infinite dimensional $C(K)$ spaces as Calkin algebras, the space $\mathcal{K}(X)$ is not complemented in $\mathcal{L}(X)$. Regarding the space $\mathcal{Y}_X$ this is not always clear. If $\mathcal{K}_\mathrm{diag}(X)$ is isomorphic to $\ell_1$ then $\mathcal{K}(\mathcal{Y}_X)$ is complemented (this happens, e.g., when $X$ is $c_0$ endowed with the summing basis). On the other hand, if $\mathcal{K}_\mathrm{diag}(X)$ contains $c_0$ then the argument used in \cite{MPZ} goes through and $\mathcal{K}(\mathcal{Y}_X)$ is not complemented in $\mathcal{L}(\mathcal{Y}_X)$.
\end{rmk}

\begin{rmk}
The construction of spaces with quasi-reflexive Calkin algebras is a step towards trying to find a space with a reflexive and infinite dimensional Calkin algebra. One way for this to be possible would be to find a space $X$ with $\mathcal{L}(X)$ reflexive. As it was pointed out to us by J. A. Ch\' avez-Dom\' inguez, \cite[Corollary 2]{B} implies that such an $X$ must be finite dimensional so this route would be a dead end.
\end{rmk}

\begin{problem}
Does there exist a Banach space the Calkin algebra of which is reflexive and infinite dimensional?
\end{problem}

\begin{rmk}
Given a Banach space $X$ with a basis we have used the Argyros-Haydon scheme for defining spaces with the scalar-plus-compact property to obtain a Calkin algebra that is isomorphic as a Banach algebra to the space $\R I\oplus\mathcal{K}_\mathrm{diag}(X)$. It is conceivable that one may use the Gowers-Maurey scheme from \cite{G} for constructing a space with an unconditional basis with the ``diagonal plus strictly singular'' property to construct a space with the property that algebra $\mathcal{L}(X)/\mathcal{SS}(X)$ is isomorphic as a Banach algebra to the whole space $\mathcal{L}_\mathrm{diag}(X)$. If one would like to have a space with Calkin algebra $\mathcal{L}_\mathrm{diag}(X)$, then a new scheme would be necessary, one that is used to a define Banach space with an unconditional basis with the diagonal-plus-compact property.
\end{rmk}

\begin{problem}
\label{diagonal plus compact}
Does there exist a Banach space with an unconditional basis so that every bounded linear operator on that space is the sum of a diagonal operator with a compact operator?
\end{problem}

\begin{problem}
\label{all diagonals}
Let $X$ be a Banach space with a Schauder basis $(e_i)_i$. Does there exist a Banach space $Y$ the Calkin algebra of which is isomorphic, as a Banach algebra, to $\mathcal{L}_\mathrm{diag}(X)$?
\end{problem}

\begin{rmk}
Recall that for all spaces $X$ with an unconditional basis $\mathcal{L}_\mathrm{diag}(X)$ is isomorphic as a Banach algebra to $\ell_\infty$ with pointiwise multiplication. As it was explained earlier, a positive answer to Problem \ref{diagonal plus compact} could perhaps yield a positive answer to Problem \ref{all diagonals} and hence also a positive solution to the following. 
\end{rmk}

\begin{problem}
Does there exist a Banach space the Calkin algebra of which is isomorphic, as a Banach algebra to $C(K)$ for an uncountable compact topological space $K$?
\end{problem}

\begin{rmk}
\label{cant show this way for normal pointwise}
In a personal communication with the first author T. Kania asked whether the unitization of Schreier space from \cite{S} endowed with coordinate-wise multiplication with respect to its standard unconditional basis is a Calkin algebra. Our paper does not provide an answer to this. More generally, we observe that our method does not work directly to show that the unitization of an arbitrary Banach space with an unconditional basis is a Calkin algebra. If $X$ is a Banach space with an unconditional basis $(x_i)_i$ and there exists a Banach space $Y$ with a basis so that the unitization of $X$ is isomorphic as a Banach algebra to $\R I\oplus \mathcal{K}_\mathrm{diag}(Y)$ then $(x_i)_i$ is equivalent to the unit vector basis of $c_0$. Indeed, assume that $T:\R e\oplus X\to\R I\oplus\mathcal{K}(Y)$ is an algebra isomorphism. Then if $(e_i)_i$ is the basis if $Y$ write for all $k\in\N$ $Tx_k= \mathrm{SOT}$-$\sum_ia_i^ke_i^*\otimes e_i$. Since $x_k^n =  x_k$ for all $n\in\N$ we obtain that in fact there is a subset $F_k$ of $\N$ so that $Tx_k = \mathrm{SOT}$-$\sum_{i\in F_k}e_i^*\otimes e_i$ and since $x_ix_j = 0$ for $i\neq j$ the sets $F_k$ must be pairwise disjoint. Clearly, $Te = \mathrm{SOT}$-$\sum_ie_i^*\otimes e_i$. The fact that $T$ is onto implies that each $F_k$ is a singleton $F_k = \{\phi(k)\}$ and that $\cup_k F_k = \N$. If we reorder the basis $(x_k)_k$ as $(x_{\phi^{-1}(k)})_k$ then for all $n\in\N$ $\|\sum_{k=1}^nx_{\phi^{-1}(k)}\| \leq \|T^{-1}\| \|\sum_{k=1}^ne_k^*\otimes e_k\|\leq C\|T^{-1}\|$ where $C$ is the monotone constant of $(e_i)_i$. Unconditionality yields that $(x_i)_i$ is equivalent to the unit vector basis of $c_0$.
\end{rmk}

\begin{problem}
\label{unconditional pointwise unitization}
Find a Banach space $X$ with an unconditional basis $(x_i)_i$ that is not equivalent to the unit vector basis of $c_0$ so that there exist a Banach space the Calkin algebra of which is isomorphic as a Banach algebra to the unitization of $X$.
\end{problem}

\begin{rmk}
The unitization of James space $\R e_\omega\oplus J$ may be viewed as the space of all scalar sequences of bounded quadratic variation and similarly for every Banach space $X$ with a subsymmetric sequence the space $\R e_\omega\oplus J(X)$ is the space of all scalar sequences of bounded $X$-variation. One may consider for any such $X$ and any linearly ordered set $I$ the space $V_X(I)$ of all functions $f:I\to\R$ of bounded $X$-variation. The norm on such a space is submultiplicative. The spaces $V_X[0,1]$ were introduced in \cite{AMP}. By very carefully combining the method of the present paper with the method from \cite{MPZ} it is conceivable that one may obtain for every countable well ordered set $I$ a Banach space the Calkin algebra of which is isomorphic as a Banach algebra to $V_X(I)$.
\end{rmk}

\begin{problem}
Let  $X$ be a Banach space with a subsymmetric basis. For what ordered sets $I$ does there exist a Banach space the Calkin algebra of which is isomorphic as a Banach algebra to $V_X(I)$?
\end{problem}

\begin{rmk}
In \cite{AMS} the notion of a convex block homogeneous sequence is introduced, i.e., a sequence that is equivalent to its convex block sequences. Known examples of such sequences are constructed using a space $X$ with a subsymmetric basis $(x_i)_i$. The bases $(e_i)_i$ of $J(X)$ and $(v_i)_i$ of $\mathcal{J}_*(X)$ are both convex block homogenous. In either case, the difference basis is submultiplicative .
\end{rmk}

\begin{problem}
Let $X$ be a Banach space with a convex block homogeneous basis $(e_i)_i$ and set $d_1 = e_1$ and for $i\in\N$ $d_{i+1} = e_{i+1} - e_i$. Is $X$ endowed with $(d_i)_i$ submultiplicative? 
\end{problem}
Regardless of the discussion above, there is little reason to believe that the answer to this problem should be positive. In fact, by \cite[Theorem II]{AMS}, a positive answer would imply that every for every conditional spreading sequence $(e_i)_i$ the difference basis $(d_i)_i$ is submultiplicative. 
 
\section*{Acknowledgements}
The authors would like to extend their sincere gratitude to Professor S. A. Argyros for invaluable discussions that contributed to this work.

\end{document}